\documentclass[twoside,a4paper,11pt]{amsart}

\usepackage[T1]{fontenc}
\usepackage{mathpazo} 
\usepackage{eulervm}

\usepackage[utf8]{inputenc}

\usepackage{fullpage}
\usepackage{amscd}
\usepackage{amsmath}
\usepackage{amssymb}
\usepackage{amsthm}
\usepackage{color}
\usepackage{graphicx}
\usepackage{hyperref}
\usepackage[noabbrev,capitalize]{cleveref}
\usepackage{mathrsfs}
\usepackage{mathtools}
\usepackage{pict2e}
\usepackage{stackrel}
\usepackage[T1]{fontenc}
\usepackage{tikz}
\usepackage{tikz-cd}
\usetikzlibrary{decorations.pathmorphing,decorations.markings,arrows,calc,shapes.geometric,arrows.meta,positioning}
\usepackage[utf8]{inputenc}
\usepackage{wasysym}
\usepackage[all]{xy}
\usepackage{xcolor}
\usepackage{xy}
\usepackage{epigraph}

\hypersetup{
    colorlinks,
    linkcolor=[rgb]{0.8, 0.5451, 0.3961},
    citecolor=[rgb]{0.8, 0.5451, 0.3961},
    urlcolor=[rgb]{0.8, 0.5451, 0.3961}
}

\usepackage{palatino}
\usepackage{mathpazo} 
\allowdisplaybreaks

\newtheorem{lemma}{Lemma}[section]
\newtheorem{theorem}[lemma]{Theorem}
\newtheorem{Corollary}[lemma]{Corollary}
\newtheorem{Proposition}[lemma]{Proposition}

\newtheorem{theoremintro}{Theorem}

\theoremstyle{definition}
\newtheorem{Definition}[lemma]{Definition}
\newtheorem{Remark}[lemma]{\sc Remark}
\newtheorem{Warning}[lemma]{\sc Warning}

\newtheorem{Example}[lemma]{\sc Example}
\newtheorem{Notation}[lemma]{Notation}

\def\colim{\mathop{\mathrm{colim}}}

\newcommand{\E}{\mathcal{E}}
\newcommand{\Sym}{\mathbb{S}}

\newcommand{\qi}{\xrightarrow{ \,\smash{\raisebox{-0.65ex}{\ensuremath{\scriptstyle\sim}}}\,}}
\newcommand{\lqi}{\xleftarrow{ \,\smash{\raisebox{-0.65ex}{\ensuremath{\scriptstyle\sim}}}\,}}

\newcommand{\C}{\mathcal{C}}

\newcommand{\id}{\operatorname{id}}
\newcommand{\kk}{\Bbbk}

\input cyracc.def

\author{Victor Roca i Lucio}
\title{Higher Lie theory in positive characteristic}
\date{\today}

\address{Victor Roca i Lucio, Université Paris Cité, Sorbonne Université, CNRS, IMJ-PRG, F-75013 Paris, France}
\email{\href{mailto:rocalucio@imj-prg.fr}{rocalucio@imj-prg.fr}}



\thanks{}

\begin{document}
		
\begin{abstract}
The main goal of this article is to develop integration theory for \textit{absolute} partition $\mathcal{L}_\infty$-algebras, which are point-set models for the (spectral) partition Lie algebras of Brantner--Mathew where infinite sums of operations are well-defined by definition. We construct a Quillen adjunction between absolute partition $\mathcal{L}_\infty$-algebras and simplicial sets, and show that the right adjoint is a well-behaved integration functor. Points in this simplicial set are given by solutions to a Maurer--Cartan equation, and we give explicit formulas for gauge equivalences between them. We construct the analogue of the Baker--Campbell--Hausdorff formula in this setting and show it produces an isomorphic group to the classical one over a characteristic zero field. 

\medskip

We apply these constructions to show that absolute partition $\mathcal{L}_\infty$-algebras encode the $p$-adic homotopy types of pointed connected finite nilpotent spaces, up to certain equivalences which we describe by explicit formulas. In particular, these formulas also allow us to give a combinatorial description of the homotopy groups of the $p$-completed spheres as solutions to a certain equation in a given degree, up to an equivalence relation imposed by elements one degree above. Finally, we construct absolute partition $\mathcal{L}_\infty$ models for $p$-adic mapping spaces, which combined with the description of the homotopy groups gives an algebraic description of the homotopy type of these $p$-adic mapping spaces parallel to the unstable Adams spectral sequence.
\end{abstract}

\maketitle

\setcounter{tocdepth}{1}

\tableofcontents

\section*{Introduction}
Historically, Lie algebras emerged in correspondence with groups, as the natural structure on the tangent space at the identity of a smooth manifold endowed with a compatible group structure. In this analytic case, Lie's third theorem establishes a correspondence between simply connected Lie groups and finite dimensional Lie algebras over $\mathbb{R}$. The universal way to recover a such a Lie group from its Lie algebra is via the exponential map, which converges in some range and becomes a local isomorphism. The Baker--Campbell--Hausdorff formula appears as the universal way to turn the exponential into a group morphism (as it defines a group structure) and is given by 

\[
\mathrm{BCH}(x,y) = x + y + \frac{1}{2}[x,y] + \frac{1}{12}[x,[x,y]] - \frac{1}{12}[y,[x,y]] + \cdots~,
\]
\vspace{0.1pc}

for any two elements $x,y$ in a Lie algebra $\mathfrak{g}$. This universal formula is an infinite sum which only involves iterated brackets of $x$ and $y$ with fractional coefficients. We will refer to this procedure as the \textit{integration} of the Lie algebra $\mathfrak{g}$. In the case of Lie groups and Lie algebras over $\mathbb{R}$, one relies on the topological properties of the real numbers to make the exponential and the Baker--Campbell--Hausdorff formula converge. See \cite{SerreLie} for an account of these results.

\medskip

A purely algebraic formulation of this correspondence can be stated over a general field $\kk$ of characteristic zero by imposing nilpotency conditions that make these sums finite. Via this procedure, one can get two equivalences. A group theoretic one, between the categories of finite dimensional nilpotent Lie algebras over the rationals $\mathbb{Q}$ and torsion free radicable nilpotent abstract groups. And a geometric one, between the categories of $\kk$-unipotent algebraic groups and finite dimensional nilpotent Lie algebras over $\kk$.

\medskip

\textbf{Higher Lie theory in characteristic zero.} This integration procedure, which is at the heart of classical Lie theory, admits a generalization at the homotopical level, and this opens the door for what can be called \textit{higher Lie theory.} If one considers differential graded (dg) Lie algebras, it is natural to consider them up to quasi-isomorphism. Their homotopy category (or $\infty$-category in modern terminology) is equivalent to the homotopy category of derived infinitesimal deformation problems. This statement, now a theorem by Lurie and Pridham \cite{Lurie11,Pridham10}, gives a precise meaning to the old idea that "deformation problems are encoded by Lie algebras", a principle stated by Drinfeld, Deligne and many others, and which goes back to the work of Kodaira--Spencer \cite{KodairaSpencer58} and Gerstenhaber \cite{Gerstenhaber64}. See, for instance, \cite{toenfmp, fmps} for more on this subject.

\medskip

This means that the integration of a dg Lie algebra should not be a group, but an $\infty$-groupoid, and moreover the $\infty$-groupoid of the deformation problem it encodes, where points are given by infinitesimal deformations, paths by equivalences between deformations, and so on. A first general approach to the integration of dg Lie algebras is given by the seminal work of Hinich in \cite{Hinich01}, using methods from Sullivan \cite{Sullivan77}. A refined version of the integration procedure was constructed by Getzler in \cite{Getzler09} for nilpotent $\mathcal{L}_\infty$-algebras (homotopy Lie algebras). The key point is that now, if one views a nilpotent Lie algebra as a nilpotent $\mathcal{L}_\infty$-algebra in degree $0$ and applies Getzler's functor to it, one gets a simplicial set which is \textit{isomorphic} to the classifying space of the abstract group produced by the Baker--Campbell--Hausdorff formula. 

\medskip

Another approach to integration is given by the work of Robert-Nicoud and Vallette in \cite{robertnicoud2020higher}. Using operadic methods, they gave a tractable description of Getzler's functor. These methods also allowed them to construct  \textit{higher Baker--Campbell--Hausdorff formulas}: they showed that any horn-filling problem in the integration $\infty$-groupoid is solved by explicit infinite sums involving the higher brackets, which generalise the classical Baker--Campbell--Hausdorff. Thus the integration procedure produces an $\infty$-groupoid which is not only a homotopy type, but that can be considered a \textit{group up to homotopy} in an algebraic sense. Finally, the author generalized these results to the case involving curvature in \cite{lucio2022integration} by considering curved \textit{absolute} $\mathcal{L}_\infty$-algebras, a new type of algebraic structures where all formal sums of operations are well-defined by definition and which includes nilpotent examples by default. 

\medskip

\textbf{Rational homotopy theory.} The story of dg Lie algebras is also related to rational homotopy models for spaces. Using them, Quillen constructed the first rational models for simply connected spaces in \cite{Quillen69}. While later Sullivan constructed simpler to compute rational models dg commutative algebras in \cite{Sullivan77}, these two approaches are not unrelated. The relationship between these two types of models is called Koszul duality. This duality links dg commutative algebras and dg Lie algebras, and it is also at the heart of why dg Lie algebras encode infinitesimal deformation problems. 

\medskip

It is a modern insight of \cite{Liemodels}, and \cite{robertnicoud2020higher} in the $\mathcal{L}_\infty$ case, that integration theory can help simplify Quillen's original approach and give a direct construction of Lie models for spaces. This was also generalized to curved absolute $\mathcal{L}_\infty$-algebras in \cite{lucio2022integration}, where they were shown to provide rational models for non-necessarily pointed nor connected finite type nilpotent spaces.

\medskip

\textbf{Problems in positive characteristic.} Over a field $\kk$ of positive characteristic, the correspondences between groups and Lie algebras largely break down. Neither the exponential nor the Baker--Campbell--Hausdorff formula are fully defined, as they involve fractional coefficients with factorial denominators. 

\medskip

The tangent space at the identity of an algebraic group still has a Lie structure, and even an extra operation which makes it a $p$-\textit{restricted Lie algebra}, see \cite[Section I.7]{Jantzen}. However, there is no way in general to recover a group only from the data of its tangent $p$-restricted Lie algebra. Furthermore, two different algebraic groups, even two $\kk$-unipotent algebraic groups, can have the same $p$-restricted Lie algebra, as for example the additive group $\mathbb{G}_a$ and its Frobenius kernel $\alpha_p$. And many aspects of the classical theory of Lie algebras break down in positive characteristic, see \cite{Straderestricted}. 

\medskip

There is a partial group theoretic correspondence achieved by Lazard in \cite{Lazardnilpotent}. It establishes an equivalence of categories between the category of nilpotent $p$-restricted Lie algebras with order $p^n$ and nilpotency class $c$ and the category of finite abstract $p$-groups with order $p^n$ and nilpotency class $c$, if $c < p$. The key ingredient is to use the fact that the prime divisors of the denominator of the weight $n$ coefficients in the Baker--Campbell--Hausdorff formula only involve primes $\leq n$. Nevertheless, as Lazard points out in his introduction: "On the other hand, although the study of the relationship between groups and Lie algebras is still very incomplete, we can nevertheless estimate that Lie algebras will prove insufficient, even for the study of $p$-finite groups. We should therefore investigate whether other algebraic structures could be used to construct new categories of groups."\footnote{"D'autre part, si l'étude des relations entre groupes et algèbres de Lie est encore très incomplète, on peut estimer néanmoins que les algèbres de Lie se révéleront insuffisantes, même pour l'étude des $p$-groupes finis. Il conviendrait donc de rechercher si d'autres structures algébriques pourraient permettre la construction de nouvelles catégories de groupes." —Page 105, line 13 of \cite{Lazardnilpotent}.}

\medskip

One of the main ideas of this paper is to try to fix these problems by replacing the notion of Lie algebras with homotopically meaningful analogues in positive characteristic. Thus, instead of trying to generalize Lie theory to a positive characteristic setting, we focus on generalizing higher Lie theory, with the goal that a better behaved version of classical Lie theory might also emerge in this framework. 

\medskip

\textbf{Deformation theory and partition Lie algebras.} While classical Lie theory breaks down over a positive characteristic field, the notion of derived infinitesimal deformations does not. And it still makes sense to ask whether their homotopy theory is equivalent to the homotopy theory of some algebraic structure. This question was first settled by Brantner and Mathew in \cite{brantnermathew}: they proved that the $\infty$-category of such deformation problems is equivalent to the $\infty$-category of algebras over a monad, which they called \textit{partition Lie algebras}. Let us mention that in positive characteristic, there are two possible notions of derived infinitesimal deformations, either defined in terms of simplicial commutative algebras or in terms of $\mathbb{E}_\infty$-algebras, which are commutative up to homotopy algebras, and each is equivalent to a different type of partition Lie algebras. We will only be working with $\mathbb{E}_\infty$-infinitesimal deformation problems and their equivalent (spectral) partition Lie algebras in this paper —the adjective spectral will mostly be omitted from now on. 

\medskip

Since we want concrete algebraic objects to work with, we will consider point-set objets which, considered up to quasi-isomorphisms, give back the $\infty$-category of partition Lie algebras. Brantner, Campos and Nuiten showed that algebras over a class of operads do provide us with these pointed-set models in \cite{pdalgebras}. We make a specific choice for these point-set models for two main reasons: this choice allows us to directly apply the results of \cite{bricevictor}, the second is more involved and explained in Subsection \ref{subsection: comparison results}. Let us give an explicit description of these point-set models, which we call \textit{partition }$\mathcal{L}_\infty$\textit{-algebras}. We consider a differential graded $\kk$-vector space $\mathfrak{h}$ endowed with a family of operations

\[
\left\{ l_n^{(\sigma_0,\cdots,\sigma_r)}: \mathfrak{g}^{\otimes n} \longrightarrow \mathfrak{g} \right\}~,
\]
\vspace{0.1pc}

where $l_n^{(\sigma_0,\cdots,\sigma_r)}$ is of degree $-r-1$, for all $n \geq 2$, and all $r+1$-tuples of permutations in $\mathbb{S}_n$, where $r \geq 0$ and $\sigma_i \neq \sigma_{i+1}$ for $0 \leq i \leq r-1$. The relations satisfied by these operations are described in Appendix \ref{Appendice formules}. One should understand these operations and these relations as witnessing the fact that $\mathfrak{h}$ is endowed with a bracket $l_2^{\mathrm{id}}(-,-)$ which satisfies the Jacobi relation and is (anti)-symmetric \textit{only up to higher homotopies}, see Remark \ref{Rmk: explaining the structure} for more details. 

\medskip

\textbf{Main results.} From now on, we work over a field $\kk$ without any further assumption. The first main goal of this paper is to develop the integration theory of \textit{absolute} partition $\mathcal{L}_\infty$-algebras, which are, roughly speaking, partition $\mathcal{L}_\infty$-algebras in which all formal power series of the structural operations explained before have a well-defined image by definition. This algebraic framework is convenient as integration theory involves many different infinite sums of operations (Maurer--Cartan equations, Baker--Campbell--Hausdorff type formulas, etc). Otherwise, one needs to make sense of these infinite sums by imposing complete filtrations or nilpotency conditions. Although we will not emphasize this in this introduction, we will also allow \textit{curvature}, as it is necessary in some arguments to work in this more general framework. In order to do so, we build upon the ideas introduced in \cite{lucio2022integration} and heavily rely on the positive characteristic homotopical operadic calculus developed by Le Grignou and the author in \cite{bricevictor}. In particular, these absolute algebras appear because they are the Koszul dual notion of $\mathbb{E}_\infty$-coalgebras.

\medskip

Using these methods, we construct an adjunction 
\[
\begin{tikzcd}[column sep=5pc,row sep=3pc]
          \mathsf{sSet}_* \arrow[r, shift left=1.1ex, "\mathcal{L}_*"{name=A}]
          &\mathsf{abs}~\mathcal{L}_\infty^\pi\text{-}\mathsf{alg}^{\mathsf{qp}\text{-}\mathsf{comp}} \arrow[l, shift left=.75ex, "\mathcal{R}_*"{name=B}]
           \arrow[phantom, from=A, to=B, , "\dashv" rotate=-90]
\end{tikzcd}
\]

between the category of pointed simplicial sets and of absolute partition $\mathcal{L}_\infty$-algebras. The key ingredient is to first consider the explicit $\mathbb{E}_\infty$-coalgebra structure on the cellular chains functor $C_*^c(-)$ constructed by Berger and Fresse in \cite{BergerFresse04}. And then to push-forward this functor along the complete bar-cobar adjunction that links coalgebras over a suitable model for the $\mathbb{E}_\infty$ operad and absolute partition $\mathcal{L}_\infty$-algebras. On \textit{trivial} absolute partition $\mathcal{L}_\infty$-algebras (otherwise known as chain complexes), this adjunction coincides with the Dold--Kan correspondence.

\medskip

In this adjunction, the functor $\mathcal{R}_*$ is \textit{the integration functor} we were looking for. We use the results of \cite{bricevictor} to transfer the model structure on $\mathbb{E}_\infty$-coalgebras where weak equivalences are given by quasi-isomorphisms to qp-complete absolute partition $\mathcal{L}_\infty$-algebras along the complete bar-cobar adjunction, which then becomes a Quillen equivalence. Asking for qp-completeness is a small technical assumption that amounts to requiring that the topology induced by a canonical filtration is separated. For this model structure, the adjunction $\mathcal{L}_* \dashv \mathcal{R}_*$ becomes a Quillen adjunction, where we endow pointed simplicial sets with the Kan--Quillen model structure. The following result which roughly states that $\mathcal{R}_*$ is a well-behaved integration functor directly follows. 

\begin{theoremintro}[Theorem \ref{thm: propriétés de l'intégration}]\label{thm A intro}\leavevmode
\begin{enumerate}
\item For any qp-complete absolute partition $\mathcal{L}_\infty$-algebra $\mathfrak{g}$, the simplicial set $\mathcal{R}_*(\mathfrak{g})$ is a Kan complex.

\medskip

\item Let $f: \mathfrak{g} \twoheadrightarrow \mathfrak{h}$ be a degree-wise surjection of qp-complete absolute partition $\mathcal{L}_\infty$-algebras. Then 
\[
\mathcal{R}_*(f): \mathcal{R}_*(\mathfrak{g}) \twoheadrightarrow \mathcal{R}_*(\mathfrak{h})
\]
is a fibration of simplicial sets. 

\medskip

\item The functor $\mathcal{R}_*$ preserves weak equivalences. In particular, it sends any filtered quasi-isomorphism $f: \mathfrak{g} \qi \mathfrak{h}$ of qp-complete absolute partition $\mathcal{L}_\infty$-algebras to a weak-homotopy equivalence of simplicial sets.
\end{enumerate}
\end{theoremintro}

Let $\mathfrak{g}$ be a qp-complete absolute partition $\mathcal{L}_\infty$-algebra and let us omit certain subtleties concerning infinite sums in this introduction for simplicity's sake. Leveraging the explicit nature of the formulas in \cite{BergerFresse04}, we can give a combinatorial description of the simplicial set $\mathcal{R}_*(\mathfrak{g})$. The $0$-simplices of $\mathcal{R}(\mathfrak{g})$ are given by \textit{Maurer--Cartan elements}, which are defined as elements $\alpha$ in degree $0$ which satisfy the following equation
\[
\sum_{n \geq 2} l_n^{\mathrm{id}}(\alpha, \cdots, \alpha) + d_\mathfrak{g}(\alpha) = 0~.
\]
We also completely characterize $1$-simplices in $\mathcal{R}(\mathfrak{g})$ by explicit formulas, which correspond to \textit{gauge equivalences} between Maurer--Cartan elements. This allows us to describe $\pi_0(\mathcal{R}(\mathfrak{g}))$, that is, the set of Maurer--Cartan elements up to gauge equivalence; when $\mathfrak{g}$ encodes a derived deformation problem, this set is the set of deformations up to equivalences and gives back the underlying classical deformation problem. In this direction, we show that under a certain \textit{homotopy completeness} condition on a partition $\mathcal{L}_\infty$-algebra, we can present at the point-set level the derived deformation problem it encodes via our integration functor, see Proposition \ref{Prop: gives back the fmp}. 

\medskip

We also describe the higher homotopy groups $\pi_k(\mathcal{R}(\mathfrak{g}),0)$ at the Maurer--Cartan element $0$ by explicit formulas. An element $\varepsilon$ of $\mathfrak{g}$ in degree $k \geq 1$ is a \textit{representative element} if it satisfies the following equation: 

\[
\displaystyle d_\mathfrak{g}(\varepsilon) + \sum_{n \geq 2} \sum_{\substack{w \in \E(n)_{k(n-1)}                             \\ \overline{w}_1 = \mathrm{id}_{\Sym_n}, \, \overline{w}_j \in \Sym_n}} (-1)^{\left[\frac{k(k-1)}{2}\right] \frac{(n+2)(n-1)}{2}} \prod_{j = 2}^k \mathrm{sign}(\overline{w}_j) ~ l^w_n(\varepsilon,\cdots,\varepsilon)= 0~, 
\]
where the sum runs over all $w = (\sigma_0,\cdots, \sigma_{k.(n-1)})$ in $\mathcal{E}(n)_{k.(n-1)}$ such that every $n$-tuple $\overline{w}_j \coloneqq (\sigma_{(j-1)(n-1)}(1),\cdots, \sigma_{j(n-1)}(1))$ is a permutation in $\mathbb{S}_n$ for all $1 \leq j \leq k$ and in particular $\overline{w}_1$ is the identity permutation of $\mathbb{S}_n$. Here $\mathrm{sign}(\overline{w}_j)$ stands for the signature of the permutation $\overline{w}_j$. Two representative elements $\varepsilon_1$ and $\varepsilon_2$ are \textit{interval equivalent} if there exists an element $\varphi$ in $\mathfrak{g}$ of degree $k+1$ which links the two elements via explicit combinatorial formulas, see Definition \ref{def: interval equivalences} for more details. 

\begin{theoremintro}[Theorem \ref{thm: fake Berglund theorem}]\label{thm B intro}
Let $k \geq 1$. There is a canonical natural bijection

\[
\pi_k(\mathcal{R}_*(\mathfrak{g}),0) \cong \mathrm{rep}(\mathfrak{g}_k)/\sim_{\mathrm{int}}~,
\]
\vspace{0.1pc}

between the $k$-th homotopy group of $\mathcal{R}_*(\mathfrak{g})$ at the Maurer--Cartan $0$ and the set of representative elements of degree $k$ in $\mathfrak{g}$ up to interval equivalence.  
\end{theoremintro}

In essence, these formulas come from the explicit $\mathbb{E}_\infty$-coalgebra structure on the reduced cellular chains on the $k$-sphere $\tilde{C}_*^c(S^k)$. Similar formulas exist for the homotopy groups at the other base points, determined by the unreduced cellular chains on the $k$-sphere. It is interesting to point out that when $\kk$ is a field of characteristic zero, this structure is homotopically trivial; thus a representative element is just a cycle and interval equivalences reduce to common boundaries. This coincides with Berglund's theorem in \cite{Berglund15} which states that the homotopy groups of the integration functor are given by the homology groups of the algebra. 

\medskip

By Theorem \ref{thm A intro}, we know that $\mathcal{R}_*(\mathfrak{g})$ is a Kan complex, thus a model for an $\infty$-groupoid. Inspired by the work of \cite{robertnicoud2020higher}, we show that $\mathcal{R}(\mathfrak{g})$ is in fact an \textit{algebraic Kan complex}, where all horn-fillers are characterized and given by explicit combinatorial formulas, see Proposition \ref{prop: general formula for higher bch}. In general, the remaining obstacle to fully compute these formulas is a full computation of the $\mathbb{E}_\infty$-coalgebra structure of the cellular chains on the $n$-simplex $C_*^c(\Delta^n)$. We carry out this computation in full detail for the case $n=2$, which is the case that corresponds to the classical Baker--Campbell--Hausdorff formula in  \cite[Section 5.3]{robertnicoud2020higher}, since it gives a formula for the composition of paths in $\pi_1(\mathcal{R}(\mathfrak{g}),0)$, thus determines its group structure. 

\medskip

Let $x$ and $y$ be two representative elements in $\mathfrak{g}_1$. The \textit{horn-filler product} $\mathrm{HF}_0(x,y)$ is given by 

\[
\mathrm{HF}_0(x,y) = \sum_{m \geq 0} \sum_{\tau \in \mathrm{SCPT}^{\mathrm{left}}_m} \sum_{L^{01,12}(\tau)} \alpha^{(\tau;~ 01, \cdots, 12)} \gamma_\mathfrak{g}\left(\tau\left(x,\cdots, y, \cdots, x, \cdots, y; x+y \right)\right) ~, 
\]
\vspace{0.1pc}

where the sum is taken over all "left-handed symmetric corked planar trees" with leaves labelled by either $x$, $y$ or $x+y$, according to the specific labelling of the tree. These trees have vertices labelled by tuples of permutations and one sums over the composition along these trees of the associated structure operations of $\mathfrak{g}$, applied to the labels of the leaves. See Theorem \ref{thm: horn-filler formula (BCH)} for more details on what left-handed symmetric corked planar trees are and how these coefficients are defined. As an illustration, the first terms of this formula are given by 
\[ 
\begin{aligned}
\mathrm{HF}_0(x,y) =&~x + y - l_2^{(12)}(x,y) - l_3^{((123),(213))}(x,x,y) - l_3^{((132),(213))}(x,x,y) - l_3^{((123),(231))}(x,x,y)-\\
&-l_3^{((132),(231))}(x,x,y) - l_3^{((123),(231))}(y,x,y) - l_3^{((132),(231))}(y,x,y) +\\
&+ l_3^{((123),(132))}(x,y,y) + l_3^{((123),(312))}(x,y,y) + l_3^{((123),(321))}(x,y,y)+\\
&+l_3^{((123),(231))}\left(l_2^{(12)}(x,y),x,y\right) + l_3^{((132),(231))}\left(l_2^{(12)}(x,y),x,y\right) + \cdots \\
\end{aligned}
\]
In particular, all the coefficients in this sum are $\pm 1$, hence it is well-defined in any characteristic. 

\medskip

Any absolute partition $\mathcal{L}_\infty$-algebra can be restricted to an absolute $\mathcal{L}_\infty$-algebra, see Subsection \ref{subsection: comparison results}. An interesting case where this formula can be applied is given by nilpotent partition $\mathcal{L}_\infty$-algebras concentrated in degree $1$, since their restriction gives a nilpotent Lie algebra seen as an absolute $\mathcal{L}_\infty$-algebra in degree $1$ (we are working with \textit{shifted structures}, otherwise it would sit in degree $0$). For examples, see Examples \ref{Example: 2-nilpotent associative algebras}, \ref{example: 3-nilpotent monodic Lie algebra} and \ref{Example: comparison for 2-nilpotent associative algebras}. 

\begin{theoremintro}[Theorems \ref{thm: integration gives a nilpotent group} and \ref{thm: gives back bch in characteristic zero}]\label{thm C intro}\leavevmode 
Let $\mathfrak{g}$ be a nilpotent partition $\mathcal{L}_\infty$-algebra concentrated in degree $1$.

\medskip

\begin{enumerate}
\item The simplicial set $\mathcal{R}_*(\mathfrak{g})$ is isomorphic to the classifying space of the group 
\[
\big(\mathfrak{g},\mathrm{HF}_0(-,-),0\big)~,
\]
which is a nilpotent group. 

\medskip

\item If $\kk$ is a field of characteristic zero, there is an isomorphism of groups 
\[
\big(\mathfrak{g},\mathrm{HF}_0(-,-),0\big) \cong \big(\mathfrak{g},\mathrm{BCH}(-,-),0\big)~,
\]
between the nilpotent group obtained with the horn-filler product and the exponential group obtained from the underlying nilpotent Lie algebra of $\mathfrak{g}$ using the Baker--Campbell--Hausdorff formula.  
\end{enumerate}
\end{theoremintro}

However, let us mention that unlike in characteristic zero, an absolute partition $\mathcal{L}_\infty$-algebra does not need to be concentrated in degree $1$ to produce a classifying space, as can be seen from Theorem \ref{thm B intro}. The horn-filler product, applied to two representative elements still produces a representative element which, up to interval equivalence, corresponds to their product in the first homotopy group. There is a more general comparison statement given by Proposition \ref{prop: comparison characteristic zero}, which says that under the appropriate induction/restriction functors, the integration functor we have defined is naturally weakly equivalent to the one defined in \cite{lucio2022integration}, and thus to \cite{Getzler09} and \cite{robertnicoud2020higher} under the appropriate hypothesis. In particular, one can apply the induction functor to any nilpotent Lie algebra in degree $1$ and then apply the integration functor constructed here: it gives a space weakly equivalent to the classifying space of its exponential group, where the product on the first homotopy group is determined by the formula for the horn-filler product. Thus, in characteristic zero, equivalences mentioned before between nilpotent Lie algebras and types of groups could be "translated" to equivalences with certain nilpotent partition $\mathcal{L}_\infty$-algebras, perhaps no longer restricted in degree $1$. Understanding if similar statements hold in positive characteristic shall be the subject of future research. 

\medskip

\textbf{Applications to $p$-adic homotopy theory.} The second main goal of this paper is to apply the constructions performed so far to the study of $p$-adic homotopy types, like \cite{Liemodels, robertnicoud2020higher, lucio2022integration} in the characteristic zero case, using the left adjoint functor $\mathcal{L}_*$. For that, we fix $\kk$ to be an algebraically closed field of characteristic $p$ and we use Koszul duality to dualize Mandell's results in \cite{Mandell}.

\begin{theoremintro}[Theorem \ref{thm: modèles d'homotopie rationnel type fini}]\label{thm D intro}
Let $X$ be a pointed connected finite type nilpotent simplicial set. The unit of adjunction
\[
\eta_X: X \qi \mathcal{R}_*\mathcal{L}_*(X) 
\]

is an $\mathbb{F}_p$-equivalence. 
\end{theoremintro}

The relationship between these models and Mandell's work then becomes akin to the relationship that links Quillen's models using Lie algebras and Sullivan's models using commutative algebras in rational homotopy theory. However, it should be noted that we first consider absolute partition $\mathcal{L}_\infty$-algebras up to transferred weak equivalences from $\mathbb{E}_\infty$-coalgebras and not up to quasi-isomorphisms, and thus that with these weak equivalences, they present the same underlying $\infty$-category as their Koszul dual $\mathbb{E}_\infty$-coalgebras. This is consistent with Lurie's result in \cite{Lurie11Rational} which states that cochains on spaces are formally étale over a positive characteristic field: indeed, up to comparison results, this entails that the models considered here are in fact acyclic. 

\medskip

Nevertheless, we can leverage Theorem \ref{thm B intro} to define another notion of weak equivalence on absolute partition $\mathcal{L}_\infty$-algebras, which is this time transferred from pointed simplicial sets via the adjunction $\mathcal{L}_* \dashv \mathcal{R}_*$. We show that absolute partition $\mathcal{L}_\infty$-algebras admit a model structure with these equivalences in Theorem \ref{thm: local model structure with pi-equivalences} and that the resulting $\infty$-category is a coreflective $\infty$-subcategory of that of $\mathbb{E}_\infty$-coalgebras. Since these equivalences coincide with quasi-isomorphisms in degrees $\geq 1$ when $\kk$ is a field of characteristic zero, the question remains on whether the $\infty$-category described by this model structure is the $\infty$-category of algebras over some $\infty$-categorical monad in general.

\medskip

\textbf{A new combinatorial description of the homotopy groups of the $p$-completed spheres.} Let $S^m$ be the $m$-sphere and let $(S^m)_{\mathbb{F}_p}$ denote its $p$-completion. Then it follows from Theorem \ref{thm D intro} that, for every $k \geq 1$, there is a an isomorphism of groups

\[
\pi_k((S^m)_{\mathbb{F}_p},*) \cong \pi_k(\mathcal{R}_*\mathcal{L}_*(S^m),0)~, 
\]
\vspace{0.1pc}

and it follows from Theorem \ref{thm B intro} that these latter homotopy groups admit a description in terms of representative elements of degree $k$ in $\mathcal{L}_*(S^m)$ up to interval equivalences. See Theorem \ref{thm: homotopy groups of X} for the general statement. In this case, the absolute partition $\mathcal{L}_\infty$-algebra $\mathcal{L}_*(S^m)$ can be fully determined: as a graded vector space, it admits a basis given by symmetric rooted trees with leaves labelled by the single generator of $\tilde{C}_*^c(S^m)$ by Lemma \ref{Lemma: basis elements}, and all the terms of the differential can be computed by Lemma \ref{lemma: coalgebra structure on the reduced spheres}. The structural operations act by grafting the corresponding trees. So computing representative elements and determining when they are interval equivalent reduces to a purely combinatorial problem, albeit not an easy one. Therefore this gives a new way to try to compute these homotopy groups, which are at the heart of algebraic topology. Furthermore, it should be possible to plug these equations into a computer, using for instance the Python computer package developed in \cite{pythonbarratteccles}. Finding representative elements by hand is not obvious, and thus pushing the computational aspects of these results is beyond the scope of the present paper. Aside from this computational aspects, let us mention that this results embeds these homotopy groups into an ambient algebraic structure, and therefore also opens the question of whether the ambient structure descends or not to these homotopy groups.

\medskip

\textbf{Mapping spaces.} Finally, using the theory of mapping coalgebras developed by Le Grignou in \cite{grignou2022mappingII, grignou2022mapping}, we construct models for $p$-adic mapping spaces. Note that we work in the more general curved setting to get unpointed mapping spaces.

\begin{theoremintro}[Theorem \ref{thm: vrai thm mapping spaces}]\label{thm E intro}
Let $\mathfrak{g}$ be a qp-complete curved absolute partition $\mathcal{L}_\infty$-algebra and let $X$ be a simplicial set. There is a weak equivalence of Kan complexes

\[
\mathrm{Map}(X, \mathcal{R}(\mathfrak{g})) \simeq \mathcal{R}\left(\mathrm{hom}(C^c_*(X),\mathfrak{g})\right)~,
\]
\vspace{0.25pc}

which is natural in $X$ and in $\mathfrak{g}$, where $\mathrm{hom}(C^c_*(X),\mathfrak{g})$ denotes the convolution curved absolute partition $\mathcal{L}_\infty$-algebra. Furthermore, it is possible to replace the cellular chains $C^c_*(X)$ by the homology $\mathrm{H}_*(X)$ to obtain a smaller model, meaning that there is a weak equivalence of Kan complexes
\[
\mathrm{Map}(X, \mathcal{R}(\mathfrak{g})) \simeq \mathcal{R}\left(\mathrm{hom}(\mathrm{H}_*(X),\mathfrak{g})\right)~,
\]

which is now only natural in $\mathfrak{g}$. 
\end{theoremintro}

So when one takes $\mathfrak{g}$ to be a model for the $p$-completion of a space $Y$, Theorem \ref{thm E intro} gives a model for the mapping space of $X$ into the $p$-completion of $Y$, without any hypotheses on the source. To the best of our knowledge, this is the first "algebraic" model for such mapping spaces, which were extensively studied in relationship with the Sullivan conjecture in \cite{Miller,Carlsson,LannesSchwartz,LannesT}. These authors usually work at the level of power operations, that is, they compute the homotopy type of such mapping spaces using the unstable coalgebra structure on the homology of spaces, either by spectral sequences arguments (the unstable Adams' spectral sequence of Bousfield and Kan in \cite{BousfieldKanUnstable}) or by algebraic construction at that level (Lannes' $T$-functor). The model constructed in Theorem \ref{thm E intro} can be thought of as a lift from the power operations level to the "algebraic" level of some of these constructions. In particular, applying Theorem \ref{thm B intro} to the convolution algebra $\mathrm{hom}(\mathrm{H}_*(X),\mathfrak{g})$ gives a combinatorial presentation of the information in the unstable Adams' spectral sequence. 

\medskip

\subsection*{Acknowledgments}
It is my pleasure to thank Lukas Brantner, Ricardo Campos, Mario Fuentes Rumí, Najib Idrissi, Brice Le Grignou, Joost Nuiten, Daniel Robert-Nicoud, Jérôme Scherer, Bruno Vallette and Felix Wierstra for stimulating conversations about these and related topics. I would also like to thank Jérôme Scherer for useful comments on a draft version.

\subsection*{Conventions} 
Let $\kk$ be a field. In this paper, we will work with two different base categories. The standard base category of differential graded (dg) $\kk$-modules, and the lager base category of \textit{pre-differential graded (pdg)} $\kk$-modules. A pre-differential graded $\kk$-module is the data of a graded $\kk$-module together with a degree $-1$ endomorphism. We work with the \textit{homological convention} in both cases. Differential graded $\kk$-modules are a full subcategory of pre-differential graded $\kk$-modules. Both categories form symmetric monoidal categories endowed with the tensor product of graded $\kk$-modules together with the Koszul sign rule, where the unit is given by $\kk$ concentrated in degree $0$. We will omit the prefix $\kk$ when referring to $\kk$-modules from now on. 

\medskip

Let $\mathcal{C}$ be a category and let $\mathrm{W}$ be a class of arrows in $\mathcal{C}$. We will denote $\mathcal{C}~[\mathrm{W}^{-1}]$ the $\infty$-category obtained by localizing $\mathcal{C}$ at $\mathrm{W}$. When working at the $\infty$-categorical level, limits and colimits should be understood as meaning homotopy limits and colimits. Given a Quillen adjunction between model categories, we will add the prefixes $\mathbb{L}$ or $\mathbb{R}$ for the left (resp. right) derived functors.

\section{Absolute partition $\mathcal{L}_\infty$-algebras}
The goal of this section is to introduce absolute partition $\mathcal{L}_\infty$-algebras. They are the \textit{absolute} analogues of partition $\mathcal{L}_\infty$-algebras, which are explicit point-set models for the (spectral) partition Lie algebras introduced in \cite{brantnermathew}. Roughly speaking, absolute types of algebras are types of algebraic structures where \textit{infinite sums} of structural operations have a well-defined image \textit{by definition}. For an introduction to this type of algebraic structures, we refer to \cite{absolutealgebras}. 

\subsection{Explicit models for partition Lie algebras}
Formal moduli problems, defined over Artinian $\mathbb{E}_\infty$-algebras, encode infinitesimal deformations in the context of spectral algebraic geometry, see \cite[Part IV]{sag}. Brantner--Mathew showed in \cite{brantnermathew} that the $\infty$-category of such formal moduli problems is equivalent to the $\infty$-category of algebras over an $\infty$-categorical monad; these algebras are called \textit{(spectral) partition Lie algebras}. Later, Brantner--Campos--Nuiten proved in \cite{pdalgebras} that this $\infty$-category admits presentations by point-set models localized at weak equivalences. The goal of this subsection is to introduce a particular choice of point-set models; this choice is motivated by the homotopical operadic calculus developed in \cite{bricevictor} and also by considerations explained in Subsection \ref{subsection: comparison results}.

\medskip

Let us denote by $\mathcal{E}$ the the dg Barratt--Eccles operad of \cite{BergerFresse04} and by $\mathcal{E}^{\mathrm{nu}}$ its reduced version (meaning $\mathcal{E}^{\mathrm{nu}}(0) = 0$). The operadic bar construction $\mathrm{B}\mathcal{E}^{\mathrm{nu}}$ of $\mathcal{E}^{\mathrm{nu}}$ forms a conilpotent dg cooperad. Its linear dual $(\mathrm{B}\mathcal{E}^{\mathrm{nu}})^*$ is a dg operad, which is isomorphic to $\Omega (\mathcal{E}^{\mathrm{nu}})^*$, the cobar construction on the linear dual cooperad of $\mathcal{E}^{\mathrm{nu}}$. 

\begin{Definition}[Partition $\mathcal{L}_\infty$-algebra]
A \textit{partition} $\mathcal{L}_\infty$\textit{-algebra} $(\mathfrak{h},\gamma_\mathfrak{h},d_\mathfrak{h})$ amounts to the data of a dg module $(\mathfrak{h},d_\mathfrak{h})$ together with a dg $\Omega (\mathcal{E}^{\mathrm{nu}})^*$-algebra structure
\[
\gamma_\mathfrak{h}: \bigoplus_{n \geq 1} \Omega (\mathcal{E}^{\mathrm{nu}})^*(n) \otimes_{\mathbb{S}_n} ~ \mathfrak{h}^{\otimes n} \longrightarrow \mathfrak{h}~.
\]
\end{Definition}

\begin{Notation}
If there is no ambiguity, we will often drop the adjective \textit{spectral} when referring to $\infty$-categorical partition Lie algebras or to their point-set models, as we will not deal with the simplicial version in this paper. We will also refer to partition $\mathcal{L}_\infty$-algebras as $\mathcal{L}_\infty^\pi$-algebras sometimes.
\end{Notation}

The data of the structural morphism $\gamma_\mathfrak{h}$ is equivalent to the data of a family of operations

\[
\left\{ l_n^{(\sigma_0,\cdots,\sigma_r)}: \mathfrak{h}^{\otimes n} \longrightarrow \mathfrak{h} \right\}~,
\]
\vspace{0.1pc}

where $l_n^{(\sigma_0,\cdots,\sigma_r)}$ is of degree $-r-1$, for all $n \geq 2$, and all $r+1$-tuples of permutations in $\mathbb{S}_n$, where $r \geq 0$ and $\sigma_i \neq \sigma_{i+1}$ for $0 \leq i \leq r-1$. These operations are subject to relations, imposed by the differential of the operad $\Omega (\mathcal{E}^{\mathrm{nu}})^*$. This differential is partially determined by the cooperad structure $(\mathcal{E}^{\mathrm{nu}})^*$, which is quite hard to compute; we refer to Appendix \ref{Appendice formules} for more details.

\begin{Remark}\label{Rmk: explaining the structure}
Notice that a \textit{shifted} $\mathcal{L}_\infty$-algebra is the data of a family of \textit{symmetric} operations
\[
\left\{ l_n: \mathfrak{h}^{\odot n} \longrightarrow \mathfrak{h} \right\}
\]
of degree $-1$, for all $n \geq 2$, which satisfy some compatibility conditions. In a partition $\mathcal{L}_\infty$-algebra, the structural operations are no longer symmetric. Operations labelled by $r+1$-tuples of distinct permutations in $\mathbb{S}_n$ can be interpreted as higher coherences for this lack of symmetry of the operations $l_n^\mathrm{id}$, labelled by the identity permutations. Notice, however, that these homotopies compute \textit{derived invariants} and not derived coinvariants.
\end{Remark}

\begin{Proposition}[{\cite[Proposition 4.34]{pdalgebras}}]\label{Prop: semi-model on partition L infinity}
The category of partition $\mathcal{L}_\infty$-algebras admits a cofibrantly generated semi-model structure determined by the following classes of morphisms

\begin{enumerate}
\item weak equivalences are given by quasi-isomorphisms,

\item fibrations are given by degree-wise epimorphisms,

\item cofibrations are determined by the left-lifting property.
\end{enumerate} 
\end{Proposition}

\begin{proof}
Follows from the fact that the underlying dg $\mathbb{S}$-module of $\Omega (\mathcal{E}^{\mathrm{nu}})^*$ is cofibrant in the tame model structure of dg $\mathbb{S}$-modules. We refer to \cite[Chapter 4]{pdalgebras} for more details.
\end{proof}

The category of $\mathcal{L}^\pi_\infty$-algebras, localized at quasi-isomorphisms, presents the $\infty$-category of partition Lie algebras in the sense of \cite{brantnermathew}. 

\begin{theorem}[{\cite[Theorem 4.47]{pdalgebras}}]
There is an equivalence of $\infty$-categories 

\[
\mathcal{L}_\infty^\pi\text{-}\mathsf{alg}~[\mathrm{Q.iso}^{-1}] \qi \mathsf{Alg}_{\mathrm{Lie}^\pi}(\mathsf{Mod}_\kk)~,
\]
\vspace{0.1pc}

between the $\infty$-category of partition $\mathcal{L}_\infty$-algebras localized at quasi-isomorphisms and the $\infty$-category of (spectral) partition Lie algebras in the $\infty$-category of dg modules. 
\end{theorem}

\begin{Remark}\label{Rmk: change of models for partition Lie algebras}
Let $\mathcal{O}$ be an $\mathbb{E}_\infty$-operad, that is, any $\mathbb{S}$-projective resolution of the operad $\mathcal{C}om$. Then dg $(\mathrm{B}\mathcal{O})^*$-algebras admit a semi-model structure, which localized at quasi-isomorphisms also presents the $\infty$-category of partition Lie algebras. For instance, the surjections operad also provides us with point-set models for partition $\mathcal{L}_\infty$-algebras, as in \cite[Definition 4.46]{pdalgebras}. As to why we do not work with this particular model, see Subsection \ref{subsection: comparison results}.
\end{Remark}

\begin{Remark}
As a particular case of the results in \cite{deuxiemepapier}, one can \textit{directly} show that the $\infty$-category of formal moduli problems defined over $\mathbb{E}_\infty$-algebras is equivalent to the $\infty$-category of this point-set version partition $\mathcal{L}_\infty$-algebras localized at quasi-isomorphisms, without resorting to their $\infty$-categorical definition. 
\end{Remark}

\subsection{Absolute partition $\mathcal{L}_\infty$-algebras}\label{subsection: absolute partition} In this subsection, we introduce the \textit{absolute} version of partition $\mathcal{L}_\infty$-algebras. These can be considered the positive characteristic analogues of the absolute $\mathcal{L}_\infty$-algebras introduced in \cite{lucio2022integration}. Absolute types of algebras appear when one considers algebras over the dual cooperad instead of algebras over the operad; here we consider algebras over the cooperad $\mathrm{B}\mathcal{E}^{\mathrm{nu}}$. We also give an explicit algebraic description of this structure.

\begin{Definition}[Absolute partition $\mathcal{L}_\infty$-algebra]
An \textit{absolute partition} $\mathcal{L}_\infty$\textit{-algebra} $\mathfrak{g}$ amounts the data $(\mathfrak{g},\gamma_\mathfrak{g},d_\mathfrak{g})$ of a dg $\mathrm{B}\mathcal{E}^{\mathrm{nu}}$-algebra. 
\end{Definition}

\begin{Remark}
For the definition of an algebra over a cooperad, introduced in \cite{linearcoalgebras}, see for instance \cite[Section 3]{linearcoalgebras}, \cite[Section 1]{absolutealgebras} or \cite[Section 3]{bricevictor}.  
\end{Remark}

Let us make this definition more explicit. An absolute $\mathcal{L}^\pi_\infty$-algebra structure on a dg module $(\mathfrak{g},d_{\mathfrak{g}})$ amounts to the data of a structural morphism 

\[
\gamma_\mathfrak{g}: \prod_{n \geq 1} \mathrm{Hom}_{\mathbb{S}_n}\left(\mathrm{B}\mathcal{E}^{\mathrm{nu}}(n), \mathfrak{g}^{\otimes n}\right) \longrightarrow \mathfrak{g}~,
\]

which satisfies the following conditions: it is compatible with the differentials and it satisfies the associativity condition of an algebra over a monad, which is given by the left-hand side endofunctor. Unlike algebras over an operad, this monad involves a product over the arity instead of a direct sum. 

\begin{lemma}\label{Lemma: Iso}
Let $(\mathfrak{g},d_{\mathfrak{g}})$ be a dg module. There is an isomorphism of dg modules
\[
\prod_{n \geq 1} \mathrm{Hom}_{\mathbb{S}_n}\left(\mathrm{B}\mathcal{E}^{\mathrm{nu}}(n), \mathfrak{g}^{\otimes n}\right) \cong \prod_{n \geq 1} \Omega (\mathcal{E}^{\mathrm{nu}})^*(n) \otimes_{\mathbb{S}_n} \mathfrak{g}^{\otimes n}~,
\]
natural in $\mathfrak{g}$. 
\end{lemma}

\begin{proof}
There is an isomorphism of dg modules
\[
\mathrm{Hom}_{\mathbb{S}_n}\left(\mathrm{B}\mathcal{E}^{\mathrm{nu}}(n), \mathfrak{g}^{\otimes n}\right) \cong \left(\Omega (\mathcal{E}^{\mathrm{nu}})^*(n)  \otimes \mathfrak{g}^{\otimes n}\right)^{\mathbb{S}_n}~,
\]

since $\mathrm{B}\mathcal{E}^{\mathrm{nu}}(n)$ is a degree-wise finite dimension bounded below dg module. Furthermore, there is an isomorphism 
\[
\left(\Omega (\mathcal{E}^{\mathrm{nu}})^*(n) \otimes \mathfrak{g}^{\otimes n}\right)^{\mathbb{S}_n} \cong \Omega (\mathcal{E}^{\mathrm{nu}})^*(n) \otimes_{\mathbb{S}_n} \mathfrak{g}^{\otimes n}~,
\]

given by the norm map, since $\Omega (\mathcal{E}^{\mathrm{nu}})^*(n)$ is a quasi-free dg $\mathbb{S}_n$-module for all $n \geq 0$ (meaning it is a degree-wise free $\mathbb{S}_n$-module). See \cite[Section 1.2]{bricevictor} for more details on quasi-freeness.
\end{proof}

A \textit{symmetric rooted tree} is a rooted tree where vertices have at least two incoming edges and where each vertex $v$ is labelled by a $(r_v+1)$-tuple $(\sigma_{0}, \cdots, \sigma_{r_v})$ of permutations $\sigma_j$ in $\mathbb{S}_{\mathrm{In}(v)}$, where $\mathrm{In}(v)$ is the number of incoming edges of $v$ and $r_v$ is an integer $r_v \geq 0$. We ask that $\sigma_i \neq \sigma_{i+1}$ for all $0 \leq i \leq r_v-1$. The \textit{degree} $\mathrm{deg}(\tau)$ of a symmetric rooted tree $\tau$ is given by 

\[
\mathrm{deg}(\tau) \coloneqq \sum_{v \in \mathrm{V}(\tau)} r_v + 1~,
\]

where $\mathrm{V}(\tau)$ is the set of all vertices of $\tau$. We denote $\mathrm{SRT}^{\delta}$ the finite set of symmetric rooted trees of degree $\delta$, where $\delta \geq 0$. The \textit{arity} of a symmetric corked rooted tree is the number of leaves. We denote by $\mathrm{SRT}_n^\delta$ the set of symmetric rooted trees of degree $\delta$ and of arity $n$. In summary, these are rooted trees labelled by elements in the Barratt--Eccles operad; they form a vector basis of the $\mathbb{S}$-module $\mathrm{B}\mathcal{E}^{\mathrm{nu}}(n)$.

\medskip

We will denote by $c_n^{(\sigma_0,\cdots,\sigma_r)}$ the $n$-corolla labelled by $(\sigma_0,\cdots,\sigma_r)$. Weight one symmetric rooted trees are given by $n$-corollas labelled by a single permutation $\sigma$ in $\mathbb{S}_n$. The only symmetric rooted tree of weight (and of degree) $0$ is the trivial tree of arity one with zero vertices.

\begin{lemma}\label{Lemma: basis elements}
Let $\mathfrak{g}$ be a graded module with a basis $\left\{ g_b ~|~b \in B \right\}~.$ The graded module
\[
\prod_{n \geq 1} \Omega(\mathcal{E}^{\mathrm{nu}})^*(n) \otimes_{\mathbb{S}_n} \mathfrak{g}^{\otimes n}
\]
admits a vector basis given by formal power series of linear combinations of symmetric rooted trees with leaves labelled by the basis elements of $\mathfrak{g}$. That is, elements of the form
\[
\sum_{n\geq 1} \sum_{\delta \geq 0} \sum_{\tau \in \mathrm{SRT}_n^\delta} \sum_{j \in \mathrm{I}_\tau} \lambda_\tau^{(j)} \tau \left(g_{i_1}^{(j)}, \cdots, g_{i_n}^{(j)}\right)~,
\]
where $\mathrm{I}_\tau$ is a finite set, $\lambda_\tau$ is a scalar in $\kk$ and $(i_1,\cdots,i_n)$ is in $B^n$. Here $\tau(g_{i_1}, \cdots, g_{i_n})$ refers to the symmetric rooted tree $\tau$ with input leaves decorated by the elements $\tau(g_{i_1}, \cdots, g_{i_n})$. The homological degree of $\tau(g_{i_1}, \cdots, g_{i_n})$ is given by $\sum_j \mathrm{deg}(g_{i_j})- \mathrm{deg}(\tau)$, where $\mathrm{deg}(\tau)$ is the degree of $\tau$ defined above.
\end{lemma}

\begin{proof}
It follows by direct inspection from the description of $\Omega(\mathcal{E}^{\mathrm{nu}})^*(n)$ in terms of symmetric rooted trees, since it is the linear dual of $\mathrm{B}\mathcal{E}^{\mathrm{nu}}$. 
\end{proof}

Let $(\mathfrak{g},d_\mathfrak{g})$ be a dg module. Let us describe the differential on the dg module 
\[
\prod_{n \geq 1} \Omega(\mathcal{E}^{\mathrm{nu}})^*(n) \otimes_{\mathbb{S}_n} \mathfrak{g}^{\otimes n}~.
\]
It is given by the sum of three terms $d_1$, $d_2$ and $d_3$. Let us describe their image on a symmetric rooted tree $\tau(g_{i_1}, \cdots, g_{i_n})$ with leaves labelled by elements in $\mathfrak{g}$; their images on a formal power series of such trees is obtained by the formal sum of their images on each tree.

\medskip

\begin{enumerate}
\item The first term $d_1$ is given by the sum 
\[
d_1(\tau(g_{i_1}, \cdots, g_{i_n})) = \sum_{j = 1}^n (-1)^{\epsilon} ~ \tau\left(g_{i_1}, \cdots, d_{\mathfrak{g}}(g_{i_j}), \cdots, g_{i_n}\right)~,
\]
where the sign $\epsilon = \sum_{k=1}^{j-1}|g_{i_k}|$ is determined by the Koszul sign rule.

\medskip

\item The second term $d_2$ is determined by its image on $n$-corollas $c_n^{(\sigma_0,\cdots,\sigma_r)}$, as it corresponds to the differential on $(\mathcal{E}^{\mathrm{nu}})^*$ (the linear dual of the differential in \cite{BergerFresse04}). For such a corolla, it is given by  
\[
d_2\left( c_n^{(\sigma_0,\cdots,\sigma_r)}(g_{i_1}, \cdots, g_{i_n}) \right) = \sum_{i=0}^r (-1)^i \sum_{\substack{\sigma \in \mathbb{S}_n\\ \sigma \neq \sigma_{i-1},\sigma_{i}}} c_n^{(\sigma_0, \cdots, \sigma, \cdots, \sigma_r)}(g_{i_1}, \cdots, g_{i_n})~,
\]
where the first sum is taken over all $\sigma$ in $\mathbb{S}_n$ which are different from the permutation $\sigma_{i-1}$ right before and the permutation $\sigma_{i}$ right after, since $\sigma$ is added in the $i$-th spot. Equivalently, one can sum over all permutations $\sigma$ in $\mathbb{S}_n$ and declare that any resulting tuple with two equal consecutive permutations is sent to zero. The value on a general tree $\tau(g_{i_1}, \cdots, g_{i_n})$ is given by the sum over all vertices of $\tau$ of the above image, together with the appropriate sign.

\medskip

\item The third term $d_3$ is induced by the partial decompositions maps of partial cooperad $(\mathcal{E}^{\mathrm{nu}})^*$. Applied to a corolla labelled by $(\sigma_{0}, \cdots, \sigma_{r_v})$, it gives a sum of symmetric rooted trees with two vertices $v^{(1)}$ and $v^{(2)}$, labelled by all the possible partial decompositions of the operation $(\sigma_{0}, \cdots, \sigma_{r_v})$ inside $(\mathcal{E}^{\mathrm{nu}})^*$. The corolla $v^{(2)}$ lies on the $i$-th input of $v^{(1)}$ precisely when the partial decomposition $\Delta_i$ has been applied to $(\sigma_{0}, \cdots, \sigma_{r_v})$. 

\medskip

In general, the term $d_3$ is the sum over all the vertices of a symmetric rooted tree of the previous rule on corollas. For more concrete formulas of the partial decompositions of $(\mathcal{E}^{\mathrm{nu}})^*$, see the Appendix \ref{Appendice formules}. All the previous sums involve signs which are determined by the Koszul sign rule. 

\medskip
\end{enumerate} 

\begin{Remark}[The norm isomorphism]\label{Remark: norm map}
Let us describe the isomorphism induced by the norm map in Lemma \ref{Lemma: Iso}. It is given by 
\[
\begin{tikzcd}[column sep=1pc,row sep=0.5pc]
\mathbb{N}: \Omega(\mathcal{E}^{\mathrm{nu}})^*(n) \otimes_{\mathbb{S}_n} \mathfrak{g}^{\otimes n} \arrow[r]
&\mathrm{Hom}_{\mathbb{S}_n}\left(\mathrm{B}\mathcal{E}^{\mathrm{nu}}(n), \mathfrak{g}^{\otimes n}\right) \\
\tau(g_{1},\cdots ,g_{n}) \arrow[r,mapsto]
&\displaystyle \sum_{\sigma \in \mathbb{S}_n} \left[\mathrm{ev}_{(g_{\sigma^{-1}(1)},\cdots, g_{\sigma^{-1}(n)})}(\sigma \bullet \tau): \sigma \bullet \tau \mapsto g_{\sigma^{-1}(1)} \otimes \cdots \otimes g_{\sigma^{-1}(n)} \right]~.
\end{tikzcd}
\]
where $\mathrm{ev}_{(g_{1},\cdots, g_{n})}(\tau)$ is a Dirac function which is zero everywhere, except on $\tau$, where its value is given by $g_{1} \otimes \cdots \otimes g_{n}$. This allows us to pass from one presentation to the other. For example, this morphism is given on corollas labelled by $w = (\sigma_0,\cdots,\sigma_r)$ as:
\[
\mathbb{N}\left(c_n^{w}(g_{1},\cdots ,g_{n})\right) = \sum_{\sigma \in \mathbb{S}_n} \left[\mathrm{ev}_{(g_{\sigma^{-1}(1)},\cdots, g_{\sigma^{-1}(n)})}\left(c_n^{\sigma.w}\right): c_n^{\sigma.w} \mapsto (g_{\sigma^{-1}(1)},\cdots, g_{\sigma^{-1}(n)}) \right]~,
\]

where $\sigma.w$ is given by $(\sigma.\sigma_0,\cdots,\sigma.\sigma_r)$. 
\end{Remark}

\textbf{Structural data.} The structural data of an absolute $\mathcal{L}^\pi_\infty$-algebra structure thus amounts to the data of a map of dg modules
\[
\gamma_\mathfrak{g}: \prod_{n \geq 1} \Omega(\mathcal{E}^{\mathrm{nu}})^*(n) \otimes_{\mathbb{S}_n} \mathfrak{g}^{\otimes n} \longrightarrow \mathfrak{g}~,
\]
which sends any formal power series of the form
\[
\sum_{n\geq 1} \sum_{\delta \geq 0} \sum_{\tau \in \mathrm{SCRT}_n^\delta} \sum_{i \in \mathrm{I}_\tau} \lambda_\tau^{(i)} \tau \left(g_{i_1}^{(i)}, \cdots, g_{i_n}^{(i)}\right)
\]
to a well-defined image 
\[
\gamma_\mathfrak{g}\left( \sum_{n\geq 1} \sum_{\delta \geq 0} \sum_{\tau \in \mathrm{SCRT}_n^\delta} \sum_{i \in \mathrm{I}_\tau} \lambda_\tau^{(i)} \tau \left(g_{i_1}^{(i)}, \cdots, g_{i_n}^{(i)}\right)\right)~,
\]
in the dg module $\mathfrak{g}$. This \textbf{does not presuppose} an underlying topology on $\mathfrak{g}$. 

\medskip

The \textit{elementary operations} induced by the structural map $\gamma_\mathfrak{g}$ are the following family of operations 
\[
\left\{l_n^{(\sigma_0,\cdots,\sigma_r)} \coloneqq \gamma_\mathfrak{g} \left( c_n^{(\sigma_0,\cdots,\sigma_r)}(-,\cdots,-) \right): \mathfrak{g}^{\otimes n} \longrightarrow \mathfrak{g}~\right\}~,
\]
where $l_n^{(\sigma_0,\cdots,\sigma_r)}$ is of degree $-r-1$, for all $n \geq 2$, and all $r+1$-tuples of permutations in $\mathbb{S}_n$, where $r \geq 0$ and $\sigma_i \neq \sigma_{i+1}$ for $0 \leq i \leq r-1$. Recall that $c_n$ stands for the $n$-corolla.

\medskip

\textbf{Quasi-planar completeness.} The conilpotent dg cooperad $\mathrm{B}\mathcal{E}^{\mathrm{nu}}$ is a \textit{quasi-planar} conilpotent dg cooperad, and therefore it admits a canonical filtration, called the \textit{quasi-planar filtration}. This filtration can be considered an analogue of the coradical filtration in a positive characteristic setting. We refer to  \cite[Subsection 2.6]{bricevictor} for the definition of a quasi-planar cooperad, to \cite[Subsection 2.7]{bricevictor} for the proof that $\mathrm{B}\mathcal{E}^{\mathrm{nu}}$ is quasi-planar and to \cite[Subsection 2.10]{bricevictor} for the definition of the canonical quasi-planar filtration of a quasi-planar cooperad. 

\medskip

The quasi-planar filtration on $\mathrm{B}\mathcal{E}^{\mathrm{nu}}$ an exhaustive filtration 

\[
0 \hookrightarrow F_0 \mathrm{B}\mathcal{E}^{\mathrm{nu}} \hookrightarrow F_1 \mathrm{B}\mathcal{E}^{\mathrm{nu}} \hookrightarrow F_2 \mathrm{B}\mathcal{E}^{\mathrm{nu}} \hookrightarrow \cdots \hookrightarrow \colim_{\delta} F_\delta \mathrm{B}\mathcal{E}^{\mathrm{nu}} \cong \mathrm{B}\mathcal{E}^{\mathrm{nu}}~, 
\]

where $F_\delta \mathrm{B}\mathcal{E}^{\mathrm{nu}}$ is the conilpotent sub-cooperad that contains symmetric rooted trees of degree at most $\delta$. This filtration induces a canonical filtration on any absolute $\mathcal{L}^\pi_\infty$-algebra, which we call the qp-filtration.

\medskip

Let $\mathfrak{g}$ be a absolute $\mathcal{L}^\pi_\infty$-algebra, its \textit{qp-filtration} $\mathrm{W}_\delta\mathfrak{g}$ is defined as the following pushout 

\[
\begin{tikzcd}[column sep=3pc,row sep=4pc]
    \displaystyle \prod_{n \geq 1} \mathrm{Hom}_{\mathbb{S}_n}\left(\mathrm{B}\mathcal{E}^{\mathrm{nu}}(n), \mathfrak{g}^{\otimes n}\right) \ar[r, "\pi_\delta",twoheadrightarrow] \arrow[dr, phantom, "\ulcorner", very near end]  \ar[d,"\gamma_\mathfrak{g}",swap]
    & \displaystyle \prod_{n \geq 1} \mathrm{Hom}_{\mathbb{S}_n}\left(F_\delta\mathrm{B}\mathcal{E}^{\mathrm{nu}}(n), \mathfrak{g}^{\otimes n}\right)
    \ar[d]
    \\
    \mathfrak{g} \ar[r,twoheadrightarrow]
    &\mathrm{W}_\delta\mathfrak{g}~, 
\end{tikzcd}
\]

where $\pi_\delta$ is the projection induced by the inclusion $F_\delta\mathrm{B}\mathcal{E}^{\mathrm{nu}} \hookrightarrow \mathrm{B}\mathcal{E}^{\mathrm{nu}}$. By Lemmas \ref{Lemma: Iso} and \ref{Lemma: basis elements}, it can be computed that 

\[
\mathrm{W}_\delta\mathfrak{g} = \mathrm{Im}\left(\gamma_\mathfrak{g} \vert_{W^\delta} : \prod_{n \geq 0} W^\delta \Omega(\mathcal{E}^{\mathrm{nu}})^*(n) \otimes_{\mathbb{S}_n}  \mathfrak{g}^{\otimes n} \longrightarrow \mathfrak{g} \right)~,
\]

for any $\delta \geq 0$, where $W^\delta\Omega(\mathcal{E}^{\mathrm{nu}})^*$ is the sub-dg-$\mathbb{S}$-module of symmetric rooted trees of degree greater or equal to $\delta$. An element $g$ in $\mathfrak{g}$ is of \textit{weight} $\delta_0$, meaning it is in $\mathrm{W}_{\delta_0}\mathfrak{g}$, if and only if it can be written as
\[
g = \gamma_\mathfrak{g}\left(\sum_{\delta \geq \delta_0} \sum_{\tau \in \mathrm{SCRT}(\delta)} \lambda_\tau \tau(g_{i_1}, \cdots, g_{i_n})\right)~. 
\]
Each step of the qp-filtration $\mathrm{W}_\delta\mathfrak{g}$ is an ideal of $\mathfrak{g}$, meaning $\mathfrak{g}/\mathrm{W}_\delta\mathfrak{g}$ has an unique absolute $\mathcal{L}^\pi_\infty$-algebra structure induced by the structure of $\mathfrak{g}$. For example, $\mathfrak{g}/\mathrm{W}_1\mathfrak{g}$ is a dg module with a trivial absolute $\mathcal{L}^\pi_\infty$-algebra structure, given by the generators of $\mathfrak{g}$. We refer to \cite[Section 3.6]{bricevictor} for more details on qp-filtrations. 

\begin{Definition}[Qp-complete absolute $\mathcal{L}^\pi_\infty$-algebra] 
Let $(\mathfrak{g},\gamma_\mathfrak{g}, d_\mathfrak{g})$ be an absolute $\mathcal{L}^\pi_\infty$-algebra. It is \textit{qp-complete} if the canonical map
\[
\varphi_\mathfrak{g}: \mathfrak{g} \twoheadrightarrow \lim_{\delta} \mathfrak{g}/\mathrm{W}_\delta\mathfrak{g}
\]
is an isomorphism of absolute $\mathcal{L}^\pi_\infty$-algebras.
\end{Definition}

\begin{Remark}
The map $\varphi_\mathfrak{g}$ is always an epimorphism, see \cite[Lemma 29]{bricevictor}. Thus being qp-complete essentially means that the topology induced by the canonical filtration is separated. 
\end{Remark}

Qp-complete absolute $\mathcal{L}^\pi_\infty$-algebras form a reflexive full sub-category of absolute $\mathcal{L}^\pi_\infty$-algebras, where the reflector is given, for an absolute $\mathcal{L}^\pi_\infty$-algebra $\mathfrak{g}$, by the completion 
\[
\widehat{\mathfrak{g}} \coloneqq \lim_{\delta} \mathfrak{g}/\mathrm{W}_\delta\mathfrak{g}~.
\]
In general, we will restrict to qp-complete absolute $\mathcal{L}^\pi_\infty$-algebras, as they are algebraically and homotopically better behaved. 

\medskip

\textbf{Comparison functors.} We can compare absolute $\mathcal{L}^\pi_\infty$-algebras with $\mathcal{L}^\pi_\infty$-algebras via the following adjunction.

\begin{lemma}\label{lemma: abs res adjunction}
There is an adjunction 
\[
\begin{tikzcd}[column sep=5pc,row sep=3pc]
            \mathcal{L}^\pi_\infty\text{-}\mathsf{alg} \arrow[r, shift left=1.1ex, "\mathrm{Ab}" {name=F}] &~\mathsf{abs}~\mathcal{L}^\pi_\infty\text{-}\mathsf{alg} \arrow[l, shift left=.75ex, "\mathrm{Res}"{name=U}]
            \arrow[phantom, from=F, to=U, , "\dashv" rotate=-90]
\end{tikzcd}
\]
between the category of absolute $\mathcal{L}^\pi_\infty$-algebras and the category of $\mathcal{L}^\pi_\infty$-algebras, where the right adjoint $\mathrm{Res}$ is given by restricting the structure to the elementary operations.
\end{lemma}

\begin{proof}
The canonical inclusion map of monads 
\[
\iota: \bigoplus_{n \geq 1} \Omega (\mathcal{E}^{\mathrm{nu}})^*(n) \otimes_{\mathbb{S}_n} (-)^{\otimes n} \longrightarrow \prod_{n \geq 1} \Omega (\mathcal{E}^{\mathrm{nu}})^*(n) \otimes_{\mathbb{S}_n} ~ (-)^{\otimes n}~,
\]
induces an adjunction between their respective categories of algebras. The restriction along this morphism of the structure map 
of an absolute $\mathcal{L}^\pi_\infty$-algebra is given by the elementary operations. 
\end{proof}

\begin{Definition}[Nilpotent partition $\mathcal{L}_\infty$-algebra]\label{def: nilpotent partition Linfiniti algebra}
Let $(\mathfrak{h},\gamma_\mathfrak{h},d_\mathfrak{h})$ be a $\mathcal{L}^\pi_\infty$-algebra. It is \textit{nilpotent} if the structural morphism 
\[
\gamma_\mathfrak{h}: \bigoplus_{n \geq 1} \Omega (\mathcal{E}^{\mathrm{nu}})^*(n) \otimes_{\mathbb{S}_n} ~ \mathfrak{h}^{\otimes n} \longrightarrow \mathfrak{h}~.
\]
factors through 
\[
\gamma_{\mathfrak{h}}: \bigoplus_{n \geq 1}^{k} F_\delta\Omega(\mathcal{E}^{\mathrm{nu}})^*(n) \otimes_{\mathbb{S}_n} \mathfrak{h}^{\otimes n} \longrightarrow \mathfrak{h}
\]
for some $k \geq 1$ and some $\delta \geq 0$, where $F_\delta\Omega(\mathcal{E}^{\mathrm{nu}})^*$ is the sub-operad containing only symmetric rooted trees of degree $\leq \delta$. 
\end{Definition}

\begin{Remark}
One can define a \textit{lower central series} for partition $\mathcal{L}_\infty$-algebras, in an analogue way to \cite[Definition 4.2]{Getzler09}. Being \textit{arity-nilpotent} amounts to this lower central series terminating. Nevertheless, \textit{arity-nilpotency} does not imply \textit{weight-nilpotency} in this case, as there are operations of arbitrarily high weight in each arity, and vice-versa. Our notion of nilpotency requires both being arity-nilpotent and weight-nilpotent. 
\end{Remark}

Any nilpotent partition $\mathcal{L}_\infty$-algebra is a qp-complete absolute partition $\mathcal{L}_\infty$-algebra. In fact,  the restriction functor of Lemma \ref{lemma: abs res adjunction} is fully faithful on nilpotent partition $\mathcal{L}_\infty$-algebras, and the adjunction restricts to an equivalence on nilpotent objects.

\begin{Remark}[Pro-nilpotent partition $\mathcal{L}_\infty$-algebras]
Since qp-complete absolute partition $\mathcal{L}_\infty$-algebras are stable under limits, which are computed in the ground category of dg modules, any limit of nilpotent partition $\mathcal{L}_\infty$-algebras is again a qp-complete absolute partition $\mathcal{L}_\infty$-algebra. In particular, any pro-nilpotent partition $\mathcal{L}_\infty$-algebras is. 
\end{Remark}

\textbf{Maurer--Cartan elements.} Finally, we define the analogue of the Maurer--Cartan equation for absolute $\mathcal{L}^\pi_\infty$-algebras. 

\begin{Definition}[Maurer--Cartan equation]
Let $(\mathfrak{g},\gamma_\mathfrak{g}, d_\mathfrak{g})$ be an absolute $\mathcal{L}^\pi_\infty$-algebra. A \textit{Maurer--Cartan element} $\alpha$ is an element in $\mathfrak{g}$ of degree $0$ which satisfies the following equation
\[
\gamma_\mathfrak{g}\left(\sum_{n \geq 2} c_n^{\mathrm{id}}(\alpha, \cdots, \alpha) \right) + d_\mathfrak{g}(\alpha) = 0~.
\]
The set of Maurer--Cartan elements in $\mathfrak{g}$ is denoted by $\mathcal{MC}(\mathfrak{g})~.$
\end{Definition}

\begin{Remark}
Let $\mathfrak{g}$ be a qp-complete absolute $\mathcal{L}^\pi_\infty$-algebra. If $\alpha$ is in weight $1$, the above sum \textit{splits} as the corresponding infinite sum, in $\mathfrak{g}$, of the elementary operations indexed by the identity permutation applied to $\alpha$, by qp-completeness.
\end{Remark}

\subsection{Curved absolute partition $\mathcal{L}_\infty$-algebras}\label{subsection: curved absolute partition} Finally, we introduce the more general notion of a curved absolute partition $\mathcal{L}_\infty$-algebra. Absolute partition $\mathcal{L}_\infty$-algebras, like their $\mathcal{L}_\infty$ counterparts, are \textit{pointed}, in the sense that the element $0$ is always a Maurer--Cartan element. One can think of curved absolute $\mathcal{L}_\infty^\pi$-algebras as the \textit{unpointed} version of this structure, where adding a distinguished element, called the \textit{curvature}, prevents $0$ from being a Maurer--Cartan element. We refer to the introduction of \cite{lucio2022integration} for more details about the curved setting. The goal of this subsection, parallel to Subsection \ref{subsection: absolute partition}, is to make the definition of curved absolute partition $\mathcal{L}_\infty$-algebras explicit. 

\medskip

We start by considering the dg operad $\mathcal{E}$, which is the unital version of the Barratt-Eccles dg operad, meaning that $\mathcal{E}(0) = \kk$. This dg operad is no longer augmented, hence we cannot perform the classical operadic bar construction on $\mathcal{E}$. We consider instead the conilpotent curved cooperad $\mathrm{B}^{\mathrm{s.a}}\mathcal{E}$, where $\mathrm{B}^{\mathrm{s.a}}$ stands for the semi-augmented bar construction of \cite{HirshMilles12}. See also \cite[Appendix]{lucio2022integration} for more details about this bar construction. 

\medskip

A vector basis of $\mathrm{B}^{\mathrm{s.a}}\mathcal{E}$ is given by symmetric \textit{corked} rooted trees. A symmetric corked rooted tree is a rooted tree where vertices either have at least two incoming edges or zero incoming edges. Vertices with zero incoming edges are called \textit{corks}. Each vertex with at least two incoming edges is labelled by a $r_v+1$-uple $(\sigma_{0}, \cdots, \sigma_{r_v})$, where $\sigma_j$ is an element of $\mathbb{S}_{\mathrm{In}(v)}$, where $r_v \geq 0$, and where $\sigma_i \neq \sigma_j$ for $i \neq j$. The \textit{degree} $\mathrm{deg}(\tau)$ of a symmetric corked rooted tree $\tau$ is given by 

\[
\mathrm{deg}(\tau) \coloneqq \mathrm{Cork}(\tau) + \sum_{v \in \mathrm{V}(\tau)} r_v + 1~,
\]

where $\mathrm{Cork}(\tau)$ is the number of corks of $\tau$ and where $\mathrm{V}(\tau)$ is the set of all vertices of $\tau$. The arity of a tree is given by the number of leaves and the weight by the number of vertices (including the corks). We denote $\mathrm{SCRT}_n^\omega$ denote the set of symmetric corked rooted trees of arity $n$ and with $\omega$ internal edges. Notice that any symmetric rooted trees are included in symmetric corked rooted trees. 

\begin{lemma}\label{lemma: quasi-planar}\leavevmode
\begin{enumerate}
\item The conilpotent curved cooperad $\mathrm{B}^{\mathrm{s.a}}\mathcal{E}$ is quasi-planar.

\medskip

\item Its canonical quasi-planar filtration is given by 
\[
0 \hookrightarrow F_0 \mathrm{B}^{\mathrm{s.a}}\mathcal{E} \hookrightarrow F_1 \mathrm{B}^{\mathrm{s.a}}\mathcal{E} \hookrightarrow F_2 \mathrm{B}^{\mathrm{s.a}}\mathcal{E} \hookrightarrow \cdots \hookrightarrow \colim_{\delta} F_\delta \mathrm{B}^{\mathrm{s.a}}\mathcal{E} \cong \mathrm{B}^{\mathrm{s.a}}\mathcal{E}~, 
\]
where $F_\delta \mathrm{B}^{\mathrm{s.a}}\mathcal{E}$ is the conilpotent curved sub-cooperad containing symmetric corked rooted trees of degree at most $\delta$. 
\end{enumerate}
\end{lemma}

\begin{proof}
The analogue of this result for $\mathrm{B}\mathcal{E}$ is proven in \cite[Section 2.7]{bricevictor}, and it holds \textit{mutatis mutandis} for the semi-augmented bar construction as well. Let us explain why it is quasi-planar: the filtration induced by considering only symmetric corked rooted trees of degree $\leq \delta$ induces an $\omega$-ladder of conilpotent curved cooperads whose colimit is $\mathrm{B}^{\mathrm{s.a}}\mathcal{E}$. The underlying conilpotent graded cooperads of each step of the filtration is clearly planar and the differential simply vanishes on the associated graded, hence it is a quasi-planar filtration. The canonical quasi-planar filtration is induced by the $\mathcal{E}$-comodule structure of $\Omega\mathrm{B}^{\mathrm{s.a}}\mathcal{E}$, which can be explicitly computed like in \cite[Proposition 16]{bricevictor} and coincides with the above filtration.
\end{proof}

\begin{Warning}
When considering curved structures, we will work over the base category of \textit{pre-differential (pdg) modules}, which are graded modules with a degree $-1$ endomorphism which does not necessarily square to zero. 
\end{Warning}

\begin{Definition}[Curved absolute partition $\mathcal{L}_\infty$-algebra]
A \textit{curved absolute partition} $\mathcal{L}_\infty$\textit{-algebra} $\mathfrak{g}$ amounts the data $(\mathfrak{g},\gamma_\mathfrak{g},d_\mathfrak{g})$ of a curved $\mathrm{B}^{\mathrm{s.a}}\mathcal{E}$-algebra. 
\end{Definition}

\textbf{Structural data.} A curved absolute $\mathcal{L}^\pi_\infty$-algebra structure on a pdg module $(\mathfrak{g},d_{\mathfrak{g}})$ amounts to the data of a structural morphism 
\[
\gamma_\mathfrak{g}: \prod_{n \geq 0} \mathrm{Hom}_{\mathbb{S}_n}\left(\mathrm{B}^{\mathrm{s.a}}\mathcal{E}(n), \mathfrak{g}^{\otimes n}\right) \longrightarrow \mathfrak{g}~,
\]

which satisfies the conditions of a curved algebra over a curved cooperad. These amount to a compatibility condition with the pre-differentials, the associativity condition of an algebra over a monad and finally a compatibility condition with the curvature of the cooperad. We refer to \cite[Section 3.2]{bricevictor} for more details. 

\medskip

The linear dual dg operad of $\mathrm{B}^{\mathrm{s.a}}\mathcal{E}$ is given by $\widehat{\Omega}^{\mathrm{s.a}}\mathcal{E}^*$, the semi-coaugmented cobar construction of the linear dual of $\mathcal{E}^*$, see \cite[Appendix]{lucio2022integration}. This dg operad admits a linear basis given by formal power series of symmetric corked rooted trees. Like in Lemma \ref{Lemma: Iso}, there is an isomorphism of pdg modules 
\[
\prod_{n \geq 0} \mathrm{Hom}_{\mathbb{S}_n}\left(\mathrm{B}^{\mathrm{s.a}}\mathcal{E}(n), \mathfrak{g}^{\otimes n}\right) \cong \prod_{n \geq 0} \widehat{\Omega}^{\mathrm{s.a}}\mathcal{E}^*(n) ~\widehat{\otimes}_{\mathbb{S}_n}~ \mathfrak{g}^{\otimes n}~.
\]

where $\widehat{\otimes}$ denotes the completed tensor product, defined as 
\[
\widehat{\Omega}^{\mathrm{s.a}}\mathcal{E}^*(n) ~\widehat{\otimes}_{\mathbb{S}_n}~ \mathfrak{g}^{\otimes n} \coloneqq \lim_{\delta} \left(F_\delta\widehat{\Omega}^{\mathrm{s.a}}\mathcal{E}^*(n) \otimes_{\mathbb{S}_n} \mathfrak{g}^{\otimes n}\right)
\]
where $F_\delta\widehat{\Omega}^{\mathrm{s.a}}\mathcal{E}^*$ refers to the sub dg-operad which only involves symmetric corked rooted trees of degree $\leq \delta$. The previous isomorphism follows from the fact that this later dg sub-operad is arity-wise finite dimensional and quasi-planar. 

\medskip

The structural data of a curved absolute $\mathcal{L}^\pi_\infty$-algebra structure can thus be rewritten as a map 
\[
\gamma_\mathfrak{g}: \prod_{n \geq 0} \widehat{\Omega}^{\mathrm{s.a}}\mathcal{E}^*(n) ~\widehat{\otimes}_{\mathbb{S}_n} ~ \mathfrak{g}^{\otimes n} \longrightarrow \mathfrak{g}~. 
\]
The product on the left hand side admits a linear basis given by formal power series of symmetric corked rooted trees decorated by elements of $\mathfrak{g}$, similar to Lemma \ref{Lemma: basis elements}. Thus this structural morphism sends any formal power series of the form
\[
\sum_{n\geq 0} \sum_{\delta \geq 0} \sum_{\tau \in \mathrm{SCRT}_n^\delta} \sum_{i \in \mathrm{I}_\tau} \lambda_\tau^{(i)} \tau \left(g_{i_1}^{(i)}, \cdots, g_{i_n}^{(i)}\right)
\]
to a well-defined image 
\[
\gamma_\mathfrak{g}\left( \sum_{n\geq 0} \sum_{\delta \geq 0} \sum_{\tau \in \mathrm{SCRT}_n^\delta} \sum_{i \in \mathrm{I}_\tau} \lambda_\tau^{(i)} \tau \left(g_{i_1}^{(i)}, \cdots, g_{i_n}^{(i)}\right)\right)~,
\]
in the pdg module $\mathfrak{g}$. 

\medskip

The \textit{elementary operations} induced by the structural map $\gamma_\mathfrak{g}$ are defined in the same way as for absolute $\mathcal{L}^\pi_\infty$-algebras, namely by restricting the structural morphism to corollas. In the curved case, there is an extra elementary operation $l_0: \kk \longrightarrow \mathfrak{g}$ called the \textit{curvature}, induced by the image of the single cork by the structural morphism $\gamma_\mathfrak{g}$. 

\medskip

\textbf{Curved condition.} The structural map $\gamma_\mathfrak{g}$ needs to satisfy associativity and a compatibility with the pre-differential conditions that are analogous to those explained in \cite[Section 1]{lucio2022integration}. Furthermore, in order to form a \textit{curved} absolute $\mathcal{L}^\pi_\infty$-algebra structure, the following equation on the elementary operations
\begin{equation}\label{curved condition}
d_\mathfrak{g}^2(g) = l_2^{(12)}(l_0,g) + l_2^{(21)}(l_0,g)~,
\end{equation}
needs to hold. This equation comes precisely from the requirement that the pdg $\mathrm{B}^{\mathrm{s.a}}\mathcal{E}$-algebra structure on $\mathfrak{g}$ needs to be compatible with the curvature of the conilpotent curved cooperad, that is, the diagram in \cite[Definition 44]{bricevictor} needs to commute. 

\medskip

\textbf{Quasi-planar completeness.} We make the same definitions as in Subsection \ref{subsection: absolute partition}, using the quasi-planar filtration of $\mathrm{B}^{\mathrm{s.a}}\mathcal{E}$. The $\delta$-stage of the canonical filtration on a curved absolute $\mathcal{L}^\pi_\infty$-algebra is given by the image by the structural morphism of formal power series which only involve symmetric corked rooted trees of degree equal or higher than $\delta$, for $\delta \geq 0$. \textit{Mutatis mutandis,} the same results hold for this canonical filtration as in the previous subsection. See \cite[Sections 3.6 and 3.7]{bricevictor} for the general case. 

\medskip

\textbf{Comparison with absolute $\mathcal{L}_\infty^\pi$-algebras.} Absolute $\mathcal{L}_\infty^\pi$-algebras are particular examples of curved absolute $\mathcal{L}_\infty^\pi$-algebras, where the curvature $l_0$ is zero. 

\begin{Proposition}\label{Prop: curved non curved comparison}
There is an adjunction 
\[
\begin{tikzcd}[column sep=7pc,row sep=3pc]
\mathsf{curv}~\mathsf{abs}~\mathcal{L}_\infty^\pi\text{-}\mathsf{alg} \arrow[r, shift left=1.1ex, "(-)_*"{name=F}]      
&\mathsf{abs}~\mathcal{L}_\infty^\pi\text{-}\mathsf{alg}~, \arrow[l, shift left=.75ex, "\mathrm{U}"{name=U}]
\arrow[phantom, from=F, to=U, , "\dashv" rotate=-90]
\end{tikzcd}
\]
between curved absolute $\mathcal{L}_\infty^\pi$-algebras and absolute $\mathcal{L}_\infty^\pi$-algebras, where the functor $\mathrm{U}$ is fully faithful. The essential image of $\mathrm{U}$ is given by curved absolute $\mathcal{L}_\infty^\pi$-algebras such that the curvature $l_0$ is zero. 
\end{Proposition}

\begin{proof}
The canonical inclusion of dg operads $\mathcal{E}^{\mathrm{nu}} \hookrightarrow \mathcal{E}$ induces an inclusion of conilpotent curved cooperads $\mathrm{B}\mathcal{E}^{\mathrm{nu}} \hookrightarrow  \mathrm{B}^{\mathrm{s.a}}\mathcal{E}$ (any dg cooperad is a curved cooperad with zero curvature). This morphism induces the above adjunction. It is straightforward to check that $\mathrm{U}$ is fully faithful and to identify its essential image. 
\end{proof}

\textbf{Maurer--Cartan equation.} Let $(\mathfrak{g},\gamma_\mathfrak{g}, d_\mathfrak{g})$ be a curved absolute $\mathcal{L}^\pi_\infty$-algebra. A \textit{Maurer--Cartan element} $\alpha$ is an element in $\mathfrak{g}$ of degree $0$ which satisfies the following equation
\[
\gamma_\mathfrak{g}\left(\sum_{n \geq 0,~n\neq 1} c_n^{\mathrm{id}}(\alpha, \cdots, \alpha) \right) + d_\mathfrak{g}(\alpha) = 0~.
\]
Notice that this equation now involves the curvature in arity zero, thus $0$ is not automatically a Maurer--Cartan element. When the curvature is zero, it specifies to the Maurer--Cartan equation in the previous subsection. 

\subsection{Model structures}\label{subsection: model structures}
The goal of this subsection is to endow the category of qp-complete curved absolute $\mathcal{L}^\pi_\infty$-algebras with a model category structure, transferred from their Koszul dual coalgebras, which are non-necessarily conilpotent homotopy counital cocommutative coalgebras. Specifically, they are coalgebras over the dg operad $\Omega \mathrm{B}^{\mathrm{s.a}}\mathcal{E}$, which we will denote by $u\mathcal{EE}_\infty$. We perform the same constructions in the non-curved/non-counital case, and compare their homotopy theories.

\begin{Proposition}\label{prop: left transferred structure}
There is a model structure on the category of $u\mathcal{EE}_\infty$-coalgebras, left-transferred along the cofree-forgetful adjunction
\[
\begin{tikzcd}[column sep=7pc,row sep=3pc]
\mathsf{dg}~\mathsf{mod} \arrow[r, shift left=1.1ex, "\mathcal{C}(u\mathcal{EE}_\infty)(-)"{name=F}]      
&u\mathcal{EE}_\infty\textsf{-}\mathsf{coalg}~, \arrow[l, shift left=.75ex, "U"{name=U}]
\arrow[phantom, from=F, to=U, , "\dashv" rotate=-90]
\end{tikzcd}
\]
where 
\begin{enumerate}
\item the class of weak equivalences is given by quasi-isomorphisms,
\item the class of cofibrations is given by degree-wise monomorphisms,
\item the class of fibrations is given by right lifting property with respect to acyclic cofibrations.
\end{enumerate}
\end{Proposition}

\begin{proof}
By Lemma \ref{lemma: quasi-planar}, $\mathrm{B}^{\mathrm{s.a}}\mathcal{E}$ is a quasi-planar curved conilpotent cooperad. Therefore the dg operad $\Omega \mathrm{B}^{\mathrm{s.a}}\mathcal{E}$ is cofibrant in the semi-model category of \cite{Fresse09} by \cite[Proposition 11]{bricevictor}. And any cofibrant dg operad is coadmissible. See \cite[Section 3.13]{bricevictor} for more details.
\end{proof}

There is a \textit{complete bar-cobar adjunction} relating $u\mathcal{EE}_\infty$-coalgebras and qp-complete curved absolute $\mathcal{L}^\pi_\infty$-algebras: 
\[
\begin{tikzcd}[column sep=5pc,row sep=3pc]
            u\mathcal{EE}_\infty\text{-}\mathsf{coalg} \arrow[r, shift left=1.1ex, "\widehat{\Omega}_{\iota}"{name=F}] & \mathsf{curv}~\mathsf{abs}~\mathcal{L}^\pi_\infty\text{-}\mathsf{alg}^{\mathsf{qp}\text{-}\mathsf{comp}} ~,  \arrow[l, shift left=.75ex, "\widehat{\mathrm{B}}_{\iota}"{name=U}]
            \arrow[phantom, from=F, to=U, , "\dashv" rotate=-90]
\end{tikzcd}
\]

induced by the twisting morphism $\iota: \mathrm{B}^{\mathrm{s.a}}\mathcal{E} \longrightarrow \Omega \mathrm{B}^{\mathrm{s.a}}\mathcal{E}$, see \cite[Section 3.11]{bricevictor}. Using this adjunction, one can transfer the model structure on $u\mathcal{EE}_\infty$-coalgebras to the category of qp-complete curved absolute $\mathcal{L}^\pi_\infty$-algebras.

\begin{theorem}\label{thm: Equivalence the Quillen CC infini et absolues}
There is a model structure on the category of qp-complete curved absolute $\mathcal{L}^\pi_\infty$-algebras, right-transferred along the complete bar-cobar adjunction
\[
\begin{tikzcd}[column sep=5pc,row sep=3pc]
          u\mathcal{EE}_\infty\text{-}\mathsf{coalg} \arrow[r, shift left=1.1ex, "\widehat{\Omega}_{\iota}"{name=F}] & \mathsf{curv}~\mathsf{abs}~\mathcal{L}^\pi_\infty\text{-}\mathsf{alg}^{\mathsf{qp}\text{-}\mathsf{comp}}~, \arrow[l, shift left=.75ex, "\widehat{\mathrm{B}}_{\iota}"{name=U}]
            \arrow[phantom, from=F, to=U, , "\dashv" rotate=-90]
\end{tikzcd}
\]
where 
\begin{enumerate}
\item the class of weak equivalences is given by morphisms $f$ such that $\widehat{\mathrm{B}}_\iota(f)$ is a quasi-isomorphism,
\item the class of fibrations is given is given by morphisms $f$ which are degree-wise epimorphisms,
\item  and the class of cofibrations is given by left lifting property with respect to acyclic fibrations.
\end{enumerate}
\medskip

Furthermore, this adjunction is a Quillen equivalence. 
\end{theorem}

\begin{proof}
Follows from \cite[Theorems 10 and 11]{bricevictor}, using the quasi-planarity of $\mathrm{B}^{\mathrm{s.a}}\mathcal{E}$. 
\end{proof}

There is a particular kind of weak equivalence between qp-complete curved absolute partition $\mathcal{L}^\pi_\infty$-algebras which admits an easy description. 

\begin{Definition}[Associated graded]  
Let $\mathfrak{g}$ be a curved absolute $\mathcal{L}^\pi_\infty$-algebra. Its \textit{associated graded} of weight $\delta$ is defined as 
\[
\mathrm{gr}_{\delta}(\mathfrak{g}) \coloneqq \mathrm{W}_\delta \mathfrak{g}/\mathrm{W}_{\delta+1} \mathfrak{g}~,
\]
where $\mathrm{W}_\delta \mathfrak{g}$ is the canonical filtration of $\mathfrak{g}$ given by the images of formal power series involving symmetric corked rooted trees of degree greater or equal to $\delta$. 
\end{Definition}

\begin{Remark}
For any $\delta \geq 0$, the associated graded $\mathrm{gr}_{\delta}(\mathfrak{g})$ is a chain complex. Indeed, recall that 
\[
d_\mathfrak{g}^2(-) = l_2^{(12)}(l_0,-) + l_2^{(21)}(l_0,-)~. 
\]
It implies that 
\[
d_\mathfrak{g}^2(\mathrm{W}_\delta \mathfrak{g}) \subseteq \mathrm{W}_{\delta+1} \mathfrak{g}~, 
\]
hence $d_\mathfrak{g}^2$ is zero in the associated graded. 
\end{Remark}

\begin{Definition}[Filtered quasi-isomorphisms] 
Let $f: \mathfrak{g} \longrightarrow \mathfrak{u}$ be a morphism of qp-complete curved absolute partition $\mathcal{L}^\pi_\infty$-algebras. It is a \textit{filtered quasi-isomorphism} if in induces a quasi-isomorphism
\[
\mathrm{gr}_{\delta}(f): \mathrm{gr}_{\delta}\mathfrak{g} \qi \mathrm{gr}_{\delta}\mathfrak{u}
\]

between their respective associated graded complexes, for all $\delta \geq 0$.
\end{Definition}

\begin{Proposition}
Let $f: \mathfrak{g} \longrightarrow \mathfrak{u}$ be a filtered quasi-isomorphism between two qp-complete curved absolute partition $\mathcal{L}^\pi_\infty$-algebras. Then 
\[
\widehat{\mathrm{B}}_\iota(f): \widehat{\mathrm{B}}_\iota(\mathfrak{g}) \qi \widehat{\mathrm{B}}_\iota(\mathfrak{u}) 
\]
is a quasi-isomorphism and $f$ is a weak equivalence of qp-complete curved absolute partition $\mathcal{L}^\pi_\infty$-algebras.
\end{Proposition}

\begin{proof}
Follows from \cite[Section 6.5]{bricevictor}. 
\end{proof}

\textbf{The non-unital case.} We denote by $\mathcal{EE}_\infty$ the dg operad $\Omega \mathrm{B}\mathcal{E}^{\mathrm{nu}}$. Coalgebras over this operad are non-necessarily conilpotent homotopy cocommutative coalgebras, without a (homotopy) counit. The category of $\mathcal{EE}_\infty$-coalgebras admits a left-transferred model structure from dg modules, analogous to the one in Proposition \ref{prop: left transferred structure} (since $\mathrm{B}\mathcal{E}^{\mathrm{nu}}$ is also quasi-planar). There is a \textit{complete bar-cobar} adjunction 
\[
\begin{tikzcd}[column sep=5pc,row sep=3pc]
            \mathcal{EE}_\infty\text{-}\mathsf{coalg} \arrow[r, shift left=1.1ex, "\widehat{\Omega}^\flat_{\iota}"{name=F}] & \mathsf{abs}~\mathcal{L}^\pi_\infty\text{-}\mathsf{alg}^{\mathsf{qp}\text{-}\mathsf{comp}} ~,  \arrow[l, shift left=.75ex, "\widehat{\mathrm{B}}^\flat_{\iota}"{name=U}]
            \arrow[phantom, from=F, to=U, , "\dashv" rotate=-90]
\end{tikzcd}
\]
between $\mathcal{EE}_\infty$-coalgebras and qp-complete absolute $\mathcal{L}^\pi_\infty$-algebras, which we denote by $\widehat{\Omega}^\flat_{\iota} \dashv \widehat{\mathrm{B}}^\flat_{\iota}$, in order to avoid confusion with the previous complete bar-cobar adjunction. Analogously to Theorem \ref{thm: Equivalence the Quillen CC infini et absolues}, one can transfer the model structure on $\mathcal{EE}_\infty$-coalgebras and obtain a Quillen equivalence. Furthermore, filtered quasi-isomorphisms are again particular examples of weak equivalences in the transferred structure again. 

\begin{Proposition}\label{Prop: curved non curved comparison}
The adjunction 
\[
\begin{tikzcd}[column sep=7pc,row sep=3pc]
\mathsf{curv}~\mathsf{abs}~\mathcal{L}_\infty^\pi\text{-}\mathsf{alg}^{\mathsf{qp}\text{-}\mathsf{comp}} \arrow[r, shift left=1.1ex, "(-)_*"{name=F}]      
&\mathsf{abs}~\mathcal{L}_\infty^\pi\text{-}\mathsf{alg}^{\mathsf{qp}\text{-}\mathsf{comp}}~, \arrow[l, shift left=.75ex, "\mathrm{U}"{name=U}]
\arrow[phantom, from=F, to=U, , "\dashv" rotate=-90]
\end{tikzcd}
\]
between qp-complete curved absolute $\mathcal{L}^\pi_\infty$-algebras and qp-complete absolute $\mathcal{L}^\pi_\infty$-algebras is a Quillen adjunction. Furthermore, the functor $\mathrm{U}$ is homotopically fully faithful. 
\end{Proposition}

\begin{proof}
The proof is completely analogous to \cite[Proposition 1.37]{lucio2022integration}.
\end{proof}

Absolute partition $\mathcal{L}_\infty$-algebras can also be homotopically compared to partition $\mathcal{L}_\infty$-algebras. 

\begin{Proposition}\label{Prop: adjunction entre absolues et pas absolues}
The adjunction
\[
\begin{tikzcd}[column sep=5pc,row sep=3pc]
            \mathcal{L}^\pi_\infty\text{-}\mathsf{alg}^{\mathsf{qp}\text{-}\mathsf{comp}} \arrow[r, shift left=1.1ex, "\mathrm{Ab}" {name=F}] &~\mathsf{abs}~\mathcal{L}^\pi_\infty\text{-}\mathsf{alg} \arrow[l, shift left=.75ex, "\mathrm{Res}"{name=U}]
            \arrow[phantom, from=F, to=U, , "\dashv" rotate=-90]
\end{tikzcd}
\]
between qp-complete absolute partition $\mathcal{L}_\infty$-algebras and partition $\mathcal{L}_\infty$-algebras is a Quillen adjunction.
\end{Proposition}

\begin{proof}
The restriction functor preserves fibrations, which are degree-wise surjections in both cases, and weak equivalences, since any weak equivalence in this transferred model structure on absolute partition $\mathcal{L}_\infty$-algebras is in particular a quasi-isomorphism. We refer to \cite[Corollary 11]{bricevictor} for more details on this last point. 
\end{proof}

\begin{Remark}[About the underlying $\infty$-categories]\label{Rmk: underlying infinit cats}
The $\infty$-category obtained by localizing $\mathcal{EE}_\infty$-coalgebras at quasi-isomorphisms should be equivalent to the $\infty$-category of non-counital (equivalently, coaugmented) $\mathbb{E}_\infty$-coalgebras in $\mathsf{Mod}_\kk$, the $\infty$-category of $\kk$-modules. An explicit model for $\mathsf{Mod}_\kk$ is obtained by localizing dg $\kk$-modules at quasi-isomorphisms. See \cite[Section 3.1]{ellipticI} for a precise definition of $\mathbb{E}_\infty$-coalgebras. Proving this statement shall be the subject of future work. 

\medskip

Assuming this rectification result, it is clear that since absolute partition $\mathcal{L}_\infty$-algebras, with the transferred model structure, are Quillen equivalent to $\mathcal{EE}_\infty$-coalgebras, then their underlying $\infty$-category is again that of coaugmented $\mathbb{E}_\infty$-coalgebras in $\mathsf{Mod}_\kk$. However, absolute partition $\mathcal{L}_\infty$-algebras admit a right Bousfield localization at quasi-isomorphism, see \cite[Section 7.5]{bricevictor}. This gives the two following adjunctions at the level of $\infty$-categories:

\[
\begin{tikzcd}[column sep=3pc,row sep=3pc]
          \mathsf{abs}~\mathcal{L}_\infty^\pi\text{-}\mathsf{alg}^{\mathsf{qp}\text{-}\mathsf{comp}}~[\mathrm{Q.iso}^{-1}] \arrow[r, shift left=1.1ex, "\mathrm{Id}"{name=F}]           
         &\mathsf{abs}~\mathcal{L}_\infty^\pi\text{-}\mathsf{alg}^{\mathsf{qp}\text{-}\mathsf{comp}}~[\mathrm{W.eq}^{-1}] \arrow[l, shift left=.75ex, "\mathrm{Id}"{name=U}] \arrow[r, shift left=1.1ex, "\widehat{\mathrm{B}}^\flat_{\iota}"{name=A}]
         &\mathcal{EE}_\infty\text{-}\mathsf{cog}~[\mathrm{Q.iso}^{-1}]~. \arrow[l, shift left=.75ex, "\widehat{\Omega}^\flat_{\iota}"{name=B}] \arrow[phantom, from=F, to=U, , "\dashv" rotate=-90]\arrow[phantom, from=A, to=B, , "\dashv" rotate=90]
\end{tikzcd}
\]

Notice that in the right hand side adjunction, the functors are both left and right adjoints since it is an equivalence of $\infty$-categories. The composite adjunction should be a model for a dual version of the enhanced bar-cobar adjunction of Francis--Gaitsgory in \cite[Section 3.3]{FrancisGaitsgory}; this dual version should be obtained by factorizing topological André--Quillen cohomology of coaugmented $\mathbb{E}_\infty$-coalgebras through algebras over a monad at the $\infty$-categorical level. 
\end{Remark}

\subsection{Monoidal structures and convolution algebras}\label{subsection: convolution}
The category of $u\mathcal{EE}_\infty$-coalgebras is admits a monoidal structure, where the monoidal product is given by the underlying tensor product of dg modules. This monoidal structure is \textit{biclosed} and compatible with the model structure.
However, it is not stricly symmetric, only symmetric up to homotopy.  These constructions are based on the results of \cite{grignou2022mappingII, grignou2022mapping}. See also \cite[Section 1.4]{lucio2022integration} for an explanation of similar constructions. 

\begin{Proposition}
The category of $u\mathcal{EE}_\infty$-coalgebras forms a monoidal model category, where the monoidal product is given by tensor product.
\end{Proposition}

\begin{proof}
The dg operad $\Omega \mathrm{B}^{\mathrm{s.a}}\mathcal{E}$ is isomorphic to the cellular chains of the Boardman-Vogt construction of the simplicial Barratt-Eccles operad. This implies that it inherits a Hopf structure, since the Barratt-Eccles operad is a Hopf operad. Finally, the tensor product trivially satisfies the compatibility conditions of \cite[Chapter 4]{HoveyModel}, that is, the pushout-product and the unit axioms.
\end{proof}

\begin{Proposition}[{\cite[Proposition 16]{grignou2022mappingII}}]\label{Prop: mapping coalgebras}
The category of dg $u\mathcal{EE}_\infty$-coalgebras is a biclosed monoidal category, meaning that there exists a left (resp. right) internal hom bifunctor 

\[
\{-,-\}_{L/R}: \left(u\mathcal{EE}_\infty\text{-}\mathsf{coalg}\right)^{\mathsf{op}} \times u\mathcal{EE}_\infty\text{-}\mathsf{coalg} \longrightarrow u\mathcal{EE}_\infty\text{-}\mathsf{coalg}
\]
\vspace{0.1pc}

and, for any triple of $u\mathcal{EE}_\infty$-coalgebras $C,D,E$, there exists isomorphisms

\[
\mathrm{Hom}_{u\mathcal{EE}_\infty\text{-}\mathsf{coalg}}(C \otimes D, E) \cong \mathrm{Hom}_{u\mathcal{EE}_\infty\text{-}\mathsf{coalg}}(C, \{D,E\}_L)~,
\] 
\[
\mathrm{Hom}_{u\mathcal{EE}_\infty\text{-}\mathsf{coalg}}(C \otimes D, E) \cong \mathrm{Hom}_{u\mathcal{EE}_\infty\text{-}\mathsf{coalg}}(D, \{C,E\}_R)~,
\]
\vspace{0.1pc}

which are natural in $C,D$ and $E$.
\end{Proposition}

\begin{Remark}
The existence of this internal hom functors can also be directly deduced from the adjoint functor theorem. 
\end{Remark}

\textbf{On the non-symmetry of the tensor product.} The Hopf structure on the Barratt-Eccles dg operad is not symmetric, however it is well known that it is symmetric up to coherent homotopies, since it is induced by the Alexander--Whitney map. 
In fact, there are two Hopf structures on $\Omega \mathrm{B}^{\mathrm{s.a}}\mathcal{E}$, each given by the choice of a cellular approximation of the diagonal map of the interval. As in \cite[Remark 1.47]{lucio2022integration}, for any choice of Hopf structure for $\Omega \mathrm{B}^{\mathrm{s.a}}\mathcal{E}$, we will have that 
\[
C \otimes D \simeq D \otimes C~,
\]
for any two $u\mathcal{EE}_\infty$-coalgebras $C$ and $D$. These monoidal structure are both "symmetric up to homotopy", a notion which, as far as we know, has no $1$-categorical definition. Whatever this might mean, it should entail that the induced monoidal structure on the underlying $\infty$-category is symmetric. From now on, we will not specify which choice of cellular approximation of the diagonal of the interval we make, nor will we distinguish between left and right internal hom bifunctors, as they are also homotopy equivalent. 

\medskip

\textbf{Convolution structures.} Given a $u\mathcal{EE}_\infty$-coalgebra $C$ and a curved absolute partition $\mathcal{L}^\pi_\infty$-algebra $\mathfrak{g}$, there is a \textit{convolution curved absolute} $\mathcal{L}^\pi_\infty$-\textit{algebra} structure on the pdg module of graded morphisms $\mathrm{hom}(C,\mathfrak{g})$. Its structure map admits an analogous description to \cite[Definition 1.48]{lucio2022integration}. The key point is that the internal hom and the convolution algebra are compatible in the following sense.

\begin{theorem}[{\cite{grignou2022mapping}}]\label{thm: compatibilité des mappings et des convolutions}
Let $C$ be a $u\mathcal{EE}_\infty$-coalgebra and let $\mathfrak{g}$ be a curved absolute $\mathcal{L}^\pi_\infty$-algebra. There is a natural isomorphism of $u\mathcal{EE}_\infty$-coalgebras

\[
\left\{C, \widehat{\mathrm{B}}_\iota(\mathfrak{g}) \right\} \cong \widehat{\mathrm{B}}_\iota \left(\mathrm{hom}(C,\mathfrak{g}) \right)~,
\]
\vspace{0.1pc}

where $\mathrm{hom}(C,\mathfrak{g})$ denotes the convolution curved absolute $\mathcal{L}^\pi_\infty$-algebra of $C$ and $\mathfrak{g}$.
\end{theorem}

\section{The integration theory}
In this section, we develop the integration theory of (curved) absolute $\mathcal{L}^\pi_\infty$-algebras. We show that the integration functor is well-behaved in an appropriate sense by purely model categorical arguments. Then, we give a combinatorial and algebraic description of this integration functor. We describe the homotopy groups of the resulting simplicial set in terms of the initial algebra. We give explicit combinatorial formulas for the horn-fillers inside this simplicial set: in the simplest case, this formula can be thought as an analogue of the Baker--Campbell--Hausdorff formula for (curved) absolute $\mathcal{L}^\pi_\infty$-algebras. Finally, we prove some comparison results with other possible models and when $\kk$ happens to be a field of characteristic zero, with other known constructions.  

\subsection{Cellular chain functor}\label{subsection: dupont contraction}
We consider the cellular chains functor $C^c(-)$, endowed with the $\mathcal{E}$-coalgebra structure constructed in \cite{BergerFresse04}. In this section, we consider the category of simplicial sets endowed with the Kan--Quillen model structure. 

\begin{Remark}
This coalgebra structure is encoded by the notion of a coalgebra over the Barrett--Eccles \textit{operad}, since it is not in general conilpotent. Coalgebras over cooperads naturally encode \textit{conilpotent} types of coalgebras.
\end{Remark}

\begin{theorem}[\cite{BergerFresse04}]
There is a Quillen adjunction 

\[
\begin{tikzcd}[column sep=7pc,row sep=3pc]
            \mathsf{sSet} \arrow[r, shift left=1.1ex, "C^c(-)"{name=F}] 
            &\mathcal{E}\text{-}\mathsf{coalg}~, \arrow[l, shift left=.75ex, "\overline{\mathrm{R}}"{name=U}]
            \arrow[phantom, from=F, to=U, , "\dashv" rotate=-90]
\end{tikzcd}
\]

where $C^c(-)$ denotes the cellular chain functor endowed with a functorial $\mathcal{E}$-coalgebra structure and where the right adjoint is given by 
\[
\overline{\mathrm{R}}(C)_\bullet \coloneqq \mathrm{Hom}_{\mathcal{E}_\infty\text{-}\mathsf{cog}}(C^c(\Delta^\bullet),C)~.
\]
\end{theorem}

\begin{proof}
The key point is constructing the functorial $\mathcal{E}$-coalgebra structure on $C^c(-)$, which is done in \cite{BergerFresse04}. The operad $\mathcal{E}$ is \textit{coadmissible}, meaning dg $\mathcal{E}$-coalgebras admit a model structure left-transferred from dg modules, see \cite[Proposition 29]{bricevictor}, originally stated in \cite{BergerMoerdijk}. It is immediate to check that this adjunction is a Quillen adjunction. 
\end{proof}

The counit morphism $\epsilon: \Omega \mathrm{B}^{\mathrm{s.a}}\mathcal{E} \qi \mathcal{E}$ induces a Quillen adjunction:

\[
\begin{tikzcd}[column sep=7pc,row sep=3pc]
            \mathcal{E}\text{-}\mathsf{coalg} \arrow[r, shift left=1.1ex, "\mathrm{Res}_{\epsilon}"{name=F}] 
            &u\mathcal{EE}_\infty\text{-}\mathsf{coalg}~. \arrow[l, shift left=.75ex, "\mathrm{Coind}_{\epsilon}"{name=U}]
            \arrow[phantom, from=F, to=U, , "\dashv" rotate=-90]
\end{tikzcd}
\]

By composing both Quillen adjunctions, we get the following Quillen adjunction: 

\[
\begin{tikzcd}[column sep=8pc,row sep=3pc]
            \mathsf{sSet} \arrow[r, shift left=1.1ex, "\mathrm{Res}_{\epsilon} ~\circ ~C^c(-)"{name=F}] 
            &u\mathcal{EE}_\infty\text{-}\mathsf{coalg}~, \arrow[l, shift left=.75ex, "\overline{\mathcal{R}}"{name=U}]
            \arrow[phantom, from=F, to=U, , "\dashv" rotate=-90]
\end{tikzcd}
\]

The right adjoint is given, for a $u\mathcal{EE}_\infty$-coalgebra $C$, by 

\[
\overline{\mathcal{R}}(C)_\bullet \coloneqq \mathrm{Hom}_{u\mathcal{EE}_\infty\text{-}\mathsf{cog}}\left(\mathrm{Res}_{\epsilon}C^c(\Delta^\bullet),C\right) \cong \mathrm{Hom}_{\mathcal{E}\text{-}\mathsf{cog}}\left(C^c(\Delta^\bullet),\mathrm{Coind}_{\epsilon}C \right)~.
\]

\begin{Remark}
We embed the cellular chains functor $C^c(-)$ inside the bigger category of $u\mathcal{EE}_\infty$-coalgebras since their homotopy theory is better behaved than the homotopy theory of $\mathcal{E}$-coalgebras. It is not known to us if the quasi-isomorphism $\epsilon: \Omega \mathrm{B}^{\mathrm{s.a}}\mathcal{E} \qi \mathcal{E}$ induces a Quillen equivalence, since $\mathcal{E}$ is not cofibrant as a dg operad. 
\end{Remark}

\begin{Notation}
Since the inclusion functor $\mathrm{Res}_{\epsilon}$ do not change the underlying dg module nor the coalgebraic structure of the cellular chain functor $C^c(-)$, we will denote the composition $\mathrm{Res}_{\epsilon}C^c(-)$ simply by $C^c(-)$ from now on.
\end{Notation}

\begin{theorem}\label{thm: triangle commutatif}
There is a commuting triangle 

\[
\begin{tikzcd}[column sep=5pc,row sep=2.5pc]
&\hspace{1pc}u\mathcal{EE}_\infty\text{-}\mathsf{coalg} \arrow[dd, shift left=1.1ex, "\widehat{\Omega}_{\iota}"{name=F}] \arrow[ld, shift left=.75ex, "\overline{\mathcal{R}}"{name=C}]\\
\mathsf{sSet}  \arrow[ru, shift left=1.5ex, "C^c_{*}(-)"{name=A}]  \arrow[rd, shift left=1ex, "\mathcal{L}"{name=B}] \arrow[phantom, from=A, to=C, , "\dashv" rotate=-70]
& \\
&\hspace{3pc}\mathsf{curv}~\mathsf{abs}~\mathcal{L}^\pi_\infty\textsf{-}\mathsf{alg}^{\mathsf{qp}\text{-}\mathsf{comp}} ~, \arrow[uu, shift left=.75ex, "\widehat{\mathrm{B}}_{\iota}"{name=U}] \arrow[lu, shift left=.75ex, "\mathcal{R}"{name=D}] \arrow[phantom, from=B, to=D, , "\dashv" rotate=-110] \arrow[phantom, from=F, to=U, , "\dashv" rotate=-180]
\end{tikzcd}
\]

made of Quillen adjunctions, where $\mathcal{L}$ is given by the composition of the left adjoint functors and $\mathcal{R}$ by the composition of the right adjoint functors.
\end{theorem}

\begin{proof}
Follows directly from the above, together with Theorem \ref{thm: Equivalence the Quillen CC infini et absolues}.
\end{proof}

\subsection{The integration functor} The define \textit{the integration functor} to the right adjoint functor $\mathcal{R}$ in the triangle above. This definition is analogous to the one given in \cite[Section 2.1]{lucio2022integration}. The next result also follows from the same model categorical arguments as in \textit{op.cit.}

\begin{theorem}\label{thm: propriétés de l'intégration}\leavevmode
\begin{enumerate}
\item For any qp-complete curved absolute $\mathcal{L}^\pi_\infty$-algebra $\mathfrak{g}$, the simplicial set $\mathcal{R}(\mathfrak{g})$ is a Kan complex.

\medskip

\item Let $f: \mathfrak{g} \twoheadrightarrow \mathfrak{h}$ be a degree-wise epimorphism of qp-complete curved absolute $\mathcal{L}^\pi_\infty$-algebras. Then 
\[
\mathcal{R}(f): \mathcal{R}(\mathfrak{g}) \twoheadrightarrow \mathcal{R}(\mathfrak{h})
\]
is a fibration of simplicial sets. 

\medskip

\item The functor $\mathcal{R}$ preserves weak equivalences. In particular, it sends any filtered quasi-isomorphism $f: \mathfrak{g} \qi \mathfrak{h}$ of qp-complete curved absolute $\mathcal{L}^\pi_\infty$-algebras to a weak homotopy equivalence of simplicial sets.
\end{enumerate}
\end{theorem}

\begin{proof}
The functor $\mathcal{R}$ is a right Quillen functor. 
\end{proof}

\begin{Remark}
The above theorem establishes the main properties that an integration functor should satisfy, where the first two points are part of the main results of \cite{Getzler09} and the last one can be considered a Goldman--Millson type of theorem. See also \cite[Remark 2.13]{lucio2022integration}. 
\end{Remark}

\begin{Corollary}
The integration functor commutes with limits of curved absolute $\mathcal{L}^\pi_\infty$-algebras and homotopy limits of qp-complete curved absolute $\mathcal{L}^\pi_\infty$-algebras.
\end{Corollary}

\begin{proof}
The functor $\mathcal{R}$ is a right Quillen functor, and since every qp-complete curved absolute $\mathcal{L}^\pi_\infty$-algebra is fibrant, it is also equal to its derived functor. 
\end{proof}

\begin{Remark}
Let $\mathfrak{g}$ be a qp-complete curved absolute partition $\mathcal{L}^\pi_\infty$-algebra. There is an isomorphism 
\[
\mathcal{R}(\mathfrak{g})_\bullet \cong \lim_{\delta} \mathcal{R}(\mathfrak{g}/\mathrm{W}_\delta \mathfrak{g})_\bullet~.
\]
of Kan complexes. Thus $\mathcal{R}(\mathfrak{g})$ is the (homotopy) limit of a tower of simplicial sets obtained by integrating each step of the canonical qp-filtration.
\end{Remark}

\textbf{Non-abelian Dold-Kan correspondence.} The adjunction 
\[
\begin{tikzcd}[column sep=5pc,row sep=2.5pc]
\mathsf{sSet}  \arrow[r, shift left=1ex, "\mathcal{L}"{name=B}] 
&\mathsf{curv}~\mathsf{abs}~\mathcal{L}^\pi_\infty\textsf{-}\mathsf{alg}^{\mathsf{qp}\text{-}\mathsf{comp}} \arrow[l, shift left=.75ex, "\mathcal{R}"{name=D}] \arrow[phantom, from=B, to=D, , "\dashv" rotate=-90] 
\end{tikzcd}
\]
is a \textit{generalization of the Dold-Kan correspondence} over a field. Notice that a dg module $(V,d_V)$ is a particular example of a curved absolute $\mathcal{L}^\pi_\infty$-algebra, where the structural morphism 
\[
\gamma_V: \displaystyle \prod_{n \geq 0} \widehat{\Omega}^{\mathrm{s.a}}\mathcal{E}^*(n) ~\widehat{\otimes}_{\mathbb{S}_n} ~ V^{\otimes n} \longrightarrow V
\]
is the \textit{zero morphism} on any non-trivial symmetric corked rooted tree. Thus any chain complex admits an \textit{abelian} (trivial) curved absolute $\mathcal{L}^\pi_\infty$-algebra structure, which is furthermore qp-complete. 

\begin{Proposition}
Let $(V,d_V)$ be a dg module endowed with a trivial curved absolute $\mathcal{L}^\pi_\infty$-algebra structure. There is an isomorphism
\[
\mathcal{R}(V) \cong \Gamma(V)~,
\]
where $\Gamma(-)$ is the Dold--Kan functor. 
\end{Proposition}

\begin{proof}
Since the structure map $\gamma_V$ is trivial, the complete bar construction $\widehat{\mathrm{B}}_{\iota}V$ on $V$ is equal to the cofree construction $\mathscr{C}(u\mathcal{EE}_\infty)(V)$. Therefore we have the following isomorphisms
\[
\mathcal{R}(V) \cong \mathrm{Hom}_{u\mathcal{EE}_\infty\text{-}\mathsf{cog}}\left(C^c(\Delta^\bullet), \mathscr{C}(u\mathcal{EE}_\infty)(V) \right) \cong \mathrm{Hom}_{\mathsf{dg}~\mathsf{mod}}\left(C^c(\Delta^\bullet), V \right)~.
\]
\end{proof}

\textbf{A pointed version of the integration functor.} The reduced chain functor $\tilde{C}_*^c(-)$ on pointed simplicial sets
has a natural $\mathcal{E}^{\mathrm{nu}}$-coalgebra structure, which in turn gives a $\Omega \mathrm{B}\mathcal{E}^{\mathrm{nu}}$-coalgebra structure by pullback. By composing this adjunction with the complete bar-cobar adjunction between $\Omega \mathrm{B}\mathcal{E}^{\mathrm{nu}}$-coalgebras and absolute partition $\mathcal{L}_\infty$-algebras, we get a adjunction 
\[
\begin{tikzcd}[column sep=5pc,row sep=2.5pc]
\mathsf{sSet}_{*}  \arrow[r, shift left=1ex, "\mathcal{L}_*"{name=B}] 
&\mathsf{abs}~\mathcal{L}^\pi_\infty\textsf{-}\mathsf{alg}^{\mathsf{qp}\text{-}\mathsf{comp}}~, \arrow[l, shift left=.75ex, "\mathcal{R}_*"{name=D}] \arrow[phantom, from=B, to=D, , "\dashv" rotate=-90] 
\end{tikzcd}
\]
which fits in a commutative triangle of Quillen adjunctions like its unpointed analogue in Theorem \ref{thm: triangle commutatif}. Let us compare these two adjunctions.

\begin{Proposition}\label{Prop: commutativity of the pointing adjunctions}
The following square of Quillen adjunctions
\[
\begin{tikzcd}[column sep=4pc,row sep=4pc]
\mathsf{sSet}_* \arrow[r,"\mathcal{L}_*"{name=B},shift left=1.1ex] \arrow[d,"\mathrm{U}"{name=SD},shift left=1.1ex ]
&\mathsf{abs}~\mathcal{L}^\pi_\infty\text{-}\mathsf{alg}^{\mathsf{qp}\text{-}\mathsf{comp}} \arrow[d,"\mathrm{U}"{name=LDC},shift left=1.1ex ] \arrow[l,"\mathcal{R}_*"{name=C},,shift left=1.1ex]  \\
\mathsf{sSet} \arrow[r,"\mathcal{L} "{name=CC},shift left=1.1ex]  \arrow[u,"(-) \sqcup \{*\}"{name=LD},shift left=1.1ex ]
&\mathsf{curv}~\mathsf{abs}~\mathcal{L}^\pi_\infty\text{-}\mathsf{alg}^{\mathsf{qp}\text{-}\mathsf{comp}}~, \arrow[l,"\mathcal{R}"{name=CB},shift left=1.1ex] \arrow[u,"(-)_*"{name=TD},shift left=1.1ex] \arrow[phantom, from=SD, to=LD, , "\dashv" rotate=0] \arrow[phantom, from=C, to=B, , "\dashv" rotate=-90]\arrow[phantom, from=TD, to=LDC, , "\dashv" rotate=0] \arrow[phantom, from=CC, to=CB, , "\dashv" rotate=-90]
\end{tikzcd}
\] 

is commutative in the following sense: the left adjoints from the bottom-left to the top-right corner are naturally isomorphic.
\end{Proposition}

\begin{proof}
The proof is completely analogous to \cite[Proposition 3.34]{lucio2022integration}. See also \cite[Subsection 3.4]{lucio2022integration} for more details on how to compare counital and non-counital coalgebras. 
\end{proof}

\subsection{Maurer--Cartan elements} The $0$-simplices of $\mathcal{R}(\mathfrak{g})$ exactly correspond to solutions of the Maurer--Cartan equation in $\mathfrak{g}$. 

\begin{Proposition}
Let $\mathfrak{g}$ be a curved absolute $\mathcal{L}^\pi_\infty$-algebra. There is a bijection
\[
\mathcal{R}(\mathfrak{g})_0 \cong \mathcal{MC}(\mathfrak{g})~,
\]
between the $0$-simplices of $\mathcal{R}(\mathfrak{g})$ and the Maurer--Cartan elements of $\mathfrak{g}$.
\end{Proposition}

\begin{proof}
The $0$-simplices of $\mathcal{R}(\mathfrak{g})$ are determined by the $\mathcal{E}$-coalgebra structure of $C_*^c(\{*\}) \cong \kk$. One can check that this coalgebra structure is given by maps 
\[
\left\{
\begin{tikzcd}[column sep=4pc,row sep=-0.25pc]
\kk.a_0 \arrow[r]
&\displaystyle \prod_{n \geq 0} \mathrm{Hom}_{\mathbb{S}_n}\left(\mathcal{E}(n), (\kk.a_0)^{\otimes n}\right) \\
a_0 \arrow[r,mapsto]
&\displaystyle \sum_{n \geq 0,~\sigma_0 \in \mathbb{S}_n~ n \neq 1} \left[\mathrm{ev}_{a_0}: c_n^{\sigma_0} \mapsto a_0 \otimes \cdots \otimes a_0 \right]~,
\end{tikzcd}
\right.
\]
where the other decomposition maps are zero for degree reasons. Notice that 

\[
\mathbb{N}(c_n^{\mathrm{id}}(a_0,\cdots,a_0)) = \sum_{\sigma_0 \in \mathbb{S}_n} \left[\mathrm{ev}_{a_0}: c_n^{\sigma_0} \mapsto a_0 \otimes \cdots \otimes a_0 \right]
\]

under the norm map isomorphism 

\[
\mathbb{N}: \widehat{\Omega}^{\mathrm{s.a}}\mathcal{E}^*(n) ~\widehat{\otimes}_{\mathbb{S}_n}~ \mathfrak{g}^{\otimes n} \longrightarrow \mathrm{Hom}_{\mathbb{S}_n}\left(\mathrm{B}^{\mathrm{s.a}}\mathcal{E}(n), \mathfrak{g}^{\otimes n}\right)~.
\]
Thus 

\[
\widehat{\Omega}_{\iota}(\kk.a_0) \cong \left(\displaystyle \prod_{n \geq 0} \widehat{\Omega}^{\mathrm{s.a}}\mathcal{E}^*(n) ~ \widehat{\otimes}_{\mathbb{S}_n} ~ (\kk.a_0)^{\otimes n}, d_{\mathrm{cobar}}(a_0) = - \sum_{n \geq 0~,~ n \neq 1} c_n^{\mathrm{id}}(a_0,\cdots,a_0) \right)~.
\]

Therefore the data of a morphism of curved absolute $\mathcal{L}^\pi_\infty$-algebras $\widehat{\Omega}_{\iota}(\kk.a_0) \longrightarrow \mathfrak{g}$ is equivalent to the data of an element $\alpha$ in $\mathfrak{g}$ of degree $0$ such that 
\[
d_\mathfrak{g}(\alpha) = - \gamma_\mathfrak{g}\left(\sum_{n \geq 0, n \neq 1} c_n^{\mathrm{id}}(\alpha,\cdots,\alpha)\right)~.
\]
\end{proof}

\subsection{Gauge equivalences and higher homotopy groups} We compute paths in $\mathcal{R}(\mathfrak{g})$ and give a combinatorial description of them. These paths are \textit{gauge equivalences}, as in \cite[Section 2.4]{lucio2022integration}. This allows us to give a full combinatorial description of the set $\pi_0(\mathcal{R}(\mathfrak{g}))$. From the point of view of deformation theory, if $\mathfrak{g}$ encodes a spectral deformation problem, computing Maurer--Cartan elements up to gauge equivalences gives the \textit{classical} deformation problem associated to $\mathfrak{g}$. Then, we compute the higher homotopy groups of the integration functor $\mathcal{R}(\mathfrak{g})$ for any qp-complete absolute partition $\mathcal{L}_\infty^\pi$-algebra, at the $0$ Maurer--Cartan element. The group $\pi_k(\mathcal{R}(\mathfrak{g}), 0)$  is in bijection with elements in $\mathfrak{g}$ of degree $k$ which satisfy an algebraic equation, up to an equivalence relation imposed by elements in degree $k+1$, in terms of a similar algebraic equation. These equations are determined by the dg $\mathcal{E}$-coalgebra structure on the cellular chains on the interval and on the spheres. 

\medskip

\textbf{The coalgebra structure on the interval.} We use the general formulas given in \cite[Section 2.2]{BergerFresse04} for the $\mathcal{E}$-coalgebra structure of $C_*^c(\Delta^k)$ to effectively compute the structural map of $C_*^c(\Delta^1)$. Since $C_*^c(\Delta^1)$ is a canonical interval object, computing this structure will allow us to compute homotopies in the category of $u\mathcal{EE}$-coalgebras. 

\medskip

\textit{The function $\varphi_I^n$.} Let $n \geq 2$ and let $I \subseteq \{1,\cdots,n\}$ be an ordered subset, where we denote the elements by $I = \{i_1, \cdots,i_k\}$. Let $\sigma$ be a permutation in $\mathbb{S}_n$. We say that $\sigma$ is an $I$-\textit{-unshuffle} if 
\[
\sigma^{-1}(i_1) < \cdots < \sigma^{-1}(i_k)~. 
\]

We define the map $\varphi_I^n: \mathbb{S}_n \longrightarrow \{0,1\}$ by sending $\sigma$ to $1$ if it is an $I$-unshuffle and to $0$ if it is not. 

\medskip

\textit{The function $\pi^{(n,S)}_I$.} Let $n \geq 2$ and let $I,S \subseteq \{1,\cdots,n\}$ be two ordered subsets, where $S =\{o_1,\cdots, o_s\}$ is of cardinality $s \geq 1$. Let $\underline{\sigma} = (\sigma_0,\cdots, \sigma_{s-1})$ be an $s$-tuple of permutations of $\mathbb{S}_n$. We say that $\underline{\sigma}$ is an $S$\textit{-ordered partition of }$I$ if it satisfies the following conditions:
\begin{enumerate}
\item For all $i$ such that $1 \leq i \leq s$, we have that
\[
\sigma_{i-1}(\{1,\cdots, \sigma_{i-1}^{-1}(o_i) -1\}) \subseteq I~.
\]
The set $\{1,\cdots, \sigma_{i-1}^{-1}(o_i) -1\}$ is considered to be empty if $\sigma_{i-1}^{-1}(o_i) = 1$, in which case the condition is trivially satisfied. 

\item Let us denote by $\mathrm{Im}_I(\sigma_{i-1})$ the image of $\{1,\cdots, \sigma_{i-1}^{-1}(o_i)- 1\}$ by $\sigma_{i-1}$ inside $I$. Then the collections of sets defined by 
\[
I_i \coloneqq \mathrm{Im}_I(\sigma_{i-1})-\left(\bigcup_{j=0}^{i-2}\mathrm{Im}_I(\sigma_{i-1})\right)~. 
\]
forms an ordered partition $I_1 \sqcup \cdots \sqcup I_s$ of $I$.
\end{enumerate}

We define the map $\pi_I^{(n,S)}: \mathbb{S}_n^{\times s} \longrightarrow \{0,1\}$ by sending $\underline{\sigma}$ to $1$ if it is an $S$-ordered partition of $I$ and to $0$ if it is not. 

\begin{Example}
Let $\underline{\sigma}$ be in $\mathbb{S}_n^{\times s}$, and let $S = \{1,\cdots,s\}$. It is an $S$-ordered partition of $\emptyset$ if and only if the tuple $(\sigma_0(1), \cdots, \sigma_{s-1}(1))$ is the identity permutation $(1, \cdots, s)$ of $\mathbb{S}_s$. 
\end{Example}

\begin{Proposition}\label{Prop: coalgebra structure of the interval}
The $\mathcal{E}$-coalgebra structure on the interval $C_*^c(\Delta^1) \cong \kk.a_0 \oplus \kk.a_1 \oplus \kk.a_{01}$ is given by the following structure map:

\[
\hspace{-1.5pc}
\left\{
\begin{tikzcd}[column sep=0.25pc,row sep=0pc]
C_*^c(\Delta^1) \arrow[r]
&\displaystyle \prod_{n \geq 0} \mathrm{Hom}_{\mathbb{S}_n}\left(\mathcal{E}(n), (C_*^c(\Delta^1))^{\otimes n}\right) \\
a_i \arrow[r,mapsto]
&\displaystyle \sum_{n \geq 0,~\sigma_0 \in \mathbb{S}_n~ n \neq 1} \left[\mathrm{ev}_{a_i}: c_n^{\sigma_0} \mapsto a_i \otimes \cdots \otimes a_i \right]~~~~\text{for}~~~~ i=0,1~, \\
a_{01} \arrow[r,mapsto]
&\displaystyle \sum_{\substack{n \geq 2 \\ s + a + b =n}}\sum_{\substack{w \in \mathcal{E}(n)_{s-1}  \\ \sigma \in \mathbb{S}_n}} \alpha(w).\left[c_n^{\sigma.w} \mapsto ~ \sigma^{-1} \bullet \left(\underbrace{a_{01}\otimes \cdots \otimes a_{01}}_{s}\otimes \underbrace{a_{0} \otimes \cdots \otimes a_{0}}_{a} \otimes \underbrace{a_{1} \otimes \cdots \otimes a_{1}}_{b}\right) \right],
\end{tikzcd}
\right.
\]

where the coefficient $\alpha(w)$ is given by the product $\pi_{I_a}^{(n,S)}(w).\varphi_{I_b}^n(\sigma_{s-1})$, for $I_a = \{s+1,\cdots,s+a\}$, $I_b = \{s+a+1, \cdots, s+a+b\}$ and $S = \{1,\cdots,s\}$. 
\end{Proposition}

\begin{proof}
We want to compute which operations in the Barratt-Eccles operad give non-trivial decompositions of $a_{01}$; the case of $a_0$ and $a_1$ are straightforward. Since the $\mathcal{E}$-coalgebra structure is the pullback of the $\mathcal{S}urj$-coalgebra structure, let us first compute which elements in the surjections operad give such decompositions. 

\medskip

Let us fix an arity $n \geq 2$. We suppose that $u$ in $\mathcal{S}urj(n)$ produces a nont-trivial decomposition of $a_{01}$ involving $s$-copies of $a_{01}$, $a$-copies of $a_0$ and $b$-copies of $a_1$, where $s+a+b =n$. Then $u$ is of degree $s-1$. Furthermore, let us fix the following order on the decomposition
\[
a_{01}^{\otimes s} \otimes a_{0}^{\otimes a} \otimes a_{1}^{\otimes b} = \underbrace{a_{01}}_{1} \otimes \cdots \otimes \underbrace{a_{01}}_{s} \otimes \underbrace{a_{0}}_{s+1} \otimes \cdots \otimes \underbrace{a_{0}}_{s+a} \otimes \underbrace{a_{1}}_{s+a+1} \otimes \cdots \otimes \underbrace{a_{1}}_{s+a+b},
\]
with respect to the interval cut operation of $u$. For instance, this means that the first term $a_{01}$ comes from the two occurrences of $1$ in $u$ (one in the first $(s+a-1)$-terms and one in the last $(s+b-1)$-terms), or that the $(s+a+1)$-th term $a_{1}$ comes from the single appearance of $s+a+1$ in $u$, at least in the $(s+a+1)$-th position. Therefore $u$ must have a table arrangement of the form 
\[
\left|~~ 
\begin{aligned}
&s+1~\cdots~s+i_1~~ 1 \\
&s+i_1+1 ~ \cdots ~ s+i_2+i_1~~ 2 \\
&\vdots \\
&s+i_{s-2}+ \cdots +1 ~ \cdots ~ s+i_{s—1}+\cdots +i_1~~ s-1 \\
&s+i_{s-1}+ \cdots +1 ~ \cdots ~ s+i_s+\cdots +i_1 ~~ s ~~ u(s+a+1)~ \cdots ~ u(2s+a+b-1)~,
\end{aligned}
\right.
\]
where $(i_1,\cdots,i_s)$ are integers $\geq 0$ such that $i_1+\cdots+i_s=a$. If $i_j$ is zero, then $j$ is the first term in the $j$-th row of the table arrangement of $u$. Equivalently, the numbers $(i_1,\cdots,i_s)$ determine an ordered partition of $I_a = \{s+1, \cdots, s+a\}$ via $I_j = \{s+i_{j-1}+\cdots+i_1, \cdots, s+i_{j}+ \cdots +i_1\}$ for $j \geq 2$ and $I_1= \{s+1, \cdots, s+i_1\}$. The set $I_j$ is empty whenever $i_j =0$. And the table arrangement of $u$ starts with at row $j$ with $I_j$ and ends with $j$, except in the last row. This last row contains the rest of the terms of the surjection $u$; since we have also ordered the terms $a_{1}$ in the decomposition, the set $I_b= \{s+a+1,\cdots,s+a+b\}$ must appear ordered, possibly with other terms in between, in the last row. Finally, notice that such a decomposition involves no signs, as both the position and the permutation sign associated to a surjection $u$ with this type of table arrangement are trivial. 

\medskip

Now, consider a element $w = (\sigma_0, \cdots, \sigma_{s-1})$ in $\mathcal{E}(n)_{s-1}$. A surjection $u$ of the type described above appears in the image of the table reduction of $w$ if and only if the $w$ is a $I_a$-ordered partition and if $\sigma_{s-1}$ is a $I_b$-unshuffle. Finally, any decomposition involving $s$-copies of $a_{01}$, $a$-copies of $a_0$ and $b$-copies of $a_1$ can be obtained via a permutation in $\mathbb{S}_n$ of the above decomposition, hence the action of $\mathbb{S}_n$ on the $w$ in $\mathcal{E}(n)_{s-1}$ generates all the elements that produce these decompositions.
\end{proof}

\begin{Remark}
When $I_a = \emptyset$, that is, we consider only decompositions of $a_{01}$ which involve $a_{01}$ and $a_{1}$, we recover the formulas of \cite[Theorem 3.2.4]{BergerFresse04}. The sign $(-1)^\sigma$, where $\sigma = s(s-1)/2$, does not appear in the coalgebra structure since it comes from the duality pairing, see the paragraph above \cite[Proof of Lemma 3.3.3]{BergerFresse04}. And the sign $\epsilon_r(w)$ is precisely the sign that appears when a permutation $\sigma$ in $\mathbb{S}_n$ acts on the ordered decomposition $a_{01}^{\otimes s} \otimes a_{0}^{\otimes a} \otimes a_{1}^{\otimes b}$.
\end{Remark}

\textbf{Gauge equivalences.} We give a definition of gauge equivalences, which are paths in the $\infty$-groupoid associated to a curved absolute partition $\mathcal{L}^\pi_\infty$-algebra. Recall that in the characteristic zero setting, defining gauge equivalences in this way recovers the other definitions present in the literature, see \cite[Section 2.4]{lucio2022integration}.

\begin{Definition}[Gauge equivalences]
Let $\mathfrak{g}$ be a qp-complete curved absolute $\mathcal{L}^\pi_\infty$-algebra, and let $\alpha,\beta$ be two Maurer--Cartan elements. The element $\alpha$ is \textit{gauge equivalent} to $\beta$ if there exists a degree $1$ element $\lambda$ in $\mathfrak{g}$ such that

\[
\hspace{-1.5pc}
d_{\mathfrak{g}}(\lambda) = \beta - \alpha - \gamma_{\mathfrak{g}}\left(\sum_{\substack{n \geq 2 \\ s + a + b =n}} \sum_{\substack{w \in \mathcal{E}(n)_{s-1} \\ w = (\sigma_0,\cdots,\sigma_{s-1})}}\pi_{I_a}^{(n,S)}(w).\varphi_{I_b}^n(\sigma_{s-1}) c_n^{w}\left(\underbrace{\lambda, \cdots, \lambda}_{s},\underbrace{\alpha, \cdots, \alpha}_{a}, \underbrace{\beta, \cdots, \beta}_{b}\right)  \right)~.
\]

for $I_a = \{s+1,\cdots,s+a\}$, $I_b = \{s+a+1, \cdots, s+a+b\}$ and $S = \{1,\cdots,s\}$. We denote this by $\alpha \sim_\lambda \beta$. 
\end{Definition}

Since $\mathfrak{g}$ is qp-complete, the above formula can be split in several components according to the weight of the operations. 

\medskip

\textbf{Gauge actions.} The above formula can also be rewritten into the following fixed-point equation in terms of the element $\beta$:
\[
\beta = d_{\mathfrak{g}}(\lambda) + \alpha + \gamma_{\mathfrak{g}}\left(\sum_{\substack{n \geq 2 \\ s + a + b =n}} \sum_{\substack{w \in \mathcal{E}(n)_{s-1} \\ w = (\sigma_0,\cdots,\sigma_{s-1})}}\pi_{I_a}^{(n,s)}(w).\varphi_{I_b}^n(\sigma_{s-1}) c_n^{w}\left(\underbrace{\lambda, \cdots, \lambda}_{s},\underbrace{\alpha, \cdots, \alpha}_{a}, \underbrace{\beta, \cdots, \beta}_{b}\right)  \right)~.
\]
If $\alpha$ and $\lambda$ are in weight $1$ in $\mathfrak{g}$, this fixed-point equation admits an unique solution, using \cite[Appendix A]{robertnicoud2020higher}. See also Proposition \ref{prop: fixed point equation} (with different weight conventions) and the computations done in Subsection \ref{subsection: bch}. Let us briefly describe this solution.

\medskip

Let $\tau$ be a symmetric corked \textit{planar} tree. It is \textit{right-handed} if for every vertex $v$ with $\mathrm{In}(v) \geq 2$ incoming edges, the $j$-th incoming edge is linked to another edge only if all the previous $l$ incoming edges (reading from right to left) are linked to an edge, for all $0 \leq l \leq j$ and every $j \in [1,\mathrm{in}(v)]$. Furthermore, every vertex which is not a cork must have at least one \textit{leaf}. 

\medskip

A labelling of such a $\tau$ by $01$ and $0$, subsets of $[0,1]$, is \textit{licit} if for every vertex $v$ of $\tau$ which is not a cork, going from right to left of its incoming edges, $v$ has $b$ incoming edges attached to other edges, $a$ leaves labelled by $0$ and then $c$ leaves labelled by $01$, where $a + b + c = \mathrm{In}(v)$ and $a,b \geq 0$ and $c \geq 1$. In other words, a label $0$ cannot appear after a label $01$ (from right to left) and since $v$ has at least one leaf, at least one leaf (the utmost left one) is labelled by $01$. We denote the set of licit labels by $L^{01,0}(\tau)$. 

\medskip

We define a coefficient $\alpha^{(\tau;~01,\cdots,1)}$ for every right-handed symmetric corked planar tree $\tau$ together with a licit labelling by $01,0$ as follows: 
\[
\alpha^{(\tau;~ 01,\cdots,1)} = \prod_{v \in V(\tau)}\alpha^{(\underline{\sigma};~ 01, \cdots, 1)}~,
\]
where $\underline{\sigma}_v$ is the label of the vertex $v$, and where the product runs over all vertices of $\tau$ which are not corks. The coefficients at each vertex $v$ are given as follows: suppose the vertex $v$ has $b$ incoming edges attached to other edges, $c$ leaves labelled by $01$ and then $a$ leaves labelled by $0$, where $a + b + c = n$ is the total number of incoming edges, and where $a,b \geq 0$ and $c \geq 1$. Then
\[
\alpha^{(\underline{\sigma};~ 01, \cdots, 1)} = \begin{cases}
    \pi_{I_a}^{(n,I_c)}(\underline{\sigma}).\varphi_{I_b}^n(\sigma_{c-1}) & \quad \text{if } \mathrm{degree}(\underline{\sigma}) = c-1~, \\
    0 & \quad \text{otherwise}~,
\end{cases}
\]

where $\underline{\sigma} = (\sigma_0, \cdots \sigma_{c-1})$ and for $I_a = \{c+1,\cdots,s+a\}$, $I_b = \{c+a+1, \cdots, c+a+b\}$ and $I_c = \{1,\cdots,c\}$.

\medskip

Let $\mathfrak{g}$ be a qp-complete curved absolute $\mathcal{L}^\pi_\infty$-algebra, and let $\alpha$ be a Maurer--Cartan element in weight $1$ and let $\lambda$ be a degree $1$ element in $\mathfrak{g}$ of weight $1$. The \textit{gauge-action} of $\lambda$ on $\alpha$ is given by 
\[
\lambda \bullet \alpha = \sum_{m \geq 0} \sum_{\tau \in \mathrm{SCPT}^{\mathrm{right}}_m} \sum_{L^{01,0}(\tau)} \alpha^{(\tau;~ 01, \cdots, 1)} \gamma_\mathfrak{g}\left(\tau\left(\lambda ,\cdots, \alpha, \cdots, \lambda, \cdots, \alpha; d_\mathfrak{g}(\lambda) + \alpha \right)\right) ~, 
\]
\vspace{0.1pc}

where $\lambda$ appears if a leaf is labelled by $01$ and $\alpha$ by $0$, where the corks are replaced by the input $d_\mathfrak{g}(\lambda) + \alpha$ and where the sum is taken over all right-handed symmetric corked planar trees and all licit labellings of those. It is interesting to notice that the solution starts as follows
\[
\beta = \alpha + d_{\mathfrak{g}}(\lambda) + \sum_{n \geq 2} l_n^{\mathrm{id}}(\alpha, \cdots, \alpha,\lambda)  + \cdots~, 
\]
where the weight $1$ (in terms of structural operations) term is extremely similar to the \textit{twisted differential} $d_\mathfrak{g}^\alpha(\alpha)$ that appears in the same equation in the $\mathcal{L}_\infty$-algebra case, but without the coefficient $1/(n-1)!$. 

\begin{theorem}
Let $\mathfrak{g}$ be a qp-complete curved absolute $\mathcal{L}^\pi_\infty$-algebra. There is a canonical natural bijection
\[
\pi_0(\mathcal{R}(\mathfrak{g})) \cong \mathcal{MC}(\mathfrak{g})/\sim_{\mathrm{gauge}\text{ }\mathrm{equiv}}~.
\]

\end{theorem}

\begin{proof}
There are canonical bijections
\[
\pi_0(\mathcal{R}(\mathfrak{g})) \cong \mathrm{Hom}_{\mathsf{sSets}}(\{*\}, \mathcal{R}(\mathfrak{g}))/\sim_{\mathrm{hmt}} \cong \mathrm{Hom}_{u\mathcal{EE}_\infty\text{-}\mathsf{coalg}}\left(C_*^c(\{*\}), \widehat{\mathrm{B}}_\iota(\mathfrak{g}) \right)/\sim_{\mathrm{hmt}}~,
\]
since $C_*^c \dashv \overline{\mathcal{R}}$ is a Quillen adjuntion. In order to compute homotopies in $u\mathcal{EE}_\infty$-coalgebras, we need an interval object. This interval object is provided by tensoring with the interval object $C_*^c(\Delta^1)$. Notice that the $u\mathcal{EE}_\infty$-coalgebra structure of $C_*^c(\Delta^1)$ is the pullback of its $\mathcal{E}$-coalgebra structure by the projection, which is computed in Proposition \ref{Prop: coalgebra structure of the interval}. 

\medskip

Finally, the data of a morphism $h: C_*^c(\Delta^1) \longrightarrow \widehat{\mathrm{B}}_\iota(\mathfrak{g})$ is equivalent to the data of a curved twisting morphism $\iota_h: C_*^c(\Delta^1) \longrightarrow \mathfrak{g}$. In turn, this data coincides precisely with the definition of a gauge equivalence by Proposition \ref{Prop: coalgebra structure of the interval}, once we pullback the formulas in the coalgebra structure along the norm isomorphism.
\end{proof}

\begin{Remark}
It follows that gauge equivalence is an equivalence relation on the set of Maurer--Cartan elements. This is not immediate from the definition. For instance, it is not obvious \textit{a priori} that gauge equivalences are reflexive, as the formula in Proposition \ref{Prop: coalgebra structure of the interval} is not symmetric in $a_0$ and $a_1$. 
\end{Remark}

\textbf{The coalgebra structure on the pointed spheres.} Let $S^k$ be the simplicial model for the $k$-sphere given by $\Delta^k/\partial \Delta^k$. We consider the reduced chains on $\tilde{C}_*^c(S^k)$, which is a coalgebra over the non-unital Barratt--Eccles operad.

\begin{lemma}\label{lemma: coalgebra structure on the reduced spheres}
  Let $k \geq 2$. The $\mathcal{E}^{\mathrm{nu}}$-coalgebra structure on the pointed $k$-sphere $\tilde{C}_*^c(S^k) \cong \kk.a_{k}$ is given by the following structure map
\[
\left\{
\begin{tikzcd}[column sep=1pc,row sep=0pc]
\tilde{C}_*^c(S^k) \arrow[r]
&\displaystyle \prod_{n \geq 0} \Bigl[\E(n), \bigl(\tilde{C}_*^c(S^k)\bigr)^{\otimes n}\Bigr]^{\Sym_n} \\
a_{k} \arrow[r,mapsto]
&\displaystyle \sum_{n \geq 2} \sum_{\substack{w \in \E(n)_{k(n-1)}                             \\ \overline{w}_1 = \mathrm{id}_{\Sym_n}, \, \overline{w}_j \in \Sym_n}} \sum_{\sigma \in \Sym_n} \biggl[c_n^{\sigma.w} \mapsto (-1)^{j_k} \mathrm{sign}(\sigma) \prod_{j = 2}^k \mathrm{sign}(\overline{w}_j) \underbrace{a_{k} \otimes \cdots \otimes a_{k}}_{n} \biggr]~,
\end{tikzcd}
\right.
\]  
  where
  \begin{enumerate}
    \item in the second sum, we consider elements $w = (\sigma_0,\cdots, \sigma_{k(n-1)})$ in $\E(n)_{k(n-1)}$ such that, if we define the $n$-tuple
    \[
    \overline{w}_j \coloneqq (\sigma_{(j-1)(n-1)}(1),\cdots, \sigma_{j(n-1)}(1))
    \]
    for all $1 \leq j \leq k$, we have that $\overline{w}_1$ is the identity permutation of $\Sym_n$ and $\overline{w}_j$ is a permutation in $\Sym_n$ for all $2 \leq j \leq k$ (meaning all the numbers in the $n$-tuple are distinct);
    
    \medskip
    
    \item the sign of the decomposition is given by $(-1)^{j_k}$ where
    \[
    j_k(n) = \left[\frac{k(k-1)}{2}\right] \frac{(n+2)(n-1)}{2},
    \]
    where $\prod_{j = 2}^k \mathrm{sign}(\overline{w}_j)$ is the product of the signatures of the permutations $\overline{w}_j$ for all $2 \leq j \leq k$.
  \end{enumerate}
\end{lemma}

\begin{proof}
It essentially follows from the isomorphism $\tilde{C}_*^c(S^k) \cong \tilde{C}_*^c(S^1)^{\otimes k}$ in \cite[Proposition 3.2.5]{BergerFresse04}, together with formula of the Hopf structure of the operad $\mathcal{E}^{\mathrm{nu}}$, which is given by: 
\[
\Delta_{\mathcal{E}^{\mathrm{nu}}}((\sigma_0, \cdots, \sigma_l)) = \sum_{i=0}^l (\sigma_0, \ldots, \sigma_i) \otimes (\sigma_i, \ldots, \sigma_l)~,
\]
for any element $(\sigma_0, \cdots, \sigma_l)$ in $\mathcal{E}^{\mathrm{nu}}(n)_l$, and any $n \geq 1$. Indeed, by the isomorphism $\tilde{C}_*^c(S^k) \cong \tilde{C}_*^c(S^1)^{\otimes k}$ it is clear that $w = (\sigma_0,\cdots, \sigma_{k(n-1)})$ acts non-trivially on $a_k$ if and only if $w_{(j)}$ acts non trivially on the $j$-th copy of $a_1$, where 
\[
\Delta_{\mathcal{E}^{\mathrm{nu}}}^k(w) = \sum w_{(1)} \otimes \cdots \otimes \underbrace{w_{(j)}}_{j\text{-th}} \otimes \cdots \otimes w_{(k)}~. 
\]
So $w_{(j)}$ must be such that $\overline{w}_j$ is a permutation in $\mathbb{S}_n$ for all $1 \leq j \leq k$, and up to permuting the outputs, we can set the first one $\overline{w}_1$ to be the identity. Finally, the signs are computed as follows: the signature of $\overline{w}_j$ comes from the action of $w_{(j)}$ on the $j$-th copy of $a_1$. The sign-exponent term $j_k(n)$ is given by: 
\[
j_k(n) = \frac{k(k-1)}{2}(n-1) + \frac{k(k-1)}{2}\frac{n(n-1)}{2} = \left[\frac{k(k-1)}{2}\right] \frac{(n+2)(n-1)}{2}~, 
\] 
where the first term comes from applying $\Delta_{w_{(1)}} \otimes \cdots \otimes \Delta_{w_{(k)}}(a_1 \otimes \cdots \otimes a_1)$ and the second term comes from rearranging the outputs 

\[
a_1^{\otimes n} \otimes \cdots \otimes a_1^{\otimes n} \quad \text{into} \quad (\underbrace{a_1 \otimes \cdots \otimes a_1}_{k\text{ times}})^{\otimes n}
\]
\end{proof}

\textbf{Higher homotopy groups of the integration functor.} The explicit formulas for the coalgebra structures of the pointed spheres and of the interval allow us to give the following combinatorial description of the higher homotopy groups $\pi_k\mathcal{R}(\mathfrak{g})$ at the Maurer--Cartan element $0$. 

\begin{Definition}[Representative elements]\label{def: representative element}
Let $\mathfrak{g}$ be a qp-complete absolute $\mathcal{L}_\infty^\pi$-algebra. An element $\varepsilon$ of $\mathfrak{g}$ in degree $k \geq 1$ is a \textit{representative element} if it satisfies the following equation: 
\[
\displaystyle d_\mathfrak{g}(\varepsilon) + \sum_{n \geq 2} \sum_{\substack{w \in \E(n)_{k(n-1)}                             \\ \overline{w}_1 = \mathrm{id}_{\Sym_n}, \, \overline{w}_j \in \Sym_n}} (-1)^{\left[\frac{k(k-1)}{2}\right] \frac{(n+2)(n-1)}{2}} \prod_{j = 2}^k \mathrm{sign}(\overline{w}_j) ~ l^w_n(\varepsilon,\cdots,\varepsilon)= 0~, 
\]
where the sum runs over all $w = (\sigma_0,\cdots, \sigma_{k.(n-1)})$ in $\mathcal{E}(n)_{k.(n-1)}$ such that every $n$-tuple $\overline{w}_j \coloneqq (\sigma_{(j-1)(n-1)}(1),\cdots, \sigma_{j(n-1)}(1))$ is a permutation in $\mathbb{S}_n$ for all $1 \leq j \leq k$ and in particular $\overline{w}_1$ is the identity permutation of $\mathbb{S}_n$. 
\end{Definition}

Each operation $l^w_n$ is of degree $-k.(n-1)-1$ and thus in weight $k.(n-1)$. This means that there are a finite amount of terms in each weight and therefore the above sum converges in $\mathfrak{g}$ by qp-completeness. 

\begin{Definition}[Interval equivalences]\label{def: interval equivalences}
Let $\mathfrak{g}$ be a qp-complete absolute $\mathcal{L}_\infty^\pi$-algebra and let $\varepsilon_1$ and $\varepsilon_2$ be two representative elements of degree $k \geq 1$. They are \textit{interval-equivalent} if there exists an element $\varphi$ in $\mathfrak{g}$ of degree $k+1$ such that 
\[
d_{\mathfrak{g}}(\varphi) = \varepsilon_1 - \varepsilon_2 - \gamma_{\mathfrak{g}}\left(\sum_{\substack{n \geq 2, s \in [1,n] \\ a + b = n - s }} \sum_{\substack{w \in \E(n)_{(s-1)+k(n-1)} \\ \overline{w}'_j \in \Sym_n}} \Gamma_k(w) ~ c_n^{w}\left(\underbrace{\varphi, \cdots, \varphi}_{s},\underbrace{\varepsilon_1, \cdots, \varepsilon_1}_{a}, \underbrace{\varepsilon_2, \cdots, \varepsilon_2}_{b}\right)\right)
\]
where the coefficients $\Gamma_k(w)$ in $\{-1, 0, 1\}$ are given by 
\[
\Gamma_k(w) = \pi_{I_a}^{(n,s)}(w|_S).\varphi_{I_b}^n(\sigma_{s-1}).(-1)^{\left(k\left[\frac{s(s-1)}{2} + (n-1) \right] + \left[\frac{k(k-1)}{2}\right]\frac{(n+2)(n-1)}{2}\right)}.\prod_{j = 1}^k \mathrm{sign}(\overline{w}_j)~,
\]
and where the sum runs over all $w = (\sigma_0,\cdots, \sigma_{(s-1) + k.(n-1)})$ which satisfy the following condition: every $n$-tuple given by $\overline{w}'_j=(\sigma_{(s-1) + (j-1)(n-1)}(1), \cdots, \sigma_{(s-1) + j(n-1)}(1))$ is a permutation in $\mathbb{S}_n$ for all $1 \leq j \leq k$. Here $w|_S$ refers to the first $s$-tuple permutations $(\sigma_0, \cdots, \sigma_{s-1})$.
\end{Definition}

\begin{theorem}\label{thm: fake Berglund theorem}
Let $\mathfrak{g}$ be an qp-complete absolute $\mathcal{L}_\infty^\pi$-algebra. There is a canonical natural bijection

\[
\pi_k(\mathcal{R}(\mathfrak{g}),0) \cong \mathrm{rep}(\mathfrak{g}_k)/\sim_{\mathrm{int}}~,
\]
\vspace{0.1pc}

between the $k$-th homotopy group of $\mathcal{R}(\mathfrak{g})$ at the Maurer--Cartan $0$ and the set of representative elements of degree $k \geq 1$ in $\mathfrak{g}$ up to interval equivalences.  
\end{theorem}

\begin{proof}
Let $\mathfrak{g}$ be an qp-complete absolute $\mathcal{L}_\infty^\pi$-algebra. There are canonical bijections
\[
\pi_k(\mathcal{R}(\mathfrak{g}),0) \cong \mathrm{Hom}_{\mathsf{sSets}_*}\left(S^k, \mathcal{R}(\mathfrak{g})\right)/\sim_{\mathrm{hmt}} \cong \mathrm{Hom}_{\mathcal{EE}_\infty\text{-}\mathsf{coalg}}\left(\tilde{C}_*^c(S^k), \widehat{\mathrm{B}}^\flat_\iota(\mathfrak{g}) \right)/\sim_{\mathrm{hmt}}~,
\]
by the pointed version of the commuting triangle constructed in Theorem \ref{thm: triangle commutatif}.

\medskip

The data of a morphism of $\mathcal{EE}_\infty$-coalgebras $\tilde{C}_*^c(S^k) \longrightarrow \widehat{\mathrm{B}}^\flat_\iota(\mathfrak{g})$ is equivalent to the data of a twisting morphism $\tilde{C}_*^c(S^k) \longrightarrow \mathfrak{g}$, which is precisely the data of a representative element in $\mathrm{rep}(\mathfrak{g}_k)$ by Lemma \ref{lemma: coalgebra structure on the reduced spheres}.

\medskip

By Proposition \ref{Prop: curved non curved comparison}, we can compute homotopies between such maps in the category of $u\mathcal{EE}_\infty$-coalgebras. Such an homotopy is equivalent to the data of a curved twisting morphism $C_*^c(\Delta^1) \otimes \tilde{C}_*^c(S^k) \longrightarrow \mathfrak{g}$, since $C_*^c(\Delta^1)$ is an interval object. So it all boils down to computing the $u\mathcal{EE}_\infty$-coalgebra structure of $C_*^c(\Delta^1) \otimes \tilde{C}_*^c(S^k)$, which follows from Proposition \ref{Prop: coalgebra structure of the interval} and Lemma \ref{lemma: coalgebra structure on the reduced spheres}, using the same methods as in the proof of Lemma \ref{lemma: coalgebra structure on the reduced spheres}.
\end{proof}

\begin{Remark}[Berglund's theorem and formally étale cochains]\label{Rmk: Berglund fails + formality of cochains}
Let $\mathfrak{g}$ be a (shifted) $\mathcal{L}_\infty$-algebra over a characteristic zero field. Berglund showed in \cite{Berglund15} that there is an isomorphism
\[
\mathrm{H}_{k}(\mathfrak{g}) \cong \pi_k(\mathcal{R}(\mathfrak{g}),0)~,
\]
where the integration functor considered is that of \cite{Getzler09}; more generally, homotopy groups of $\mathcal{R}(\mathfrak{g})$ at a Maurer--Cartan $\alpha$ are given by the homology of the twisted $\mathcal{L}_\infty$-algebra by that element. These isomorphisms can be proved using the same methods as the ones used here, and ultimately rely on the fact that the coalgebra structure on the cellular chains of sphere is (homotopically) trivial over a characteristic zero field, see \cite[Subsection 2.6]{lucio2022integration}. In general, however, this coalgebra structure is non-trivial (see after the next paragraph). 

\medskip

This is consistent with the fact, proven by Lurie in \cite[Proposition 2.4.12]{Lurie11Rational}, that the $\mathbb{E}_\infty$-algebra of cochains $C^*_c(X)$ of a finite space $X$ is formally étale over a algebraically closed field of characteristic $p$. This means that its cotangent complex is acyclic. And, up to comparison and rectification results, the linear dual this cotangent complex should coincide with $\mathcal{L}_*(X)$. Thus $\mathcal{L}_*(X)$ should also be acyclic. If Berglund's theorem was to hold in that setting, it would be impossible to construct $p$-adic models using our functor $\mathcal{L}_*$, as we will do in Section \ref{Section: Lie models}. 
\end{Remark}

\textbf{Relationship with May-Steenrod chain level operations.} Kaufmann and Medina-Mardones give in \cite{chainSteenrod} an effective construction for (higher) Steenrod operations at the chain level. These terms appear as the labels of the operations in the sum of Definition \ref{def: representative element}. Let us give some examples. For instance, the cycle in the Barratt-Eccles operad that models the cup-$i$ product of arity $2$ is degree $i$ element $\sigma_i = ((12),(21),(12), \cdots)$, where $\sigma_i$ ends with $(12)$ if $i$ is odd and with $(21)$ if it is even. The first term in the equation of Definition of \ref{def: representative element} for a representative element of degree $k$ is:
\[
d_{\mathfrak{g}}(\varepsilon) + l_2^{\sigma_{(k+1)}}(\varepsilon,\varepsilon) + \cdots = 0~, 
\]
and it is labelled by the cycle representing the cup-$(k+1)$ product of arity $2$. Recall that cup-$i$ of arity $2$ are in general non-trivial over a fields of characteristic 2. Moreover, these cycles that represent the cup-$(k+1)$ (co)products in a general $\mathcal{E}$-(co)algebra appear in the above formula because of the particular $\mathcal{E}$-coalgebra structure on the reduced chains on the spheres $\tilde{C}_*^c(S^k)$. And it is interesting to notice that they lie precisely in complexity $(k+1)$ in the sense of \cite[Section 1.6]{BergerFresse04}, meaning that they belong to the $\mathbb{E}_{k+1}$ suboperad of the $\mathbb{E}_\infty$ that is $\mathcal{E}$. This agrees with the recent result of Heuts and Land in \cite{heuts2024formalitymathbbenalgebrascochainsspheres} which states that the cochains on spheres $C_c^*(S^k)$ are formal as $\mathbb{E}_k$-algebras but not as a $\mathbb{E}_{k+1}$-algebra.

\medskip

More generally, the higher cup-$(k+1)$ operations that appear in these formulas lie always in $\mathbb{E}_{k+1}$, so one should expect them to act in a homotopically non-trivial way on the spheres over a positive characteristic fields in general, where these arities that appear are those divided by the characteristic. This in important because, in principle, only the elements $w$ in $\mathcal{E}$ which act homotopically non-trivially on $\tilde{C}_*^c(S^k)$ should contribute to the sum in Definition \ref{def: representative element}. For example, in characteristic $0$, $\tilde{C}_*^c(S^k)$ is homotopically trivial as a $\mathcal{E}$-coalgebra, and one recovers Berglund's theorem explained in Remark \ref{Rmk: Berglund fails + formality of cochains}. And at a fixed prime $p$, it should be possible to simplify this sum as well. For instance, the first homotopically non-trivial terms only appear in arity $p$. This shall be the subject of future research. 

\subsection{The analogues of the (higher) Baker--Campbell--Hausdorff formulas in positive characteristic}\label{subsection: bch}
The goal of this subsection is to show that, for any absolute partition $\mathcal{L}_\infty$-algebra $\mathfrak{g}$, the horn-fillers in the $\infty$-groupoid $\mathcal{R}(\mathfrak{g})$ are given by explicit algebraic formulas, and thus that $\mathcal{R}(\mathfrak{g})$ is an algebraic $\infty$-groupoid in the sense of \cite{AlgebraicKancomplex}. This notion can be thought as a "group up to homotopy" in an algebraic sense. In particular, the formula for the horn-fillers with respect to $\Lambda_1^2$ can be thought as the analogue of the Baker--Campbell--Hausdorff formula in positive characteristic. This opens the door to a generalization of the Lie correspondences between unipotent or abstract groups and nilpotent Lie algebras to a positive characteristic setting. 

\medskip

\textbf{Formal fixed-points equations.} We adapt the algebraic methods of \cite[Appendix A.1]{robertnicoud2020higher} for the particular context we are interested in. Let $(V,d_V)$ be a chain complex over a field $\kk$. We suppose there exists a decreasing filtration $F_*V$ on $V$
\[
V = F_0V \supseteq F_1V \supseteq F_2V \supseteq \cdots \supseteq F_pV \supseteq \cdots
\]
such that $V$ is \textit{complete} with respect to it, meaning that $V \cong \lim_{p \in \mathbb{N}} V/F_pV~.$ We consider a family of multilinear maps 
\[
P_j: V^{\otimes j} \longrightarrow V \quad \text{such that} \quad P_j(F_{i_1}V, \cdots,F_{i_j}V) \subset F_{i_1 + \cdots + i_j + j}V~,
\]
that is, such that $P_j \in F_j\mathrm{Hom}(V^{\otimes j},V)$, for all $j \geq 0$. The term $P_0$ corresponds to an element $P_0(1)$ in $V$. We consider the following fixed-point equation 
\[
P(x) = P_0 + \sum_{j \geq 1}P_j(x^{\otimes j}) = x~,
\]
where we want to find a solution $x$ in $V$. Let $\tau$ be a planar tree with $d$ leaves. We denote by $\tau^P$ the linear map $V^{\otimes d} \longrightarrow V$ obtained by replacing the vertices of $\tau$ with their corresponding operation and composing these operations along $\tau$, meaning every vertex $v$ in $\tau$ with $k$ incoming edges is replaced by the operation $P_k$, and we take the composition of all these operations as they are disposed in the planar tree $\tau$. 

\begin{Proposition}\label{prop: fixed point equation}
The exists an unique solution $x$ in $V$ to the above fixed-point equation, given by 
\[
x  = \sum_{\tau \in \mathrm{PT}} \tau^P(P_0, \cdots, P_0)~,
\]
where the sum runs over all planar trees $\tau$. 
\end{Proposition}

\begin{proof}
The proof is exactly the same as in \cite[Proposition A.5]{robertnicoud2020higher}. The only difference is the weight convention. In \textit{loc.cit.} the operations $P_j$ are in weight $1$ for all $j \geq 0$, whereas here the operation $P_j$ is in weight $j$ for all $j \geq 0$. It is straightforward to check that the same arguments apply with this new weight convention. 
\end{proof}

\textbf{Horn-fillers in Maurer--Cartan spaces.} Like in \cite[Section 5.1]{robertnicoud2020higher} and in \cite[Section 2.5]{lucio2022integration}, we give an algebraic characterization of all the horn-fillers in $\mathcal{R}(\mathfrak{g})$, here in the case where $\mathfrak{g}$ is a qp-complete curved absolute partition $\mathcal{L}_\infty$-algebra. These horn-fillers are given by explicit algebraic formulas, solutions to a fixed-points equation. 

\medskip

Let $D^n$ denote the chain complex given by $\kk.u$ in degree $n$ and $\kk.du$ in degree $n-1$, where $d(u) = du$. We consider the free qp-complete pdg $\mathrm{B}^{\mathrm{s.a}}\mathcal{E}$-algebra on $D^n$, which is given by
\[
(D^n)^{\mathrm{B}^{\mathrm{s.a}}\mathcal{E}} = \prod_{m \geq 0} \widehat{\Omega}^{\mathrm{s.a}}\mathcal{E}^*(m) ~\widehat{\otimes}_{\mathbb{S}_n}~ (D^n)^{\otimes m}~. 
\]
It does not form a \textbf{curved} $\mathrm{B}^{\mathrm{s.a}}\mathcal{E}$-algebra since it does not satisfy the curved condition \ref{curved condition}. However, curved $\mathrm{B}^{\mathrm{s.a}}\mathcal{E}$-algebras are a reflexive full subcategory of pdg $\mathrm{B}^{\mathrm{s.a}}\mathcal{E}$-algebras, hence working in this context in enough for our immediate purposes. See \cite[Section 3.5]{bricevictor} for more details on these notions. 

\begin{lemma}\label{lemma: decomposition of L of Delta n}
There is an isomorphism of qp-complete pdg $\mathrm{B}^{\mathrm{s.a}}\mathcal{E}$-algebras 
\[
\mathcal{L}(\Delta^n) \cong \mathcal{L}(\Lambda_k^n) \amalg (D^n)^{\mathrm{B}^{\mathrm{s.a}}\mathcal{E}}~, 
\]
for all $n \geq 2$ and any $0 \leq k \leq n$, where $\Lambda_k^n$ denotes the $k$-horn of dimension $n$. 
\end{lemma}

\begin{proof}
Let us consider the following morphism of qp-complete pdg $\mathrm{B}^{\mathrm{s.a}}\mathcal{E}$-algebras 
\[
\begin{tikzcd}[column sep=1.5pc,row sep=0pc]
\varphi: \mathcal{L}(\Lambda_k^n) \amalg (D^n)^{\mathrm{B}^{\mathrm{s.a}}\mathcal{E}} \arrow[r]
&\mathcal{L}(\Delta^n) \\
a_I \arrow[r,mapsto]
&a_I \\
u \arrow[r,mapsto]
&a_{[n]} \\
du \arrow[r,mapsto]
&d(a_{[n]})~.  
\end{tikzcd}
\]
We want to construct an inverse $\psi$ of $\varphi$ with the following assignment 
\[
\begin{tikzcd}[column sep=1.5pc,row sep=0pc]
\psi: \mathcal{L}(\Delta^n)\arrow[r]
&\mathcal{L}(\Lambda_k^n) \amalg (D^n)^{\mathrm{B}^{\mathrm{s.a}}\mathcal{E}} \\
a_I \arrow[r,mapsto]
&a_I \\
a_{[n]} \arrow[r,mapsto]
&u \\
a_{\widehat{k}} \arrow[r,mapsto]
&x~,  
\end{tikzcd}
\]
where $\widehat{k} \subset [n]$ is the subset containing all elements except $k$. Let us denote 
\[
d(a_{[n]}) = \sum_{l=0}^n(-1)^l a_{\widehat{l}} - \sum_{m \geq 2} \sum_{w \in \mathcal{E}(m)} \sum_{\substack{I_1,\cdots,I_m \subseteq [n] \\ I_l \neq \emptyset}} \lambda^{(w;I_1,\cdots, I_m)}~ l^{w}_m(a_{I_1},\cdots,a_{I_m})~,
\]
the image of $a_{[n]}$ by the pre-differential, where the first sum comes from the differential of $C_*^c(\Delta^n)$ and the second sum comes from the decompositions of $a_{[n]}$ in the $\mathcal{E}$-coalgebra structure. In fact, the second sum is the pre-image of the decompositions of $a_{[n]}$ by the norm isomorphism described in Remark \ref{Remark: norm map}, like in previous computations. Notice that the coefficients $\lambda^{(w;I_1,\cdots, I_m)}$ are either $0$ or $\pm 1$. 

\medskip

Like in the proof of \cite[Lemma 5.1]{robertnicoud2020higher}, the element $x$ needs to satisfy a fixed-point equation where the operations are 
\[
P_0 = (-1)^k du - \sum_{l\neq k}(-1)^{l+k} a_{\widehat{l}}+  (-1)^k\sum_{m \geq 2} \sum_{w \in \mathcal{E}(m)} \sum_{\substack{I_1,\cdots,I_m \subseteq [n] \\ I_l \neq \emptyset,\widehat{k}}} \lambda^{(w;I_1,\cdots, I_m)} ~ l^{w}_m(a_{I_1},\cdots,a_{I_m})
\]
that is, where the second sum only involves the decompositions of $a_{[n]}$ without the term $a_{\widehat{k}}$; and 
\[
P_j = (-1)^k\sum_{m \geq 2} \sum_{w \in \mathcal{E}(m+j)} \sum_{\substack{I_1,\cdots,I_m \subseteq [n] \\ I_l \neq \emptyset,\widehat{k}}} \lambda^{(w;I_1,\cdots,\widehat{k},\cdots,\widehat{k},\cdots, I_m)} ~ l^{w}_m(a_{I_1},\cdots,—,\cdots,—,\cdots,a_{I_m})~,
\]
where the terms that appear correspond to decompositions of $a_{[n]}$ which involve $j$-times the term $a_{\widehat{k}}$, for all $j \geq 1$. These slots are left as inputs of the operation $P_j$; there are exactly $j$-inputs. 

\medskip

The term $P_0$ is in weight zero, as any $a_{\widehat{l}}$ is a generator, and the term $P_1$ is in weight $1$, since it is a sum of structural operations. Let us show that $P_j$ is necessarily in weight at least $j$ when $j \geq 2$. Any $w$ in $\mathcal{E}(n)$ which decomposes $a_{[n]}$ into $j$ copies of $a_{\widehat{k}}$ must be of degree at least $j(n-1)-n$, and therefore $l_m^w$ is in weight at least $j(n-1)-n + 1$. This number is greater than $j$ if $n \geq 3$. For $n = 2$, the result hold because of the specific formulas of \cite[Section 2]{BergerFresse04}. If $a_{02}$ appears $j$-times, the interval cut is at least of length $2j + 2$, hence the degree of $w$ is at least $j$, thus the weight of $l_m^w$ is at least $j+1$. If $a_{01}$ or $a_{12}$ appear $j$-times, the interval cut is at least of length $2j + 1$, hence the degree of $w$ is at least $j-1$, thus the weight of $l_m^w$ is at least $j$. Therefore this fixed-point equation has an unique solution $x$ by Proposition \ref{prop: fixed point equation}. We then conclude using the same arguments as in the proof of \cite[Lemma 5.1]{robertnicoud2020higher}. 
\end{proof}

\begin{theorem}\label{thm: horn-fillers}
Let $\mathfrak{g}$ be a qp-complete curved absolute $\mathcal{L}_\infty^\pi$-algebra. There are bijections 
\[
\rho_k^n: \mathfrak{g}_n \times \mathrm{Hom}_{\mathsf{sSet}}(\Lambda_k^n,\mathcal{R}(\mathfrak{g})) \cong \mathrm{Hom}_{\mathsf{sSet}}(\Delta^n, \mathcal{R}(\mathfrak{g})) ~,
\]
natural in $\mathfrak{g}$, for all $n \geq 2$ and $0 \leq k \leq n$. 
\end{theorem}

\begin{proof}
We have the following natural isomorphism 
\[
\mathrm{Hom}_{\mathsf{sSet}}(\Delta^n, \mathcal{R}(\mathfrak{g})) \cong \mathrm{Hom}_{\mathsf{curv}~\mathsf{abs}~\mathcal{L}_\infty^\pi\text{-}\mathsf{alg}^{\mathsf{qp}\text{-}\mathsf{comp}}}(\mathcal{L}(\Delta^n), \mathfrak{g})~,
\]
given by the $\mathcal{L} \dashv \mathcal{R}$ adjunction. Using Lemma \ref{lemma: decomposition of L of Delta n} and the fact that qp-complete curved absolute $\mathcal{L}_\infty^\pi$-algebras are a full subcategory of pdg $\mathrm{B}^{\mathrm{s.a}}\mathcal{E}$-algebras, we get 
\begin{align*}
\mathrm{Hom}_{\mathsf{curv}~\mathsf{abs}~\mathcal{L}_\infty^\pi\text{-}\mathsf{alg}^{\mathsf{qp}\text{-}\mathsf{comp}}}(\mathcal{L}(\Delta^n), \mathfrak{g}) &\cong \mathrm{Hom}_{\mathsf{pdg}~\mathrm{B}^{\mathrm{s.a}}\mathcal{E}\text{-}\mathsf{alg}}(\mathcal{L}(\Delta^n), \mathfrak{g}) \\
&\cong \mathrm{Hom}_{\mathsf{pdg}~\mathrm{B}^{\mathrm{s.a}}\mathcal{E}\text{-}\mathsf{alg}}((D^n)^{\mathrm{B}^{\mathrm{s.a}}\mathcal{E}}, \mathfrak{g}) \times \mathrm{Hom}_{\mathsf{pdg}~\mathrm{B}^{\mathrm{s.a}}\mathcal{E}\text{-}\mathsf{alg}}(\mathcal{L}(\Lambda_k^n), \mathfrak{g}) \\
&\cong \mathrm{Hom}_{\mathsf{pdg}~\mathsf{mod}}(D^n, \mathfrak{g}) \times \mathrm{Hom}_{\mathsf{pdg}~\mathrm{B}^{\mathrm{s.a}}\mathcal{E}\text{-}\mathsf{alg}}(\mathcal{L}(\Lambda_k^n), \mathfrak{g}) \\
&\cong \mathfrak{g}_n \times \mathrm{Hom}_{\mathsf{curv}~\mathsf{abs}~\mathcal{L}_\infty^\pi\text{-}\mathsf{alg}^{\mathsf{qp}\text{-}\mathsf{comp}}}(\mathcal{L}(\Lambda_k^n), \mathfrak{g}) \\
&\cong  \mathfrak{g}_n \times \mathrm{Hom}_{\mathsf{sSet}}(\Lambda_k^n,\mathcal{R}(\mathfrak{g}))~. 
\end{align*}
\end{proof}

Let $\mathfrak{g}$ be a qp-complete curved absolute $\mathcal{L}_\infty^\pi$-algebra. The simplicial set $\mathcal{R}(\mathfrak{g})$ is not only a Kan complex, it is an algebraic Kan complex in the sense of \cite{AlgebraicKancomplex}. Indeed, for any lifting problem $\Lambda_k^n \longrightarrow \mathcal{R}(\mathfrak{g})$, there is a canonical horn-filler given by $0$ in $\mathfrak{g}_n$ under the bijection of Theorem \ref{thm: horn-fillers}. One may thing of $\mathcal{R}(\mathfrak{g})$ not as satisfying the \textit{property} of being a Kan complex, but as being endowed with additional \textit{structure}. From this point of view, $\mathcal{R}(\mathfrak{g})$ is a \textit{group up to homotopy} in an algebraic sense. 

\medskip

The data of an $n$-simplex $\kappa$ in $\mathcal{R}(\mathfrak{g})_n$ is equivalent to the data of a curved twisting morphism $\alpha_\kappa: C_*^c(\Delta^n) \longrightarrow \mathfrak{g}$. It is in particular a collection of elements $\alpha_\kappa(a_I)$ in $\mathfrak{g}$ of degree $\mathrm{Card}(I)-1$ for all non-empty subsets $I \subseteq [n]$. They satisfy algebraic equations imposed by the $\mathcal{E}$-coalgebra structure of $C_*^c(\Delta^n)$. \textit{Mutatis mutandis}, the same holds for the data of a horn $\Lambda_k^n$ in $\mathcal{R}(\mathfrak{g})$, replacing $C_*^c(\Delta^n)$ by $C_*^c(\Lambda^n_k)$. 

\begin{Definition}[Horn-filling operations]
Let $\mathfrak{g}$ be a qp-complete curved absolute $\mathcal{L}_\infty^\pi$-algebra and let $y$ be a degree $n$ element in $\mathfrak{g}$. The \textit{horn-filler operation} with respect to $y$ is defined as
\[
\begin{tikzcd}[column sep=1.5pc,row sep=0pc]
\mathrm{HF}_y: \mathrm{Hom}_{\mathsf{sSet}}(\Lambda_k^n,\mathcal{R}(\mathfrak{g})) \arrow[r]
&\mathfrak{g}_{n-1} \\
\{x_{I}\}_{\substack{I \subset [n], \\ I \neq \emptyset,\widehat{k}}} \arrow[r,mapsto]
&\rho_k^n(y,x_I)(a_{\widehat{k}})~. 
\end{tikzcd}
\]
It sends any horn $\Lambda_k^n$ in $\mathcal{R}(\mathfrak{g})$, that is, any collection $\{x_{I}\}$ of elements in $\mathfrak{g}$ satisfying the algebraic equations imposed by the $\mathcal{E}$-coalgebra structure of $C_*^c(\Lambda^n_k)$, to the element $\rho_k^n(y,x_I)(a_{\widehat{k}})$ in $\mathfrak{g}_{n-1}$. This element is given constructed as follows: associated to the pair $(y,x_I)$, there exists an unique $n$-simplex $\rho_k^n(y,x_I)$ in $\mathcal{R}(\mathfrak{g})$ by the bijection of Theorem \ref{thm: horn-fillers}; this $n$-simplex is the data of a curved twisting morphism $\rho_k^n(y,x_I):  C_*^c(\Delta^n) \longrightarrow \mathfrak{g}$, the horn-filler product $\mathrm{HF}_y(x_I)$ is the image of $a_{\widehat{k}}$ by this curved twisting morphism. 
\end{Definition}

\begin{Remark}
This definition is analogous to the \textit{higher Baker--Campbell--Hausdorff} products defined in \cite{robertnicoud2020higher} and extended to the curved setting in \cite{lucio2022integration}. 
\end{Remark}

Recall that a symmetric corked \textit{planar} tree $\tau$ is a \textit{planar tree} where every vertex has at least two incoming edges or zero incoming edges (called corks). Every vertex $v$ is labelled by an element $\underline{\sigma}_v$ in $\mathcal{E}(\mathrm{In}(v))$, where $\mathrm{In}(v)$ is the number of incoming edges of the vertex. The set of symmetric corked planar trees is denoted by $\mathrm{SCPT}$.

\medskip

The image of the element $a_{[n]}$ in $\mathcal{L}(\Delta^n)$ by the pre-differential is given by 
\[
d(a_{[n]}) = \sum_{l=0}^n(-1)^l a_{\widehat{l}} - \sum_{m \geq 2} \sum_{w \in \mathcal{E}(m)} \sum_{\substack{I_1,\cdots,I_m \subseteq [n] \\ I_l \neq \emptyset}} \lambda^{(w;I_1,\cdots, I_m)}~ l^{w}_m(a_{I_1},\cdots,a_{I_m})~,
\]
where the coefficients $\lambda^{(w;I_1,\cdots, I_m)}$ are pre-image by the norm map of the decomposition of $a_{[n]}$ by the $\mathcal{E}$-coalgebra structure of $C_*^c(\Delta^n)$. These coefficients are either $0$ or $\pm 1$. From these coefficients, we are going to define a coefficient for any symmetric corked planar tree $\tau$ of arity $m$ with leaves labelled by a collection $(I_1, \cdots, I_m)$ of subsets $I_j \subseteq [n]$, where $I_j \neq \widehat{k}$ and $0 \leq k \leq n$. We define the coefficient $\lambda^{(\tau; I_1, \cdots, I_m)}$ as 
\[
\lambda^{(\tau; I_1, \cdots, I_m)} = \prod_{v \in V(\tau)} (-1)^k \lambda^{(\underline{\sigma}_v; I_{i_1},\cdots,\widehat{k},\cdots, I_{\mathrm{in}(v)})}~, 
\]
where $\underline{\sigma}_v$ is the label of the vertex $v$, and where the product runs over all vertices of $\tau$. 

\medskip

The subsets that appear in the coefficient $\lambda^{(\underline{\sigma}_v; I_{i_1},\cdots,\widehat{k},\cdots, I_{\mathrm{in}(v)})}$ are chosen as follows: from left to right, if an incoming edge is a leaf, the we chose the label of the leaf, if an incoming edge is attached to another vertex, we chose the subset $\widehat{k}$. 

\begin{Proposition}\label{prop: general formula for higher bch}
Let $\mathfrak{g}$ be a qp-complete curved absolute $\mathcal{L}_\infty^\pi$-algebra, let $y$ be a degree $n$ element in $\mathfrak{g}$ and let $\{x_I\}$ be a collection of elements in $\mathfrak{g}$ corresponding to a horn $\Lambda_k^n$ in $\mathcal{R}(\mathfrak{g})$. The horn-filler operation with respect to $y$ of $\{x_I\}$ is given by the following formula 
\[
\mathrm{HF}_y(x_I) = \gamma_\mathfrak{g}\left(\sum_{m \geq 0} \sum_{\tau \in \mathrm{SCPT}_m} \sum_{\substack{I_1,\cdots,I_m \subseteq [n] \\ I_l \neq \emptyset,\widehat{k}}} \lambda^{(\tau; I_1, \cdots, I_m)} \tau\left(x_{I_1}, \cdots, x_{I_m}; (-1)^k dy - \sum_{l\neq k}(-1)^{l+k}x_{\widehat{l}}\right) \right)~, 
\]
where $\tau\left(x_{I_1}, \cdots, x_{I_m}; (-1)^k dy - \sum_{l\neq k}(-1)^{l+k}x_{\widehat{l}}\right)$ denotes the planar tree obtained by labelling the leaves of $\tau$ by the elements $x_{I_1}, \cdots, x_{I_m}$ and by replacing the corks in $\tau$ by the element $(-1)^k dy - \sum_{l\neq k}(-1)^{l+k}x_{\widehat{l}}$.
\end{Proposition}

\begin{proof}
Follows directly from solving the formal fixed-point equation in the proof of Lemma \ref{lemma: decomposition of L of Delta n}, using the explicit solution of Proposition \ref{prop: fixed point equation}. 
\end{proof}

\begin{Remark}
Computing the coefficients $\lambda^{(\underline{\sigma}_v; I_{i_1},\cdots,\widehat{k},\cdots, I_{\mathrm{in}(v)})}$ is the only obstacle to fully determining the horn-filler products. Using the Python package for the Barrat--Eccles operad of \cite{pythonbarratteccles}, one can write a computer code for the $\mathcal{E}$-coalgebra structure of $C_*^c(\Delta^n)$ and compute first terms of the above formula. 
\end{Remark}

\textbf{The analogue of the Baker--Campbell--Hausdorff formula.} Let $\mathfrak{g}$ be a qp-complete absolute $\mathcal{L}_\infty^\pi$-algebra. Then $0$ is a Maurer--Cartan element. The data of a horn $\Lambda_1^2$ in $\mathcal{R}(\mathfrak{g})$ based at the Maurer--Cartan $0$ is given by 

\begin{center}
\includegraphics[width=45mm,scale=0.5]{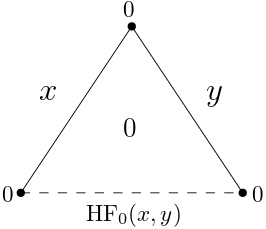},
\end{center}

where $x$ and $y$ are two representative elements in $\mathfrak{g}_1$, as defined in Definition \ref{def: representative element}. The horn-filler product $\mathrm{HF}_0(x,y)$ with respect to $0$ in $\mathfrak{g}_2$ of $x$ and $y$ gives a representative element in $\mathfrak{g}_1$ which represents the composition of the two paths represented by $x$ and $y$ inside $\pi_1(\mathcal{R}(\mathfrak{g}),0)$. This precise construction recovers the classical Baker--Campbell-Hausdorff formula in the characteristic zero setting, see \cite[Theorem 5.2.37]{Bandiera} or \cite[Corollary 5.20]{robertnicoud2020higher}. Our is to give a closed formula for the analogue of the Baker--Campbell--Hausdorff formula in the setting of absolute partition $\mathcal{L}_\infty^\pi$-algebras. 

\medskip

Let $\tau$ be a symmetric corked planar tree. It is \textit{left-handed} if for every vertex $v$ with $\mathrm{In}(v) \geq 2$ incoming edges, the $j$-th incoming edge is linked to another edge only if all the previous $l$ incoming edges (reading from left to right) are linked to an edge, for all $0 \leq l \leq j$ and every $j \in [1,\mathrm{In}(v)]$. Furthermore, every vertex which is not a cork must have at least two leaves, so incoming edges linked to another vertex can only occur in vertices with valence at least $3$. 

\medskip

A labelling of such a $\tau$ by $01$ and $12$, subsets of $[0,1,2]$, is \textit{licit} if for every vertex $v$ of $\tau$ which is not a cork, going from left to right of its incoming edges, $v$ has $c$ incoming edges attached to other edges, $a$ leaves labelled by $01$ and then $b$ leaves labelled by $12$, where $a + b + c = \mathrm{in}(v)$ and $c \geq 0$ and $a,b \geq 1$. In other words, a label $01$ cannot appear after a label $12$ and since $v$ has at least two leaves, at least one of them is labelled by $01$ and one of them by $12$. We denote the set of licit labels by $L^{01,12}(\tau)$. 

\medskip

We define a coefficient $\alpha^{(\tau;~01,\cdots,12)}$ for every left-handed symmetric corked planar tree $\tau$ together with a licit labelling by $01,12$ as follows: 
\[
\alpha^{(\tau;~ 01,\cdots,12)} = \prod_{v \in V(\tau)}(-1)\alpha^{(\underline{\sigma};~ 01, \cdots, 12)}~,
\]
where $\underline{\sigma}_v$ is the label of the vertex $v$, and where the product runs over all vertices of $\tau$ which are not corks. The coefficients at each vertex $v$ are given as follows: suppose the vertex $v$ has $c$ incoming edges attached to other edges, $a$ leaves labelled by $01$ and then $b$ leaves labelled by $12$, where $a + b + c = n$ is the total number of incoming edges, and where $c \geq 0$ and $a, b \geq 1$. Then 
\[
\alpha^{(\underline{\sigma};~ 01, \cdots, 12)} = \begin{cases}
    \epsilon.\varphi_{I_a}^n(\underline{\sigma}_{a+c-1}(1))\varphi_{I_c}^n(\underline{\sigma}_{a+c-1}(1))\pi^{(n,I_b)}_{I_a'}(\underline{\sigma}_{b}) & \quad \text{if } \mathrm{degree}(\underline{\sigma}) = n-2~, \\
    0 & \quad \text{otherwise}~,
\end{cases}
\]
where $\underline{\sigma}_{a+c-1} = (\sigma_0, \cdots \sigma_{a+c-2})$, where $\underline{\sigma}_{b}= (\sigma_{a+c-1},\cdots, \sigma_{n-2})$, and where $I_c =\{1,\cdots,c\}$, $I_b = \{c+a+1\}$ and where $I_a = \{c+1, \cdots, c+a\}$ and $I_a' = \{c+a,c+1,\cdots, c+a-1\}$. Notice that $I_a$ and $I_a'$ have the same underlying set but different orders. The sign $\epsilon$ is given by 
\[
\epsilon = \mathrm{sign}(\sigma_0(1), \cdots \sigma_{a+c-2}(1)).(-1)^{b-1}~. 
\]
In plaint words, the coefficient is non-trivial if and only if the permutation $(\sigma_0(1), \cdots \sigma_{a+c-2}(1))$ in $\mathbb{S}_{a+c-1}$ is an $(c,a-1)$-unshuffle and if the last $b$-terms $(\sigma_{a+c-1},\cdots, \sigma_{n-2})$ forms a $I_b$-ordered partition of $I_a'$. 

\begin{Example}
Here is an example of left-handed symmetric corked planar tree with a licit labelling (the only possible licit labelling of this tree in fact). 
\vspace{0.5pc}
\begin{center}
\includegraphics[width=80mm,scale=1]{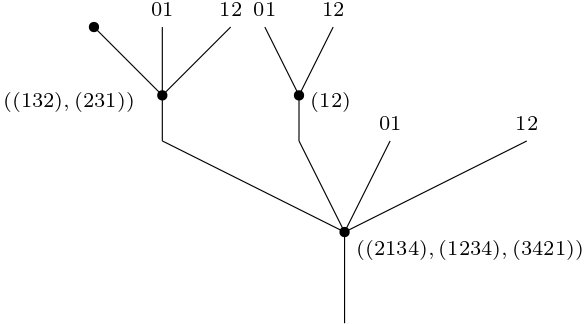},
\end{center}
\vspace{0.5pc}
The coefficients of the vertex labelled by $(12)$ and the vertex labelled by $((132),(231))$ are both $1$. The coefficient of the last vertex is $-1$. The coefficient of the whole tree is therefore $(-1)^2 = 1$. 
\end{Example}

\begin{theorem}\label{thm: horn-filler formula (BCH)}
Let $\mathfrak{g}$ be a qp-complete absolute $\mathcal{L}_\infty^\pi$-algebra, and let $x,y$ be two representative elements in $\mathfrak{g}_1$. The horn-filler product $\mathrm{HF}_0(x,y)$ is given by 

\[
\mathrm{HF}_0(x,y) = \sum_{m \geq 0} \sum_{\tau \in \mathrm{SCPT}^{\mathrm{left}}_m} \sum_{L^{01,12}(\tau)} \alpha^{(\tau;~ 01, \cdots, 12)} \gamma_\mathfrak{g}\left(\tau\left(x,\cdots, y, \cdots, x, \cdots, y; x+y \right)\right) ~, 
\]
\vspace{0.1pc}

where $x$ appears if a leaf is labelled by $01$ and $y$ by $12$, where the corks are replaced by the input $x+y$ and where the sum is taken over all left-handed symmetric corked planar trees and all licit labellings of those. 
\end{theorem}

\begin{proof}
We consider $C_*^c(\Delta^2)$ together with its $\mathcal{E}$-coalgebra structure. We want to compute which operations $\underline{\sigma}$ in $\mathcal{E}$ give a decomposition of $a_{012}$ with $c$ terms $a_{02}$, $a$ terms $a_{01}$ and $b$ terms $a_{12}$, for any $a,b,c \geq 0$. Let $n = a + b + c$, it is immediate that such an operation is of degree $n-2$. 

\medskip

The rest of the proof follows the same ideas as the proof of Proposition \ref{Prop: coalgebra structure of the interval}. We first compute which surjections $u$ act non-trivially giving us the decomposition we want. We fix the following order on the decompositions: 
\[
a_{02}^{\otimes c} \otimes a_{01}^{\otimes a} \otimes a_{12}^{\otimes b} = \underbrace{a_{02}}_{1} \otimes \cdots \otimes \underbrace{a_{02}}_{c} \otimes \underbrace{a_{01}}_{c+1} \otimes \cdots \otimes \underbrace{a_{01}}_{c+a} \otimes \underbrace{a_{12}}_{c+a+1} \otimes \cdots \otimes \underbrace{a_{12}}_{c+a+b},
\]
with respect to the interval cut operation of $u$. Notice that $a,b \geq 1$. Indeed, if there is a $1$ in the interval cut, then both $a \geq 1$ and $b \geq 1$. If there are no $1$'s, the interval cut only contains $0$'s and $2$'s. Since the operation is of arity $n$ and degree $n-2$, there are $2n-1$ cuts, and thus there must be a repetition and the resulting decomposition is trivial. In general, a surjection $u$ that produces such a decomposition has a table arrangement of the form
\[
\left|~~ 
\begin{aligned}
&u(1) \\
&\vdots \\
&u(a+c-1)\\
&c+a ~ c+1 ~\cdots~ c+i_1~~ c+a+1\\
&c+i_1+1 ~ \cdots ~ c+i_2+i_1~~ c+a+2 \\
&\vdots \\
&c+i_{b-2}+ \cdots +1 ~ \cdots ~ c+i_{b—1}+\cdots + i_1~~ c+a+b-1 \\
&c+i_{b-1}+ \cdots +1 ~ \cdots ~ c+i_b+\cdots +i_1 ~~ n ~~ u(2a + b + c)~ \cdots ~ u(2a + 2c + 2b -2)~,
\end{aligned}
\right.
\]
\vspace{0.1pc}

where $(u(1),\cdots,u(a+c-1))$ is a $(c,a-1)$-unshuffle in $\mathbb{S}_{a+c-1}$ and where $(i_1,\cdots,i_b)$ are integers $\geq 0$ such that $i_1+\cdots+i_b=a$. If $i_j$ is zero, then $c+a+j$ is the first term in the $(a+c-1+j)$-th row of the table arrangement of $u$. These terms determine an $I_b$-ordered partition of $I_a' = \{c+a,c+1,\cdots, c+a-1\}$. The position sign associated to such a surjection is $(-1)^{b-1}$, since there are $b-1$ inner intervals ending with $1$. The permutation sign is given by the signature of the permutation $(u(1),\cdots,u(a+c-1))$. Indeed, when rearranging the terms, these are the only permutations that involve two intervals of length $1$; the rest involve intervals of length $0$ and $1$, and therefore produce no signs. From the above, we get the characterization of the elements $\underline{\sigma}$ in $\mathcal{E}(n)_{n-2}$ that have a surjection $u$ of this type which appears in their table reduction. This characterization is precisely encoded in the coefficients $\alpha^{(\underline{\sigma};~ 01, \cdots, 12)}$, where incoming edges connected to another vertex correspond to the terms $a_{02}$ in the decomposition, and the labels $01$ and $12$ to the terms $a_{01}$ and $a_{12}$, respectively. The sign that appears in these coefficients corresponds to the sign of the decomposition. We conclude adapting Proposition \ref{prop: general formula for higher bch} to our particular computation. 
\end{proof}

The first terms of the formula are the following:
\[ 
\begin{aligned}
\mathrm{HF}_0(x,y) =&~x + y - l_2^{(12)}(x,y) - l_3^{((123),(213))}(x,x,y) - l_3^{((132),(213))}(x,x,y) - l_3^{((123),(231))}(x,x,y)-\\
&-l_3^{((132),(231))}(x,x,y) - l_3^{((123),(231))}(y,x,y) - l_3^{((132),(231))}(y,x,y) +\\
&+ l_3^{((123),(132))}(x,y,y) + l_3^{((123),(312))}(x,y,y) + l_3^{((123),(321))}(x,y,y)+\\
&+l_3^{((123),(231))}\left(l_2^{(12)}(x,y),x,y\right) + l_3^{((132),(231))}\left(l_2^{(12)}(x,y),x,y\right) + \cdots \\
\end{aligned}
\]
Let us make a few observations. Notice that unlike the classical Baker--Campbell--Hausdorff formula, which involves coefficients of the form $1/r!$, the coefficients in this formula are all $\pm 1$. Furthermore, by suitably permuting the labels, they can be all made positive. It is thus obvious that such a formula is well-defined in any characteristic. Another point is that unlike the classical BCH formula, which is made of iterations of the Lie bracket, here there are no iterations of the first bracket $l_2^{(12)}(x,y)$ alone. All iterations involve higher arity operations as well. However, when $\kk$ is a field of characteristic zero, the group produced by this formula can be compared to the group produced by the BCH formula; in fact, they are isomorphic. See Subsection \ref{subsection: comparison results} for more details.
 
\begin{theorem}\label{thm: integration gives a nilpotent group}
Let $\mathfrak{g}$ be a nilpotent partition $\mathcal{L}_\infty^\pi$-algebra concentrated in degree $1$. Then $\mathcal{R}(\mathfrak{g})$ is the classifying space of the group 
\[
\big(\mathfrak{g},\mathrm{HF}_0(-,-),0\big)~,
\]
which is a nilpotent group. 
\end{theorem}

\begin{proof}
By Theorem \ref{thm: fake Berglund theorem}, its clear that $\mathcal{R}(\mathfrak{g})$ is the classifying space of its first homotopy group $\pi_1(\mathcal{R}(\mathfrak{g}),0)$, since it is connected and every other homotopy group is trivial. Furthermore, every element in $\mathfrak{g}$ is a representative element, and there are no interval equivalences since $\mathfrak{g}_2 = \{0\}$. Hence the underlying set of $\pi_1(\mathcal{R}(\mathfrak{g}),0)$ is given by $\mathfrak{g}$, and moreover, the composition of paths in $\pi_1(\mathcal{R}(\mathfrak{g}),0)$ is explicitly given by the formula $\mathrm{HF}_0(-,-)$. 

\medskip

Since $\mathfrak{g}$ is concentrated in degree $1$, operations in arity $n$ must be in degree $n-2$. Therefore $\mathfrak{g}$ is weight-nilpotent if and only if it is arity-nilpotent; and in any of these cases it is nilpotent in the sense of Definition \ref{def: nilpotent partition Linfiniti algebra}. Let us prove that it is nilpotent by induction on the weight. If $\mathfrak{g}$ is abelian, meaning it endowed with the trivial structure and every operation is zero, then the horn-filler product is 
\[
\mathrm{HF}_0(x,y) = x + y~,
\]
for any $x,y$ in $\mathfrak{g}$. We get an abelian and hence a nilpotent group. Now consider the short exact sequence of $\kk$-vector spaces
\[
0 \longrightarrow \mathrm{W}_{\delta}\mathfrak{g}/\mathrm{W}_{\delta + 1}\mathfrak{g} \hookrightarrow \mathfrak{g}/\mathrm{W}_{\delta +1} \mathfrak{g} \twoheadrightarrow \mathfrak{g}/\mathrm{W}_{\delta}\mathfrak{g} \longrightarrow 0~. 
\]
These are all nilpotent partition $\mathcal{L}_\infty^\pi$-algebra maps, hence they induce a short exact sequence between their associated groups. The structure on $\mathrm{W}_{\delta}\mathfrak{g}/\mathrm{W}_{\delta + 1}\mathfrak{g}$ is trivial, hence its associated group is abelian. The group $\mathfrak{g}/\mathrm{W}_{\delta}\mathfrak{g}$ is nilpotent by induction hypothesis. Since the group associated to $\mathfrak{g}/\mathrm{W}_{\delta +1} \mathfrak{g}$ is an extension of a nilpotent group by an abelian group, it is nilpotent if this extension is central. Let $x$ be in $\mathrm{W}_{\delta}\mathfrak{g}/\mathrm{W}_{\delta + 1}\mathfrak{g}$ and $g$ be in $\mathfrak{g}/\mathrm{W}_{\delta +1} \mathfrak{g}$, then 
\[
\mathrm{HF}_0(x,g) = x + g = g + x = \mathrm{HF}_0(g,x)~. 
\] 
Indeed, since $x$ lies precisely in weight $\delta$, any operations applied to $x$ lies in weight $\delta + 1$ and hence vanishes in $\mathfrak{g}/\mathrm{W}_{\delta +1} \mathfrak{g}$. There are no terms in the formula of $\mathrm{HF}_0(x,g)$ involving operations only labelled by $g$. Therefore all higher terms vanish in the formulas for $\mathrm{HF}_0(x,g)$ and $\mathrm{HF}_0(g,x)$. Since $x$ then commutes with any element $g$ in $\mathfrak{g}/\mathrm{W}_{\delta +1} \mathfrak{g}$, this group extension is central, which concludes the proof.
\end{proof}

\begin{Remark}\label{Rmk: degree 1 and k-points of unipotent groups}
Notice that a nilpotent partition $\mathcal{L}_\infty^\pi$-algebra does not need to be concentrated in degree $1$ in order for $\mathcal{R}(\mathfrak{g})$ to be the classifying space of a group. If the representative elements in all degrees different from $1$ are trivial up to interval equivalence, then by Theorem \ref{thm: fake Berglund theorem} the space $\mathcal{R}(\mathfrak{g})$ is the classifying space of its $\pi_1$, whose group structure is still determined by  the formula in Theorem \ref{thm: horn-filler formula (BCH)}. 
\end{Remark}

\begin{Example}\label{Example: 2-nilpotent associative algebras}
Let us give a first simple non-trivial example of the construction of Theorems \ref{thm: horn-filler formula (BCH)} and \ref{thm: integration gives a nilpotent group}. The data of a $2$-nilpotent partition $\mathcal{L}_\infty$-algebra concentrated in degree $1$ amounts to the data of a $\kk$-vector space $\mathfrak{g}$ together with a binary operation
\[
l_2: \mathfrak{g} \otimes \mathfrak{g} \longrightarrow \mathfrak{g}~,
\]
such that any composition of $l_2$ with itself is zero, i.e: $l_2(l_2(-,-),-) = l_2(-,l_2(-,-)) = 0$. This structure is sometimes referred to as a $2$-nilpotent associative algebra. 

\medskip

The following statements can be checked by direct computations. 

\begin{itemize}
\item The horn-filler product of Theorem \ref{thm: horn-filler formula (BCH)} is given in this case by
\[
x \star y \coloneqq x + y - l_2(x,y)~,
\]
and it defines a group structure on $\mathfrak{g}$, where the neutral element is $0$ and the inverse is given by 
\[
x^{-1} = -x -l_2(x,x)~. 
\]

\item The group $(\mathfrak{g},\star,0)$ is nilpotent of nilpotency class $2$. Commutators are given by 
\[
x \star y \star x^{-1} \star y^{-1} = l_2(y,x) - l_2(x,y)~. 
\]

\item If $\kk$ is of characteristic $p$, then $x^{\star p} = l_2(x,x)$ and thus $x^{\star p^2} = p.l_2(x,x) = 0$, meaning every element $x$ is of order at most $p^2$. 
\end{itemize}

\medskip

For example, if $\mathfrak{g}$ is a $\kk$-vector space of dimension $3$ denoted by $\kk^{\oplus 3}$ and if $l_2$ is given by 
\[
\begin{tikzcd}[column sep=3.5pc,row sep=0pc]
l_2: \kk^{\oplus 3} \otimes \kk^{\otimes 3} \arrow[r]
&\kk^{\oplus 3} \\
(\alpha,\beta,\gamma) \otimes (\alpha',\beta',\gamma') \arrow[r,mapsto]
&(0,0,-\alpha.\beta')~, 
\end{tikzcd}
\]
then the resulting group $(\kk^{\oplus 3},\star,0)$ is isomorphic to $\mathrm{H}_3(\kk)$, the Heisenberg group of $3\text{x}3$ upper triangular matrices with $1$'s in the diagonal and coefficients in $\kk$. 
\end{Example}

\begin{Example}\label{example: 3-nilpotent monodic Lie algebra}
A $3$-nilpotent partition $\mathcal{L}_\infty$-algebra concentrated in degree $1$ amounts to the data of a vector space $\mathfrak{g}$ endowed with a binary operation 
\[
l_2^{\id}(-,-): \mathfrak{g} \otimes \mathfrak{g} \longrightarrow \mathfrak{g}~, 
\]
and with a familly of ternary operations 
\[
l_3^{(\id,\sigma)}(-,-,-): \mathfrak{g} \otimes \mathfrak{g} \otimes \mathfrak{g} \longrightarrow \mathfrak{g}~, 
\]    
for all permutations $\sigma \neq \id$ in $\mathbb{S}_3$. The operations $l_3^{(\sigma, \id)}$ are obtained from $l_3^{(\id,\sigma^{-1})}$ up to the sign $\mathrm{sign}(\sigma)$ by permuting the entries of this later operation via $\sigma$. Since it is $3$-nilpotent, the only eventually non-trivial compositions of the operations above are $l_2^{\id} \circ_1 l_2^{\id}$ and  $l_2^{\id} \circ_2 l_2^{\id}$. The only relation that these operations satisfy by Appendix \ref{Appendice formules} is 

\begin{equation}\label{eq: relation 3 order}
l_2^{\id} \circ_1 l_2^{\id} + \sum_{\sigma \neq \id \in \mathbb{S}_n} l_3^{(\sigma, \id)} = l_2^{\id} \circ_2 l_2^{\id} + \sum_{\sigma \neq \id \in \mathbb{S}_n} l_3^{(\id, \sigma)}~,
\end{equation}

which is given by the fact that $\partial(l_3^{\id})$ is zero. Then it can be checked by direct computation that 
\begin{itemize}
    \item the horn-filler product is given by 
    \[ 
\begin{aligned}
\mathrm{HF}_0(x,y) =&~x + y - l_2^{(12)}(x,y) - l_3^{((123),(213))}(x,x,y) - l_3^{((132),(213))}(x,x,y) - l_3^{((123),(231))}(x,x,y)-\\
&-l_3^{((132),(231))}(x,x,y) - l_3^{((123),(231))}(y,x,y) - l_3^{((132),(231))}(y,x,y) +\\
&+ l_3^{((123),(132))}(x,y,y) + l_3^{((123),(312))}(x,y,y) + l_3^{((123),(321))}(x,y,y)~.\\
\end{aligned}
\]

\medskip

    \item One can also show that the horn-filler product is associative by comparing $\mathrm{HF}_0(\mathrm{HF}_0(x,y),z)$ (left composition) and $\mathrm{HF}_0(x,\mathrm{HF}_0(y,z))$ (right composition). This boils down to comparing the terms 
\[ 
\begin{aligned}
l_2^{\id} \circ_1 l_2^{\id}(x,y,z) - l_3^{((123),(213))}(x,y,z) - l_3^{((123),(213))}(y,x,z)- l_3^{((132),(213))}(x,y,z) - l_3^{((132),(213))}(y,x,z) \\
-l_3^{((123),(231))}(x,y,z) -l_3^{((123),(231))}(y,x,z) -l_3^{((132),(231))}(x,y,z) -l_3^{((132),(231))}(y,x,z)~,\\
\end{aligned}
\]
    which appear on the left composition, with the terms 
\[ 
\begin{aligned}
l_2^{\id} \circ_2 l_2^{\id}(x,y,z) - l_3^{((123),(231))}(z,x,y) - l_3^{((123),(231))}(y,x,z) - l_3^{((132),(231))}(z,x,y) - l_3^{((132),(231))}(y,x,z) \\
+l_3^{((123),(132))}(x,y,z) + l_3^{((123),(132))}(x,z,y) + l_3^{((123),(312))}(x,y,z) \\
+  l_3^{((123),(312))}(x,z,y) + l_3^{((123),(321))}(x,y,z) +  l_3^{((123),(321))}(x,z,y)~,
\end{aligned}
\]
    which appear on the right composition. After reordering the entries of these operations, $4$ ternary operations cancel on each side and the remaining $10$ are equal precisely because of the first relation that monodic Lie algebras satisfy. 

    \medskip
    
    \item this associative product defines a group structure, with neutral element $0$ and with the inverse given by 
\[
\begin{aligned}
    x^{-1} &= -x - l_2^{\id}(x,x)  - l_2^{\id} \circ_1 l_2^{\id}(x,x,x) -  \sum_{\substack{\tau \in \mathbb{S}_3 \\ \tau ~~\text{transposition}}}  l_3^{(\tau, \id)}(x,x,x) \\ 
    &= -x - l_2^{\id}(x,x)  - l_2^{\id} \circ_2 l_2^{\id}(x,x,x) -  \sum_{\substack{\tau \in \mathbb{S}_3 \\ \tau ~~\text{transposition}}}  l_3^{(\id, \tau)}(x,x,x) ~, 
\end{aligned}
\]
where the latter two expressions are equal in fact by the relation \ref{eq: relation 3 order}, which evaluated on a single variable simplifies and only involves decorations by transpositions. This is because the two cycles in $\mathbb{S}_3$ are inverse to each other and have a $+1$ signature. 
\end{itemize}
\end{Example}

\subsection{Formal moduli problems and integration theory} The $\infty$-category of formal moduli problems defined in terms of $\mathbb{E}_\infty$-algebras is equivalent to the $\infty$-category of partition $\mathcal{L}_\infty$-algebras, as first shown in \cite{brantnermathew}. We work in the set-up of \cite{deuxiemepapier}. The goal of this subsection is to show that under a homotopy completeness assumption, applying the integration functor $\mathcal{R}$ to a partition $\mathcal{L}_\infty$-algebra recovers its associated formal moduli problem. 

\begin{Definition}[Homotopy complete algebra]
Let $\mathfrak{g}$ be a partition $\mathcal{L}_\infty$-algebra. It is \textit{homotopy complete} if the derived unit of adjunction 
\[
\mathbb{L}\eta_{\mathfrak{g}}: \mathfrak{g} \qi \mathrm{Res}~\mathbb{L}\mathrm{Ab}(\mathfrak{g})
\]
of Proposition \ref{Prop: adjunction entre absolues et pas absolues} is a quasi-isomorphism. 
\end{Definition}

\begin{Remark}
This definition is a divided powers analogue of the definition of homotopy completeness given by Harper and Hess in \cite{HarperHess}. 
\end{Remark}

\begin{Example}
Partition $\mathcal{L}_\infty$-algebras which admit as a quasi-free model generated by a finite dimensional complex $V$ in degrees $\leq 0$ are homotopy complete, as shown in \cite{deuxiemepapier}. These algebras correspond to representable formal moduli problems. 
\end{Example}

In order to present the formal moduli problem associated to a homotopy complete partition $\mathcal{L}_\infty$-algebra, we follow the main ideas of \cite[Section 3.3]{deuxiemepapier}. Formal moduli problems can also be defined as contravariant functors from the $\infty$-category of coArtinian $\Omega\mathrm{B}\mathcal{E}^{\mathrm{nu}}$-coalgebras to spaces that preserve certain pushouts. Indeed, there is an equivalence, given by linear duality, between the $\infty$-category of coArtinian $\Omega\mathrm{B}\mathcal{E}^{\mathrm{nu}}$-coalgebras and the $\infty$-category of Artinian $\Omega\mathrm{B}\mathcal{E}^{\mathrm{nu}}$-algebras. Notice that $\Omega\mathrm{B}\mathcal{E}^{\mathrm{nu}}$-algebras are a model for $\mathbb{E}_\infty$-algebras over $\kk$.

\begin{Proposition}\label{Prop: gives back the fmp}
Let $\mathfrak{g}$ be a partition $\mathcal{L}_\infty$-algebra which is homotopy complete. The functor
\[
\begin{tikzcd}[column sep=1.5pc,row sep=0.1pc]
\psi(\mathfrak{g}): \Omega\mathrm{B}\mathcal{E}^{\mathrm{nu}}\text{-}\mathsf{coalg} \arrow[r]
&\mathsf{sSets} \\
C \arrow[r,mapsto]
&\mathcal{R}(\mathrm{hom}(C \oplus \kk, \mathbb{L}\mathrm{Ab}(\mathfrak{g})))~.
\end{tikzcd}
\]
\vspace{0.1pc}

when restricted to coArtinian $\Omega\mathrm{B}\mathcal{E}^{\mathrm{nu}}$-coalgebras, presents the formal moduli problem associated to $\mathfrak{g}$, where $\mathrm{hom}(C \oplus \kk,\mathbb{L}\mathrm{Ab}(\mathfrak{g}))$ refers to the convolution algebra construction of subsection \ref{subsection: convolution} between $C \oplus \kk$ (cofreely added a strict counit to $C$) and and the derived homotopy completion $\mathbb{L}\mathrm{Ab}(\mathfrak{g}))$ of $\mathfrak{g}$. 
\end{Proposition}

\begin{proof}
By \cite[Lemma 9]{deuxiemepapier}, the formal moduli problem associated to a partition $\mathcal{L}_\infty$-algebra is given by the functor 

\[
\begin{tikzcd}[column sep=3.5pc,row sep=0.1pc]
\Psi: \mathsf{part}~\mathcal{L}_\infty\text{-}\mathsf{alg}~[\mathrm{Q.iso}^{-1}] \arrow[r,]
&\mathsf{Fun}\left(\mathsf{coArt}~\Omega\mathrm{B}\mathcal{E}^{\mathrm{nu}}\text{-}\mathsf{coalg},\mathcal{S}\right) \\
\mathfrak{g} \arrow[r,mapsto]
&\left[C \mapsto \mathrm{Map}_{\mathsf{part}~\mathcal{L}_\infty\text{-}\mathsf{alg}~[\mathrm{Q.iso}^{-1}]}\left(\mathrm{Res}~\widehat{\Omega}^\flat_\iota C,\mathfrak{g}\right) \right]~,
\end{tikzcd}
\]

at the $\infty$-categorical level. If $C$ is coArtinian, then $\mathrm{Res}~\widehat{\Omega}_\iota C$ is homotopy complete by \cite[Theorem 4]{deuxiemepapier}. Hence 
\[
\mathrm{Map}_{\mathsf{part}~\mathcal{L}_\infty\text{-}\mathsf{alg}~[\mathrm{Q.iso}^{-1}]}\left(\mathrm{Res}~\widehat{\Omega}^\flat_\iota C,\mathfrak{g}\right) \simeq \mathrm{Map}_{\mathsf{abs}~\mathsf{part}~\mathcal{L}_\infty\text{-}\mathsf{alg}~[\mathrm{W}^{-1}]}\left(\widehat{\Omega}^\flat_\iota C,\mathbb{L}\mathrm{Ab}(\mathfrak{g})\right)~,
\]
since $\mathfrak{g}$ is supposed to be homotopy complete as well. Therefore it suffices to give an explicit model for this mapping space. By Proposition \ref{Prop: curved non curved comparison}, we can compute this mapping space in curved absolute partition $\mathcal{L}_\infty$-algebras. A model for this mapping space is given by

\hspace{-1pc}
\begin{align*}
\mathrm{Hom}_{\Omega \mathrm{B}^{\mathrm{s.a}}\mathcal{E}\text{-}\mathsf{coalg}}\left((C\oplus \kk) \otimes C_*^c(\Delta^\bullet),\widehat{\mathrm{B}}_\iota\mathbb{L}\mathrm{Ab}(\mathfrak{g})\right) 
& \cong \mathrm{Hom}_{\mathsf{dg}~\Omega \mathrm{B}^{\mathrm{s.a}} \mathcal{E}\text{-}\mathsf{coalg}}\left(C_*^c(\Delta^\bullet),\left\{C \oplus \kk, \widehat{\mathrm{B}}_\iota\mathbb{L}\mathrm{Ab}(\mathfrak{g})\right\} \right) \\
& \cong \mathrm{Hom}_{\mathsf{dg}~\Omega \mathrm{B}^{\mathrm{s.a}} \mathcal{E}\text{-}\mathsf{coalg}}\left(C_*^c(\Delta^\bullet),\widehat{\mathrm{B}}_\iota\mathrm{hom}(C \oplus \kk,\mathbb{L}\mathrm{Ab}(\mathfrak{g})) \right) \\
& \cong \mathcal{R}(\mathrm{hom}(C \oplus \kk,\mathbb{L}\mathrm{Ab}(\mathfrak{g})))~,
\end{align*}
\vspace{0.1pc}

since $C \otimes C_*^c(\Delta^\bullet)$ is a Reedy cofibrant resolution of $C$ as	a $\Omega\mathrm{B}\mathcal{E}^{\mathrm{nu}}$-coalgebra. Finally, any quasi-isomorphism $\mathfrak{g} \qi \mathfrak{g}'$ induces a weak equivalence $\mathbb{L}\mathrm{Ab}(\mathfrak{g}) \qi \mathbb{L}\mathrm{Ab}(\mathfrak{g}')$, and the convolution construction preserves weak equivalences in the second variable. 
\end{proof}

\begin{Remark}
Like in the characteristic zero case of \cite[Section 7]{CCN19}, it should be possible to remove the homotopy completeness assumption on $\mathfrak{g}$. The idea is that, given a general partition $\mathcal{L}_\infty$-algebra $\mathfrak{g}$, a convolution construction between $\mathfrak{g}$ and a strictly coArtinian coalgebra $C$ is necessarily nilpotent, and thus one can apply the integration functor to it.
\end{Remark}

\begin{Remark}[Recovering general fmps]
Let $\C$ be a $0$-reduced quasi-planar conilpotent dg cooperad. We define formal moduli problems of Artinian $\Omega \C$-algebras in an analogous way. Furthermore, if $\C$ satisfies a homological condition called \textit{temperedness}, then this $\infty$-category is equivalent to the $\infty$-category of dg $\C^*$-algebras localized at quasi-isomorphisms. We refer to \cite{deuxiemepapier} for more details. 

\medskip

If $\C$ is quasi-planar, the dg operad $\Omega \C$ admits a canonical $\mathcal{E}$-comodule structure $\Omega\C \longrightarrow \Omega\C \otimes \mathcal{E}$, as described in \cite[Section 2]{bricevictor}. This gives an enrichment $\left\{-,-\right\}$ of dg $\Omega \C$-coalgebras over dg $\mathcal{E}$-coalgebras, and an analogue convolution construction $\mathrm{hom}(C,\mathfrak{g})$ between a dg $\Omega \C$-coalgebras $C$ and a dg $\C$-algebra $\mathfrak{g}$. This convolution is naturally a curved absolute partition $\mathcal{L}_\infty$-algebra. Applying the integration functor to this convolution algebra gives a model for formal moduli problems encoded by homotopy complete dg $\C^*$-algebras. 
\end{Remark}

\subsection{Comparison results}\label{subsection: comparison results}
We discuss comparison results. First, we explain how similar constructions can be made using other models. Secondly, since our constructions work over any field $\kk$, we compare them to known constructions in characteristic zero. In this case, we show that the integration functor constructed in this paper is naturally weakly equivalent to the one constructed in \cite{lucio2022integration}. This allows us to compare it to \cite{robertnicoud2020higher} or to \cite{Getzler09}. In particular, this gives us a way to compare the formula in Theorem \ref{thm: horn-filler formula (BCH)} with the classical Baker--Campbell--Hausdorff formula. 

\medskip

\textbf{Another quasi-planar model.} Let $\mathcal{S}urj$ be the unital surjections operad of McClure and Smith, defined in \cite{McClureSmith}. It is a $\mathbb{S}$-cofibrant resolution of the commutative operad. As mentioned in Remark \ref{Rmk: change of models for partition Lie algebras}, the category of dg $\Omega \mathcal{S}urj^*$-algebras localized at quasi-isomorphisms also present the $\infty$-category of partition Lie algebras. This is the model for partition $\mathcal{L}_\infty$-algebras considered in \cite[Definition 4.46]{pdalgebras}. 

\begin{lemma}\label{lemma: Surj quasi-planar}
The conilpotent dg cooperad $\mathrm{B}^{\mathrm{s.a}}\mathcal{S}urj$ is quasi-planar. 
\end{lemma}

The proof of the above lemma is in fact trickier than that of Lemma \ref{lemma: quasi-planar}, it will appear a the forthcoming paper \cite{najib}. The cellular chains functor $C_*^c(-)$ has a natural $\mathcal{S}urj$-coalgebra structure, constructed in \cite{McClureSmith} and in \cite{BergerFresse04}. Using this, one can construct an integration functor for curved $\mathrm{B}^{\mathrm{s.a}}\mathcal{S}urj$-algebras and Theorems \ref{thm: triangle commutatif} and \ref{thm: propriétés de l'intégration} also hold in that context. Moreover, there is a quasi-morphism of dg operads 
\[
\mathcal{TR}: \mathcal{E} \qi \mathcal{S}urj
\]
called the \textit{table reduction morphism}, constructed in \cite{BergerFresse04}. Using this quasi-isomorphism and Lemma \ref{lemma: Surj quasi-planar} together with the general theory of \cite{bricevictor}, it is straightforward to compare the constructions performed so far with their analogues for this other model. (For a blueprint of these arguments, see below the comparisons that hold in characteristic zero). 

\medskip

Let us discuss the differences between these two choices. The $\mathcal{E}$-coalgebra structure of $C_*^c(-)$ is given by pulling back the $\mathcal{S}urj$-coalgebra structure along the table reduction morphism. As explained in the proofs, the formulas in Proposition \ref{Prop: coalgebra structure of the interval} and Lemma \ref{lemma: coalgebra structure on the reduced spheres} are obtained from simpler formulas for the $\mathcal{S}urj$-coalgebra structure. This also results in simpler formulas for the analogues of Proposition \ref{prop: general formula for higher bch} and Theorem \ref{thm: horn-filler formula (BCH)}. The key difference, however, is that the dg operad $\mathcal{S}urj$ is not a Hopf operad. Thus, \textit{a priori}, one does not have a tensor product of $\Omega \mathrm{B}^{\mathrm{s.a}}\mathcal{S}urj$-coalgebras or a convolution structure. These are crucial in many proofs. For example, when we construct cylinder objects by tensoring with $C_*^c(\Delta^1)$. This can be somewhat repaired since $\Omega \mathrm{B}^{\mathrm{s.a}}\mathcal{S}urj$-coalgebras are in fact tensored over $\mathcal{E}$-coalgebras, but still things become more cumberstone. 

\medskip

\textbf{Comparisons in characteristic zero.} There is a canonical quasi-isomorphism $\varphi: \mathcal{E} \qi u\mathcal{C}om$, given by the augmentation map of $\kk[\mathbb{S}_n]$ in degree $0$, for every $n \geq 0$. When $\kk$ is a field of characteristic zero, this induces a quasi-isomorphism of cofibrant dg operads $\varphi: \Omega \mathrm{B}^{\mathrm{s.a}}\mathcal{E} \qi \Omega \mathrm{B}^{\mathrm{s.a}} u\mathcal{C}om$. Notice that outside of characteristic zero, $\Omega \mathrm{B}^{\mathrm{s.a}} u\mathcal{C}om$ is not cofibrant any more. This morphism induces the following commutative square of Quillen equivalences
\[
\begin{tikzcd}[column sep=5pc,row sep=5pc]
u\mathcal{EE}_\infty\text{-}\mathsf{coalg} \arrow[r,"\widehat{\Omega}_\iota"{name=B},shift left=1.1ex] \arrow[d,"\mathrm{Coind}_{\Omega \mathrm{B}^{\mathrm{s.a}}\varphi} "{name=SD},shift left=1.1ex ]
&\mathsf{curv}~\mathsf{abs}~\mathcal{L}^\pi_\infty\textsf{-}\mathsf{alg}^{\mathsf{qp}\text{-}\mathsf{comp}} \arrow[d,"\mathrm{Res}_{\mathrm{B}^{\mathrm{s.a}}\varphi}"{name=LDC},shift left=1.1ex ] \arrow[l,"\widehat{\mathrm{B}}_\iota"{name=C},,shift left=1.1ex]  \\
u\mathcal{CC}_\infty\text{-}\mathsf{coalg} \arrow[r,"\widehat{\Omega}^2_\iota "{name=CC},shift left=1.1ex]  \arrow[u,"\mathrm{Res}_{\Omega \mathrm{B}^{\mathrm{s.a}}\varphi}"{name=LD},shift left=1.1ex ]
&\mathsf{curv}~\mathsf{abs}~\mathcal{L}_\infty\textsf{-}\mathsf{alg}^{\mathsf{comp}}~, \arrow[l,"\widehat{\mathrm{B}}^2_\iota"{name=CB},shift left=1.1ex] \arrow[u,"\mathrm{Ind}_{\mathrm{B}^{\mathrm{s.a}}\varphi}"{name=TD},shift left=1.1ex] \arrow[phantom, from=SD, to=LD, , "\dashv" rotate=0] \arrow[phantom, from=C, to=B, , "\dashv" rotate=-90]\arrow[phantom, from=TD, to=LDC, , "\dashv" rotate=0] \arrow[phantom, from=CC, to=CB, , "\dashv" rotate=-90]
\end{tikzcd}
\] 
where $u\mathcal{CC}_\infty = \Omega \mathrm{B}^{\mathrm{s.a}} u\mathcal{C}om$ is the model for counital cocommutative up to homotopy coalgebras used in \cite{lucio2022integration}. We add the superscript $2$ to distinguish the between two complete bar-cobar adjunctions. Let us denote by $\mathcal{R}^{\mathrm{zero}}$ the integration functor constructed in \cite{lucio2022integration}. 

\begin{Proposition}\label{prop: comparison characteristic zero}
Let $\kk$ be a field of characteristic zero. Let $\mathfrak{g}$ be a qp-complete curved absolute partition $\mathcal{L}_\infty$-algebra and let $\mathfrak{h}$ be a complete curved absolute $\mathcal{L}_\infty$-algebra. The following simplicial sets are naturally weakly equivalent
\[
\mathcal{R}(\mathfrak{g}) \simeq \mathcal{R}^{\mathrm{zero}}(\mathrm{Res}_{\mathrm{B}^{\mathrm{s.a}}\varphi}(\mathfrak{g})) \quad \text{and} \quad  \mathcal{R}(\mathbb{L}\mathrm{Ind}_{\mathrm{B}^{\mathrm{s.a}}\varphi}(\mathfrak{h})) \simeq \mathcal{R}^{\mathrm{zero}}(\mathfrak{h})~. 
\]
\end{Proposition}

\begin{proof}
Let $\kk$ be the ground field with its canonical $u\mathcal{CC}_\infty$-coalgebra structure. It is straightforward to check that $\mathrm{Res}_{\Omega\mathrm{B}^{\mathrm{s.a}}\varphi}(\kk)$ is isomorphic to $\kk$ with its canonical $u\mathcal{EE}_\infty$-coalgebra structure. Let $\mathfrak{g}$ be a curved absolute partition $\mathcal{L}_\infty$-algebra, the simplicial set $\mathcal{R}^{\mathrm{zero}}(\mathfrak{g})$ has the following homotopy type
\[
\mathcal{R}(\mathfrak{g}) \simeq \mathrm{Map}_{u\mathcal{EE}_\infty\text{-}\mathsf{coalg}}\left(\kk, \widehat{\mathrm{B}}_\iota(\mathfrak{g})\right) \simeq \mathrm{Map}_{u\mathcal{EE}_\infty\text{-}\mathsf{coalg}}\left(\mathrm{Res}_{\Omega \mathrm{B}^{\mathrm{s.a}}\varphi}(\kk), \widehat{\mathrm{B}}_\iota(\mathfrak{g})\right)
\]
which is naturally equivalent via adjunction to 
\[
\mathrm{Map}_{u\mathcal{CC}_\infty\text{-}\mathsf{coalg}}\left(\kk, \mathbb{R}\mathrm{Coind}_{\Omega \mathrm{B}^{\mathrm{s.a}}\varphi}\widehat{\mathrm{B}}_\iota(\mathfrak{g})\right) \simeq \mathrm{Map}_{u\mathcal{CC}_\infty\text{-}\mathsf{coalg}}\left(\kk, \widehat{\mathrm{B}}^2_\iota(\mathfrak{g})\right)\simeq \mathcal{R}^{\mathrm{zero}}(\mathrm{Res}_{\mathrm{B}^{\mathrm{s.a}}\varphi}(\mathfrak{g}))~.
\]
The second equivalence follows from the first one, using the fact that the derived unit of adjunction $\mathfrak{h} \qi \mathrm{Res}_{\mathrm{B}^{\mathrm{s.a}}\varphi} \mathbb{L}\mathrm{Ind}_{\mathrm{B}^{\mathrm{s.a}}\varphi} \mathfrak{h}$ is a natural weak equivalence. 
\end{proof}

Let us describe the action of the restriction functor $\mathrm{Res}_{\mathrm{B}^{\mathrm{s.a}}\varphi}$ on a (curved) absolute partition $\mathcal{L}_\infty$-algebra $\mathfrak{g}$. It does not change its underlying chain complex/pre-differential module. The structure operations of the (curved) absolute $\mathcal{L}_\infty$-algebra $\mathrm{Res}_{\mathrm{B}^{\mathrm{s.a}}\varphi}(\mathfrak{g})$ are given by 
\[
\left\{l_n = \displaystyle \sum_{\sigma \in \mathbb{S}_n} l_n^{\sigma}\right\}~, 
\]
for all $n \geq 0$. All the operations $l_n^w$ where the degree of $w$ is $\geq 1$ are sent to $0$. One can check that these new operations $\{l_n\}$ are indeed symmetric of degree $-1$ and that they satisfy the axioms of an absolute $\mathcal{L}_\infty$-algebra. In particular, if $\mathfrak{g}$ is a nilpotent partition $\mathcal{L}_\infty^\pi$-algebra concentrated in degree $1$, then $\mathrm{Res}_{\mathrm{B}^{\mathrm{s.a}}\varphi}(\mathfrak{g})$ is in fact a nilpotent Lie algebra concentrated in degree $1$. 

\begin{theorem}\label{thm: gives back bch in characteristic zero}
Let $\kk$ be a field of characteristic zero. Let $\mathfrak{g}$ be a nilpotent partition $\mathcal{L}_\infty$-algebra concentrated in degree $1$. There is an isomorphism of groups 

\[
\big(\mathfrak{g},\mathrm{HF}_0(-,-),0\big) \cong \big(\mathfrak{g},\mathrm{BCH}(-,-),0\big)~,
\]
\vspace{0.1pc}

between the nilpotent group obtained with the horn-filler formula of Theorem \ref{thm: horn-filler formula (BCH)} and the exponential group obtained from the underlying Lie algebra of $\mathfrak{g}$ using the Baker--Campbell--Hausdorff formula.  
\end{theorem}

\begin{proof}
Follows directly from Proposition \ref{prop: comparison characteristic zero}. Indeed, since the simplicial sets $\mathcal{R}^{\mathrm{zero}}(\mathrm{Res}_{\mathrm{B}^{\mathrm{s.a}}\varphi}(\mathfrak{g}))$ and $\mathcal{R}(\mathfrak{g})$ are weakly equivalent, their homotopy groups, namely their $\pi_1$, are isomorphic.  
\end{proof}

\begin{Remark}
Notice that if $\mathfrak{h}$ is a nilpotent Lie algebra concentrated in degree $1$, then the absolute partition $\mathcal{L}_\infty$-algebra $\mathbb{L}\mathrm{Ind}_{\mathrm{B}^{\mathrm{s.a}}\varphi}(\mathfrak{h})$ need not be concentrated in degree $1$. This coincides with the observation made in Remark \ref{Rmk: degree 1 and k-points of unipotent groups}. However, by Proposition \ref{prop: comparison characteristic zero}, the space $\mathcal{R}(\mathbb{L}\mathrm{Ind}_{\mathrm{B}^{\mathrm{s.a}}\varphi}(\mathfrak{h}))$ is still the classifying space of the exponential group of $\mathfrak{h}$, and its group structure can also be described up to isomorphism by Theorem \ref{thm: horn-filler formula (BCH)}. 
\end{Remark}

\begin{Example}\label{Example: comparison for 2-nilpotent associative algebras}
In general, Theorem \ref{thm: gives back bch in characteristic zero} does not give an explicit construction of for the isomorphism between these two groups. However, in the case of Example \ref{Example: 2-nilpotent associative algebras}, an explicit isomorphism can be constructed. 

\medskip

Recall that the data of a $2$-nilpotent partition $\mathcal{L}_\infty$-algebra concentrated in degree $1$ is equivalent to the data of a $\kk$-vector space $\mathfrak{g}$ together with a binary map 
\[
l_2: \mathfrak{g} \otimes \mathfrak{g} \longrightarrow \mathfrak{g}~,
\]
such that any composition of $l_2$ with itself is zero, i.e: $l_2(l_2(-,-),-) = l_2(-,l_2(-,-)) = 0$. The horn-filler product in this case is given by $\mathrm{HF}_0(x,y) = x + y - l_2(x,y)$ and it can be checked by hand that it defines a group structure on $\mathfrak{g}$.

\medskip

Given this data, one can also define a Lie bracket on $\mathfrak{g}$ as follows 
\[
[x,y] = l_2(y,x) - l_2(x,y)~. 
\]
The Baker--Campbell--Hausdorff formula is then given by 
\[
\mathrm{BCH}(x,y) = x + y + \frac{1}{2}[x,y]~. 
\]
It can be checked by hand that the map $f(x) = x + \frac{l_2(x,x)}{2}$ defines a group isomorphism between the group defined with the horn-filler product and the one defined with the Baker--Campbell--Hausdorff formula. This later construction, however, only makes sense with $2$ is invertible in the ground field $\kk$. 
\end{Example}

\section{Lie-type models in $p$-adic homotopy theory}\label{Section: Lie models}
We show that (curved) absolute partition $\mathcal{L}_\infty$-algebras, with the transferred model structure from their Koszul dual coalgebras, model the $p$-adic homotopy types of connected finite type nilpotent spaces. This is done by a dualizing Mandell's $p$-adic models constructed in \cite{Mandell}. Then, we construct an intrinsic model structure on (curved) absolute partition $\mathcal{L}_\infty$-algebras, transferred from simplicial sets, and adapted to the study of $p$-adic homotopy types.  Building on the work done in the previous sections, we give an explicit description of the homotopy groups of the $p$-completion of a pointed connected finite type nilpotent space in terms of its absolute partition $\mathcal{L}_\infty$ model. Finally, we construct (curved) absolute partition $\mathcal{L}_\infty$ models for $p$-adic mapping spaces.

\subsection{Comparing derived units of adjunctions}
From now on, we assume the base field $\kk$ to be an algebraic closed field of a characteristic $p >0$. For example, the algebraic closure $\overline{\mathbb{F}_p}$ of $\mathbb{F}_p$. Furthermore, we consider the category of simplicial sets endowed with the $\mathbb{F}_p$-local model structure constructed in \cite{Bousfield75}, where cofibrations are given by monomorphisms and where weak equivalences are given by morphisms $f: X \longrightarrow Y$ such that $\mathrm{H}_*(f,\mathbb{F}_p): \mathrm{H}_*(X,\mathbb{F}_p) \longrightarrow \mathrm{H}_*(Y,\mathbb{F}_p)$ is an isomorphism.

\medskip

\textbf{Cochains adjunction:} There is an adjunction

\[
\begin{tikzcd}[column sep=5pc,row sep=2.5pc]
\mathsf{sSet} \arrow[r, shift left=1.5ex, "C^*_c(-)"{name=A}]
&\mathsf{dg}~\Omega \mathrm{B}^{\mathrm{s.a}} \mathcal{E}\textsf{-}\mathsf{alg}^{\mathsf{op}}~, \arrow[l, shift left=.75ex, "\mathrm{U}"{name=C}] \arrow[phantom, from=A, to=C, , "\dashv" rotate=-90]
\end{tikzcd}
\]

where $C^*_c(-)$ is the cellular cochains functor, endowed with a $\Omega \mathrm{B}^{\mathrm{s.a}} \mathcal{E}$-algebra structure, which is given by pulling back the $\mathcal{E}$-algebra structure along the morphism $\Omega \mathrm{B}^{\mathrm{s.a}} \mathcal{E} \longrightarrow \mathcal{E}$.

\medskip 

\textbf{Finite type nilpotent spaces.} We recall what finite type nilpotent spaces are and we state the main theorem of this subsection.

\begin{Definition}[Finite type simplicial set]
Let $X$ be a simplicial set. It is said to be of \textit{finite type} if the homology groups $\mathrm{H}_n(X,\mathbb{F}_p)$ are finite dimensional for all $n \geq 0$.
\end{Definition}

\begin{Definition}[Nilpotent simplicial set]
Let $X$ be a simplicial set. It is said to be \textit{nilpotent} if for every $0$-simplex $\alpha$ it satisfies the following conditions:

\begin{enumerate}
\item The group $\pi_1(X,\alpha)$ is nilpotent.

\medskip

\item The $\pi_1(X,\alpha)$-module $\pi_n(X,\alpha)$-module is a nilpotent $\pi_1(X,\alpha)$-module.
\end{enumerate}
\end{Definition}

\begin{theorem}[\cite{Mandell}]
Let $X$ be a connected finite type nilpotent simplicial set. There derived unit of adjunction 
\[
\eta_X: X \longrightarrow \mathbb{R}\mathrm{U}C^*_c(X) 
\]

is an $\mathbb{F}_p$-equivalence. 
\end{theorem}

\begin{proof}
The dg operad $\Omega \mathrm{B}^{\mathrm{s.a}} \mathcal{E}$ qualifies as an $\mathcal{E}_\infty$-operad, that is, a $\mathbb{S}$-projective resolution of the commutative operad. Thus the main theorem of \cite{Mandell} applies. 
\end{proof}

\begin{Proposition}\label{prop: n-ieme equivalence}
Let $X$ be a finite type simplicial set. There is a weak equivalence of simplicial sets
\[
\mathcal{R}\mathcal{L}(X) \simeq \mathbb{R}\mathrm{U}C^*_c(X)~,
\]

which is natural on the subcategory of finite type simplicial sets.
\end{Proposition}

\begin{proof}
There is an equivalence 
\[
\mathbb{R}\mathrm{U}C^*_c(X) \simeq \mathrm{Hom}_{\Omega \mathrm{B}^{\mathrm{s.a}} \mathcal{E} \text{-}\mathsf{alg}}(\Omega_\iota \mathrm{B}_\iota C_c^*(X), C_c^*(\Delta^\bullet))~.
\]

The following square of Quillen adjunctions

\[
\begin{tikzcd}[column sep=5pc,row sep=5pc]
\Omega \mathrm{B}^{\mathrm{s.a}} \mathcal{E}\text{-}\mathsf{alg}^{\mathsf{op}} \arrow[r,"\mathrm{B}_\iota^{\mathsf{op}}"{name=B},shift left=1.1ex] \arrow[d,"(-)^\circ "{name=SD},shift left=1.1ex ]
&\mathsf{curv}~\mathrm{B}^{\mathsf{s.a}}\mathcal{E}\text{-}\mathsf{coalg}^{\mathsf{op}} \arrow[d,"(-)^*"{name=LDC},shift left=1.1ex ] \arrow[l,"\Omega_\iota^{\mathsf{op}}"{name=C},,shift left=1.1ex]  \\
\Omega \mathrm{B}^{\mathrm{s.a}} \mathcal{E}\text{-}\mathsf{coalg} \arrow[r,"\widehat{\Omega}_\iota "{name=CC},shift left=1.1ex]  \arrow[u,"(-)^*"{name=LD},shift left=1.1ex ]
&\mathsf{curv}~\mathrm{B}^{\mathsf{s.a}}\mathcal{E}\text{-}\mathsf{alg}^{\mathsf{qp}\text{-}\mathsf{comp}} ~, \arrow[l,"\widehat{\mathrm{B}}_\iota"{name=CB},shift left=1.1ex] \arrow[u,"(-)^\vee"{name=TD},shift left=1.1ex] \arrow[phantom, from=SD, to=LD, , "\dashv" rotate=0] \arrow[phantom, from=C, to=B, , "\dashv" rotate=-90]\arrow[phantom, from=TD, to=LDC, , "\dashv" rotate=0] \arrow[phantom, from=CC, to=CB, , "\dashv" rotate=-90]
\end{tikzcd}
\] 

commutes. The left-hand side categories are endowed with a transferred model structure from dg modules and the right-hand side with a transferred model structure using each of these bar-cobar adjunctions. This follows from \cite[Section 8]{bricevictor}, using the fact that $\mathrm{B}^{\mathsf{s.a}}\mathcal{E}$ is a quasi-planar curved conilpotent cooperad as shown in Lemma \ref{lemma: quasi-planar}. Since $X$ is a finite type simplicial set, the homology of $\Omega_\iota \mathrm{B}_\iota C_c^*(X)$ is degree-wise finite dimensional and bounded above. Therefore the generalized Sweedler dual $(-)^\circ$ is homotopically fully-faithful and we get 
\[
\mathrm{Hom}_{\Omega \mathrm{B}^{\mathrm{s.a}} \mathcal{E} \text{-}\mathsf{alg}}(\Omega_\iota\mathrm{B}_\iota C_c^*(X), C_c^*(\Delta^\bullet)) \simeq \mathrm{Hom}_{\Omega \mathrm{B}^{\mathrm{s.a}} \mathcal{E} \text{-}\mathsf{coalg}}(C_*^c(\Delta^\bullet), \widehat{\mathrm{B}}_\iota \widehat{\Omega}_\iota C_*^c(X) )~,
\]
which concludes the proof. See \cite[Section 3]{lucio2022integration} for an analogue of this proof in characteristic zero.
\end{proof}

\begin{theorem}\label{thm: modèles d'homotopie rationnel type fini}
Let $X$ be a connected finite type nilpotent simplicial set. The unit of adjunction
\[
\eta_X: X \qi \mathcal{R}\mathcal{L}(X) 
\]

is an $\mathbb{F}_p$-equivalence. 
\end{theorem}

\begin{proof}
Follows directly from Proposition \ref{prop: n-ieme equivalence}.
\end{proof}

\begin{Remark}
Let $X$ connected finite type simplicial set. The unit of adjunction
\[
\eta_X: X \qi \mathcal{R}\mathcal{L}(X) 
\]

is weakly equivalent to the Bousfield-Kan $p$-completion. 
\end{Remark}

\begin{Corollary}
Let $X$ be a pointed connected finite type nilpotent simplicial set. The unit of adjunction
\[
\eta_X: X \qi \mathcal{R}_*\mathcal{L}_*(X) 
\]

is an $\mathbb{F}_p$-equivalence.
\end{Corollary}

\begin{proof}
Follows from Theorem \ref{thm: modèles d'homotopie rationnel type fini} and Proposition \ref{Prop: commutativity of the pointing adjunctions}. 
\end{proof}

\begin{Remark}
In \cite{bachmann2024}, Bachmann and Burklund showed that the chains functor from spaces to $\mathbb{E}_\infty$-coalgebras over $\kk$ (algebraically closed field of characteristic $p >0$) is homotopically fully faithful on all nilpotent spaces, without any pointed, connected or finite type assumption. If $\Omega \mathrm{B}^{\mathrm{s.a}} \mathcal{E}$-coalgebras rectify $\mathbb{E}_\infty$-coalgebras over $\kk$, that is, if the $\infty$-category of $\Omega \mathrm{B}^{\mathrm{s.a}} \mathcal{E}$-coalgebras localized at quasi-isomorphisms is equivalent to the $\infty$-category of $\mathbb{E}_\infty$-coalgebras over $\kk$, then it would follow that Theorem \ref{thm: modèles d'homotopie rationnel type fini} holds for all nilpotent spaces as well.
\end{Remark}

\subsection{Minimal models}
In this subsection, we construct minimal resolutions of qp-complete curved absolute $\mathcal{L}^\pi_\infty$-algebras with respect to transferred weak equivalences. 

\begin{Definition}[Minimal model]
Let $\mathfrak{g}$ be a qp-complete curved absolute $\mathcal{L}^\pi_\infty$-algebra. A \textit{minimal model} $(V,\varphi_{d_V},\psi_V)$ amounts to the data of 

\begin{enumerate}
\item A graded module $V$ together with a map
\[
\varphi_{d_V}: V \longrightarrow \prod_{n \geq 0} \widehat{\Omega}^{\mathrm{s.a}}\mathcal{E}^*(n)^{(\geq 1)} ~\widehat{\otimes}_{\mathbb{S}_n} ~ V^{\otimes n}~
\]

which lands on elements of weight greater or equal to $1$, such that the induced derivation $d_V$ satisfies
\[
d_V^2 = l_2^{(12)}(l_0,g) + l_2^{(21)}(l_0,g)~. 
\]

\item A weak equivalence of curved absolute $\mathcal{L}^\pi_\infty$-algebras 
\[
\psi_V: \left(\prod_{n \geq 0} \widehat{\Omega}^{\mathrm{s.a}}\mathcal{E}^*(n) ~\widehat{\otimes}_{\mathbb{S}_n} ~ V^{\otimes n},d_V \right) \qi \mathfrak{g}~.
\]
\end{enumerate}
\end{Definition}

\begin{Remark}
Given a graded module $V$, the data of the derivation $d_V$ amounts to a $u\mathcal{EE}_\infty$-coalgebra structure on $V$. Thus the data $(V,\varphi_{d_V})$ above amounts to the data of a \textit{minimal} $u\mathcal{EE}_\infty$-coalgebra, that is, a $u\mathcal{EE}_\infty$-coalgebra whose underlying differential is zero.
\end{Remark}

\begin{Proposition}
Let $(V,\varphi_{d_V},\psi_V)$ and $(W,\varphi_{d_W},\psi_W)$ be two minimal models of a qp-complete curved absolute $\mathcal{L}^\pi_\infty$-algebra $\mathfrak{g}$. There is an isomorphism of graded modules $V \cong W$. 
\end{Proposition}

\begin{proof}
If $(V,\varphi_{d_V},\psi_V)$ and $(W,\varphi_{d_W},\psi_W)$ are two minimal models of $\mathfrak{g}$, then, by definition, there exists a weak equivalence of curved absolute $\mathcal{L}^\pi_\infty$-algebras
\[
\left(\prod_{n \geq 0} \widehat{\Omega}^{\mathrm{s.a}}\mathcal{E}^*(n) ~\widehat{\otimes}_{\mathbb{S}_n} ~ V^{\otimes n},d_V \right) \qi \left( \prod_{n \geq 0} \widehat{\Omega}^{\mathrm{s.a}}\mathcal{E}^*(n) ~\widehat{\otimes}_{\mathbb{S}_n} ~ W^{\otimes n},d_W\right)~.
\]

This data amounts to an $\infty$-quasi-isomorphism between the $u\mathcal{EE}_\infty$-coalgebras $V$ and $W$. Since their differentials are trivial, this data is an $\infty$-isomorphism, and in particular, an isomorphism of graded modules between $V$ and $W$. See \cite[Section 7]{bricevictor} for the statements related to $\infty$-morphisms of coalgebras in positive characteristic.
\end{proof}

\begin{Proposition}\label{prop; generateurs du model minimal, homologie de la Bar}
Let $\mathfrak{g}$ be a qp-complete curved absolute $\mathcal{L}^\pi_\infty$-algebra. It admits a minimal model where the graded module of generators is given by the homology of its complete bar construction.
\end{Proposition}

\begin{proof}
Since we are working over a field, we can choose a contraction between $\widehat{\mathrm{B}}_\iota \mathfrak{g}$ and its homology. We can then apply the homotopy transfer theorem of \cite[Section 7]{bricevictor} in order to transfer an $\infty$-quasi-isomorphic $u\mathcal{EE}_\infty$-coalgebra structure onto its homology. This gives a weak equivalence of qp-complete curved absolute $\mathcal{L}^\pi_\infty$-algebras
\[
\widehat{\Omega}_\iota\mathrm{H}_*\left(\widehat{\mathrm{B}}_\iota \mathfrak{g}\right) \qi \widehat{\Omega}_\iota\widehat{\mathrm{B}}_\iota \mathfrak{g} \qi \mathfrak{g}~, 
\]
by composing the $\infty$-quasi-isomorphism with the counit of the complete bar-cobar adjunction.
\end{proof}

\begin{Definition}[$p$-adic model]\label{def: p-adic model}
Let $\mathfrak{g}$ be a qp-complete curved absolute $\mathcal{L}^\pi_\infty$-algebra. It is a $p$\textit{-adic model} if there exists a simplicial set $X$ and a zig-zag of weak equivalences of qp-complete curved absolute $\mathcal{L}^\pi_\infty$-algebras 
\[
\mathcal{L}(X) \lqi \cdot \qi \cdots \lqi \cdot \qi \mathfrak{g}~,
\]

and if, furthermore, $\mathcal{R}(\mathfrak{g})$ is $\mathbb{F}_p$-weakly equivalent to $X$. 
\end{Definition}

\begin{Example}
If $\mathfrak{g}$ is weakly equivalent to $\mathcal{L}(X)$, where $X$ is a connected finite type nilpotent simplicial set, then $\mathfrak{g}$ is a $p$-adic model. 
\end{Example}

\begin{Proposition}\label{prop: homologie de la Bar complète}
Let $\mathfrak{g}$ be a qp-complete curved absolute $\mathcal{L}^\pi_\infty$-algebra which is a $p$-adic model. The canonical morphism of $u\mathcal{EE}_\infty$-algebras 
\[
C_*^c(\mathcal{R}(\mathfrak{g})) \qi \widehat{\mathrm{B}}_\iota \mathfrak{g}
\]
is a quasi-isomorphism.
\end{Proposition}

\begin{proof}
The canonical map $C_*^c(\mathcal{R}(\mathfrak{g})) \qi \widehat{\mathrm{B}}_\iota \mathfrak{g}$ is a quasi-isomorphism since its transpose $\mathcal{L}(\mathcal{R}(\mathfrak{g})) \qi \mathfrak{g}$ is a weak equivalence.
\end{proof}

\begin{Corollary}
Let $\mathfrak{g}$ be a qp-complete curved absolute $\mathcal{L}^\pi_\infty$-algebra which is a $p$-adic model for a simplicial set $X$. Then its minimal model is generated by $\mathrm{H}_*(X)$.
\end{Corollary}

\begin{proof}
Immediate by Propositions \ref{prop: homologie de la Bar complète} and \ref{prop; generateurs du model minimal, homologie de la Bar}.
\end{proof}

\subsection{Recovering the homotopy groups of spaces and an intrinsic model structure}
In this subsection, we will focus on pointed simplicial sets and on the adjunction $\mathcal{L}_* \dashv \mathcal{R}_*$. On the one hand, we show how to recover the homotopy groups of the $p$-completion of a pointed connected finite type nilpotent space from its absolute partition $\mathcal{L}_\infty$ model. On the other hand, we use this description to define a new class of weak equivalences of absolute partition $\mathcal{L}_\infty$-algebras and show that one can transfer the model structure on simplicial sets onto absolute partition $\mathcal{L}_\infty$-algebras via the $\mathcal{L}_* \dashv \mathcal{R}_*$. 

\begin{theorem}\label{thm: homotopy groups of X}
Let $X$ be a pointed connected finite type nilpotent space. There is a bijection

\[
\pi_k(X_{\mathbb{F}_p}) \cong \mathrm{rep}(\mathcal{L}_*(X)_k)/\sim_{\mathrm{int}}~,
\]
\vspace{0.1pc}

between the $k$-th homotopy group $\pi_k(X_{\mathbb{F}_p})$ of its $p$-completion and representative elements of degree $k$ in $\mathcal{L}_*(X)$ up to interval equivalences, for all $k \geq 1$.
\end{theorem}

\begin{proof}
Follows directly from Theorems \ref{thm: fake Berglund theorem} and \ref{thm: modèles d'homotopie rationnel type fini}. 
\end{proof}

The above theorem has interesting consequences even in the basic case of the spheres where $X=S^m$. In this case, the absolute partition $\mathcal{L}_\infty$-algebra $\mathcal{L}_*(S^m)$ can be fully described: its underlying graded module admits a basis in terms of formal power series of symmetric rooted trees with leaves labelled by the element $a_{[m]}$ where $\tilde{C}_*^c(S^m) \cong \overline{\mathbb{F}}_p.a_{[m]}$, as explained in Lemma \ref{Lemma: basis elements}. The differential on symmetric rooted trees is given by the terms described after the aforementioned lemma, and the image by the differential on the labelling element $a_{[m]}$ is determined by the formula in Lemma \ref{lemma: coalgebra structure on the reduced spheres}. Computing representative elements in $\mathcal{L}_*(S^m)$ amounts to find formal power series of a given degree
\[
\varepsilon = \sum_{\delta \geq 0} \sum_{\tau \in \mathrm{SRT}(\delta)} \lambda_\tau \tau \left(a_{[m]}, \cdots, a_{[m]}\right)~,
\]
which satisfy the following equation 
\[
\displaystyle d_\mathfrak{g}(\varepsilon) + \sum_{n \geq 2} \sum_{\substack{w \in \E(n)_{k(n-1)}                             \\ \overline{w}_1 = \mathrm{id}_{\Sym_n}, \, \overline{w}_j \in \Sym_n}} (-1)^{\left[\frac{k(k-1)}{2}\right] \frac{(n+2)(n-1)}{2}} \prod_{j = 2}^k \mathrm{sign}(\overline{w}_j) ~ l^w_n(\varepsilon,\cdots,\varepsilon)= 0~, 
\]
where the action of the operations $l_n^{w}$ on symmetric rooted trees is given by grafting. This equation can be decomposed along the degree and the arity of the symmetric rooted trees that it involves, thus it can be checked inductively. Representative elements should be considered up to interval equivalences, but then again, these equivalences are given by purely combinatorial formulas (Definition \ref{def: interval equivalences}) as well, and they can be constructed weight by weight. In principle, both of this questions should be implementable in a computer. However, this is beyond the scope of the present paper and shall be the subject of future research. 

\begin{Example}
The generator $a_{[m]}$ in $\mathcal{L}_*(S^m)$ is a representative element in degree $m$, which corresponds to the canonical map $S^m \longrightarrow S^m_{\mathbb{F}_p}$. Since $a_{[m]}$ is in weight zero, it can be checked that it can not be interval equivalent to $0$, thus $\pi_m(S^m_{\mathbb{F}_p}) \neq 0$. 
\end{Example}

\begin{Remark}
Of course, in the above discussion, one can replace $\mathcal{L}_*(X)$ and in particular $\mathcal{L}_*(S^m)$ by any other absolute partition $\mathcal{L}_\infty$-algebra model in the sense of Definition \ref{def: p-adic model}. Thus computations could greatly be simplified if smaller models were found.
\end{Remark}

\medskip

\textbf{An intrinsic model structure.} We define $\pi$-equivalences of absolute partition $\mathcal{L}_\infty$-algebras and show that they are weak equivalences in a model category structure. We then compare this new model structure with the previous model structure considered so far, which was transferred from $\mathcal{EE}_\infty$-coalgebras. 

\begin{Definition}[$\pi$-equivalence]
Let $f: \mathfrak{g} \longrightarrow \mathfrak{h}$ be a map of absolute partition $\mathcal{L}_\infty$-algebras. It is a $\pi$\textit{-equivalence} if the induced map 
\[
\pi_*(\mathcal{R}(f),0):\pi_*(\mathcal{R}(\mathfrak{g}),0) \qi \pi_*(\mathcal{R}(\mathfrak{h}),0)
\]
is an isomorphism of groups for all $* \geq 1$. 
\end{Definition}

\begin{Remark}
By Theorem \ref{thm: fake Berglund theorem}, a map $f: \mathfrak{g} \longrightarrow \mathfrak{h}$ is a $\pi$-equivalence if and only if it induces a bijection 
\[
f: \mathrm{rep}(\mathfrak{g}_*)/\sim_{\mathrm{int}} \qi \mathrm{rep}(\mathfrak{h}_*)/\sim_{\mathrm{int}}
\]
between the sets of representative elements up to interval equivalences in all degrees $* \geq 1$. 
\end{Remark}

\begin{theorem}\label{thm: local model structure with pi-equivalences}
There is a model category structure on the category of absolute partition $\mathcal{L}_\infty$-algebras determined by the following classes of maps
\begin{enumerate}

\medskip
\item the class of weak equivalences is given by $\pi$-equivalences;

\medskip
\item the class of fibrations is given by maps $f$ such that $\mathcal{R}_*(f)$ is a Kan fibration;

\medskip
\item the class of cofibrations is determined by the left-lifting property against acyclic fibrations.

\medskip
\end{enumerate}
The adjunction $\mathcal{L}_* \dashv \mathcal{R}_*$ is also a Quillen adjunction with respect to this model structure. 
\end{theorem}

\begin{proof}
Let us consider the composite adjunction
\[
\begin{tikzcd}[column sep=5pc,row sep=3pc]
         \mathsf{sSet}_0 \arrow[r, shift left=1.1ex, "\mathrm{inc}"{name=F}] 
          &\mathsf{sSet}_* \arrow[l, shift left=.75ex, "\mathrm{cn}"{name=U}] \arrow[r, shift left=1.1ex, "\mathcal{L}_*"{name=A}]
          &\mathsf{abs}~\mathcal{L}_\infty^\pi\text{-}\mathsf{alg}^{\mathsf{qp}\text{-}\mathsf{comp}} \arrow[l, shift left=.75ex, "\mathcal{R}_*"{name=B}]
            \arrow[phantom, from=F, to=U, , "\dashv" rotate=-90] \arrow[phantom, from=A, to=B, , "\dashv" rotate=-90]
\end{tikzcd}
\]

where $\mathsf{sSet}_0$ denotes the category of reduced simplicial sets together with their model structure constructed in \cite[Chapter V, Proposition 6.2]{GoerssJardine}, and the Quillen adjunction $\mathrm{inc} \dashv \mathrm{cn}$ is given by the inclusion of reduced simplicial sets into pointed simplicial sets on the one hand, and by the contracting the connected component of the base point on the other hand. The category $\mathsf{sSet}_0$ together with this model structure presents the $\infty$-category of pointed connected spaces. The main idea is to use the "old" model structure (transferred from $\mathcal{EE}_\infty$-coalgebras) on qp-complete absolute $\mathcal{L}_\infty^\pi$-algebras to prove that we can use the right transfer theorem for model structures along the above composite adjunction. 

\medskip

The class of maps which are sent to weak equivalences of reduced simplicial sets by $\mathcal{R}_*$ is given by $\pi$-equivalences, it includes all previous weak equivalences transferred from $\mathcal{EE}_\infty$-coalgebras). The class maps $f$ such that $\mathcal{R}_*(f)$ is a Kan fibration contains all degree-wise surjections by Theorem \ref{thm: propriétés de l'intégration}. Therefore every qp-complete absolute $\mathcal{L}_\infty^\pi$-algebra $\mathfrak{g}$ is fibrant with respect to this new class of fibrations and, moreover the diagonal map $\mathfrak{g} \longrightarrow \mathfrak{g} \oplus \mathfrak{g}$ can be factored into a weak equivalence followed by a fibration, simply by considering the factorization of this map with respect to the old model structure. Thus the hypothesis of the right transfer theorem for model structures are satisfied. 
\end{proof}

\begin{Remark}
One can also apply the same arguments in order to transfer the model category structure on pointed simplicial sets $\mathsf{sSet}_*$ to qp-complete absolute $\mathcal{L}_\infty^\pi$-algebras along the adjunction $\mathcal{L}_* \dashv \mathcal{R}_*$. In this case, fibrations can be characterized as degree-wise surjections in degrees $\geq 0$. However, weak equivalences are now given by maps which induce weak homotopy equivalences between all the connected components of the Maurer--Cartan space, and not only on the connected component of the zero Maurer--Cartan element. By computing the (unreduced) $\mathcal{E}$-coalgebra structure of $C_*^c(S^k)$, it is possible to give an algebraic characterization of these equivalences similar to Theorem \ref{thm: fake Berglund theorem}. This model structure with "global equivalences" is analogue to the ones considered in \cite[Section 8.1]{Liemodels} and in \cite[Section 6.4]{robertnicoud2020higher}.
\end{Remark}

\begin{Proposition}\label{Prop: local model structure is an infinity sub cat}
The Quillen adjunction given by the identity functor
\[
\begin{tikzcd}[column sep=5pc,row sep=3pc]
         \mathcal{EE}_\infty\text{-}\mathsf{coalg}~[\mathrm{Q.iso}^{-1}] \simeq \mathsf{abs}~\mathcal{L}_\infty^\pi\text{-}\mathsf{alg}^{\mathsf{qp}\text{-}\mathsf{comp}}~[\mathrm{W}^{-1}]  \arrow[r, shift left=1.1ex, "\mathrm{Id}"{name=A}]
          &\mathsf{abs}~\mathcal{L}_\infty^\pi\text{-}\mathsf{alg}^{\mathsf{qp}\text{-}\mathsf{comp}}~[\pi\text{-}\mathrm{equi}^{-1}]  \arrow[l, shift left=.75ex, "\mathbb{L}\mathrm{Id}"{name=B}]
            \arrow[phantom, from=A, to=B, , "\dashv" rotate=90]
\end{tikzcd}
\]
exhibits the $\infty$-category of qp-complete absolute $\mathcal{L}_\infty^\pi$-algebras up to $\pi$-equivalences as a coreflective sub-$\infty$-category of $\mathcal{EE}_\infty$-coalgebras up to quasi-isomorphism.
\end{Proposition}

\begin{proof}
Clearly, the identity functor forms a Quillen adjunction since it preserves all fibrations and weak equivalences. It is straightforward to check that the derived unit of this adjunction is an isomorphism, so in particular a $\pi$-equivalence.
\end{proof}

\begin{Remark}
When $\kk$ is a field of characteristic zero, the notion of $\pi$-equivalence corresponds to the notion of a quasi-isomorphisms in degrees $\geq 1$, see Remark \ref{Rmk: Berglund fails + formality of cochains} and combine it with Proposition \ref{prop: comparison characteristic zero}. So, in particular, absolute partition $\mathcal{L}_\infty$-algebras up to quasi-isomorphism in degrees $\geq 1$ do present an $\infty$-category of an algebraic nature, that is, of algebras over a monad on the $\infty$-category of chain complexes. It would be very interesting to determine if this remains true for the $\infty$-category of absolute partition $\mathcal{L}_\infty$-algebras up to $\pi$-equivalences over a field of positive characteristic. 
\end{Remark}

\subsection{Mapping spaces}\label{Subsection: mapping spaces}
In this subsection, we construct explicit $p$-adic models for mapping spaces, without any assumption on the source simplicial set. Recall that, if $X$ and $Y$ are simplicial sets, there is a model for their mapping space given by 

\[
\mathrm{Map}(X,Y)_\bullet \coloneqq \mathrm{Hom}_{\mathsf{sSet}}(X \times \Delta^\bullet, Y)~,
\]
\vspace{0.25pc}

which forms a Kan complex when $Y$ is so.

\begin{lemma}\label{prop: les chaines sont lax monoidales en infini morphismes}
Let $X$ and $Y$ be two simplicial sets. The Eilenberg-Zilber map can be extended into an $\infty$-quasi-isomorphism 
\[
\mathrm{EZ}: C^c_*(X) \otimes C_*^c(Y) \rightsquigarrow C^c_*(X \times Y) ~,
\]
of $u\mathcal{EE}_\infty$-coalgebras, which furthermore is natural in $X$ and $Y$.
\end{lemma}

\begin{proof}
The Eilenberg-Zilber map does not commute with the $\mathcal{E}$-coalgebra structures, so it does not define a strict morphism. However, it can be extended into an $\infty$-quasi-isomorphism by \cite[Section 5]{fresse2003deriveddivisionfunctorsmapping} in a functorial way. Concretely, by Proposition 5.4, we get a morphism of $\mathrm{B}^{\mathrm{s.a}} \mathcal{E}$-coalgebras between $\mathrm{B}_\pi(C^*_c (X \times Y)) \longrightarrow \mathrm{B}_\pi(C^*_c(X) \otimes C^*_c(Y))$ which extends the Eilenberg-Zilber map, so it defines an $\infty$-quasi-isomorphism $C^*_c (X \times Y) \rightsquigarrow C^*_c(X) \otimes C^*_c(Y)$ of the $\mathcal{E}$-algebras between cochains on the product and the tensor product of the cochains. Thus, by passing to the linear dual, we get an $\infty$-quasi-isomorphism of $\mathcal{E}$-coalgebras $C^c_*(X) \otimes C_*^c(Y) \rightsquigarrow C^c_*(X \times Y)$ when $X$ and $Y$ are of finite type that we can extend it to all spaces $X,Y$ by colimits like in the proof of \cite[Lemma 3.27]{lucio2022integration}.
\end{proof}

\begin{theorem}\label{thm: vrai thm mapping spaces}
Let $\mathfrak{g}$ be a qp-complete curved absolute $\mathcal{L}^\pi_\infty$-algebra and let $X$ be a simplicial set. There is a weak equivalence of Kan complexes
\[
\mathrm{Map}(X, \mathcal{R}(\mathfrak{g})) \simeq \mathcal{R}\left(\mathrm{hom}(C^c_*(X),\mathfrak{g})\right)~,
\]

which is natural in $X$ and in $\mathfrak{g}$, where $\mathrm{hom}(C^c_*(X),\mathfrak{g})$ denotes the convolution curved absolute $\mathcal{L}^\pi_\infty$-algebra. Furthermore, it is possible to replace the cellular chains $C^c_*(X)$ by the homology $\mathrm{H}_*(X)$ to obtain a smaller model, meaning that there is a weak equivalence of Kan complexes
\[
\mathrm{Map}(X, \mathcal{R}(\mathfrak{g})) \simeq \mathcal{R}\left(\mathrm{hom}(\mathrm{H}_*(X),\mathfrak{g})\right)~,
\]

which is now only natural in $\mathfrak{g}$. 
\end{theorem}

\begin{proof}
There is an isomorphism of simplicial sets
\[
\mathrm{Map}(X,\mathcal{R}(\mathfrak{g}))_\bullet \coloneqq \mathrm{Hom}_{\mathsf{sSet}}(X \times \Delta^\bullet, \mathcal{R}(\mathfrak{g})) \cong \mathrm{Hom}_{\mathsf{sSet}}( C^c_*(X \times \Delta^\bullet), \widehat{\mathrm{B}}_\iota(\mathfrak{g}))~.
\]
We can pre-compose by the $\infty$-quasi-isomorphism extending the Eilenberg-Zilber map of Lemma \ref{prop: les chaines sont lax monoidales en infini morphismes}, giving a weak equivalence of simplicial sets
\[
\mathrm{Hom}_{\Omega \mathrm{B}^{\mathrm{s.a}} \mathcal{E}\text{-}\mathsf{cog}}(C^c_*(X \times \Delta^\bullet), \widehat{\mathrm{B}}_\iota(\mathfrak{g})) \qi \mathrm{Hom}_{\Omega \mathrm{B}^{\mathrm{s.a}} \mathcal{E}\text{-}\mathsf{cog}}( C^c_*(X)  \otimes C^c_*(\Delta^\bullet), \widehat{\mathrm{B}}_\iota(\mathfrak{g}))~,
\]

since both $C^c_*(X \times \Delta^\bullet)$ and $C^c_*(X)  \otimes C^c_*(\Delta^\bullet)$ are Reedy cofibrant. Let's compute this last simplicial set:
\begin{align*}
\mathrm{Hom}_{\Omega \mathrm{B}^{\mathrm{s.a}} \mathcal{E}\text{-}\mathsf{cog}}( C^c_*(X)  \otimes C^c_*(\Delta^\bullet), \widehat{\mathrm{B}}_\iota(\mathfrak{g})) 
&\cong \mathrm{Hom}_{\Omega \mathrm{B}^{\mathrm{s.a}} \mathcal{E}\text{-}\mathsf{cog}} \left( C^c_*(\Delta^\bullet), \left\{ C^c_*(X), \widehat{\mathrm{B}}_\iota(\mathfrak{g}) \right\} \right) \\
&\cong \mathrm{Hom}_{\Omega \mathrm{B}^{\mathrm{s.a}} \mathcal{E}\text{-}\mathsf{cog}} \left( C^c_*(\Delta^\bullet), \widehat{\mathrm{B}}_\iota \left(\mathrm{hom}( C^c_*(X), \mathfrak{g}) \right) \right) \\
& \cong \mathrm{Hom}_{\mathsf{curv}~\mathsf{abs}~\mathcal{L}^\pi_\infty\text{-}\mathsf{alg}} \left( \widehat{\Omega}_\iota(C^c_*(\Delta^\bullet)), \mathrm{hom}( C^c_*(X), \mathfrak{g}) \right) \\
& \cong \mathcal{R}\left(\mathrm{hom}( C^c_*(X), \mathfrak{g}) \right)~.
\end{align*}
Upon choosing a contraction between $C^c_*(X)$ and its homology $\mathrm{H}_*(X)$, we can use the homotopy transfer theorem to obtain a $u\mathcal{CC}_\infty$-coalgebra structure on $\mathrm{H}_*(X)$ which is weakly equivalent to $C^c_*(X)$. This implies that $\mathrm{H}_*(X)  \otimes C^c_*(\Delta^\bullet)$ and $C^c_*(X)  \otimes C^c_*(\Delta^\bullet)$ are weakly equivalent as Reedy cofibrant objects, and thus we can replace the later by the former in the above computation.
\end{proof}

\begin{Corollary}\label{Cor: True mapping spaces.}
Let $X$ be a simplicial set and let $Y$ be a connected finite type nilpotent simplicial set. There is a weak equivalence of simplicial sets 
\[
\mathrm{Map}(X, Y_{\mathbb{F}_p}) \simeq \mathcal{R}\left(\mathrm{hom}(\mathrm{H}_*(X), \mathcal{L}(Y))\right)~,
\]
where $Y_{\mathbb{F}_p}$ denotes the Bousfield-Kan p-completion of $Y$. 
\end{Corollary}

\begin{proof}
There is a weak equivalence of simplicial sets 
\[
Y_{\mathbb{F}_p} \simeq \mathcal{RL}(Y)~,
\]
which induces the following weak equivalences
\[
\mathrm{Map}(X, Y_{\mathbb{F}_p}) \simeq \mathrm{Map}(X, \mathcal{RL}(Y)) \simeq \mathcal{R}\left(\mathrm{hom}(\mathrm{H}_*(X), \mathcal{L}(Y))\right)~.
\]
\end{proof}

\textbf{Relationship with previous work on mapping spaces.} Let us give a brief overview of the relationship between Theorem \ref{thm: vrai thm mapping spaces} and existing work about $p$-adic mapping spaces, in particular around the Sullivan conjecture \cite{Miller,Carlsson,LannesSchwartz,LannesT}. There are many different statement in the literature which can be called the Sullivan conjecture, however in general they revolve around proving that the pointed mapping space 
\[
\mathrm{Map}_*(\mathrm{B}V, Y_{\mathbb{F}_p}) \simeq \{*\}
\]
is contractible under some finiteness assumptions on $Y$, where $\mathrm{B}V$ is the classifying space of a finite dimensional $\mathbb{F}_p$-vector space $V$. 

\medskip

Previous work on these questions was done at the level of so called power operations. The homology functor 
\[
\mathrm{H}_*: u\mathcal{EE}_\infty\text{-}\mathsf{cog} \longrightarrow \mathsf{gr}~\kk\text{-}\mathsf{mod}
\]
from $u\mathcal{EE}_\infty$-coalgebras to graded $\kk$-vector spaces factors through the category of graded counital coalgebras equipped with Dyer--Lashof operations, which are called power operations. Moreover, on the homology of $u\mathcal{EE}_\infty$-coalgebras of the form $C_*^c(X)$, where $X$ is a space, Dyer--Lashof operations give back Steenrod operations, and $\mathrm{H}_*(X)$ is an unstable counital coalgebra over the Steenrod algebra. Dually, the cohomology of a $u\mathcal{EE}_\infty$-algebra is also endowed with Dyer--Lashof operations, which specify to Steenrod operations in the case of the cohomology of a space $\mathrm{H}^*(X)$. See \cite{Lawson} for a modern point of view on power operations.

\medskip

In \cite{LannesT}, Lannes considers the tensor product of both unstable algebras and unstable coalgebras over the Steenrod algebra. He shows that the tensor product of unstable coalgebras admits a right adjoint, given by a self-enrichment, and that the tensor product with a degree-wise finite dimensional unstable algebra admits a left adjoint. These two constructions can be compared under several finiteness assumptions, see \cite[Section 1.13]{LannesT}. See also the textbook account of  \cite{SchwartzBook}. 

\medskip

Lannes' well-known $T_V$ functor is defined as the left adjoint to the tensor product $\mathrm{H}^*(\mathrm{B}V) \otimes -$ in unstable modules over the Steenrod algebra, and later it is proved that it coincides with the left adjoint to this same tensor product in unstable algebras over the Steenrod algebras. Lannes works with the functor $T_V$ as it is better behaved in general, nevertheless, one of the key points of his work is to be able to compare it to the self-enrichment $\mathbf{hom}(\mathrm{H}_*(\mathrm{B}V),-)$ in unstable coalgebras. In particular, when one takes a cosimplicial resolution on the target. 

\medskip

The reason is that taking a cosimplicial resolution on the target unstable coalgebra is what originally allowed Bousfield and Kan to define the unstable Adam's spectral sequence in \cite{BousfieldKanUnstable}, which converges to the homotopy groups of $p$-adic mapping spaces. Analysing this spectral sequence was the key ingredient in Miller's original proof in \cite{Miller}. 

\medskip

The mapping coalgebra construction of Proposition \ref{Prop: mapping coalgebras} is a lift of the self-enrichment of unstable coalgebras constructed by Lannes in \cite{LannesT} from the level of power operations on homology to the level of $u\mathcal{EE}_\infty$-coalgebras. A functorial fibrant resolution at the $u\mathcal{EE}_\infty$-coalgebra level is given by the complete bar-cobar construction; it can be considered a lift of the cosimplicial resolution at the unstable coalgebra level and it is also related to the André--Quillen cohomology, in this case, of $u\mathcal{EE}_\infty$-coalgebras, see Remark \ref{Rmk: underlying infinit cats}. Let $C,D$ be two $u\mathcal{EE}_\infty$-coalgebras, by Theorem \ref{thm: compatibilité des mappings et des convolutions} there is a natural isomorphism 
\[
\left\{C, \widehat{\mathrm{B}}_\iota \widehat{\Omega}_\iota D \right\} \cong \widehat{\mathrm{B}}_\iota\mathrm{hom}\left(C,\widehat{\Omega}_\iota D \right)~, 
\]
where on the left hand side we consider the mapping coalgebra construction and on the right hand side the complete bar construction on the convolution absolute partition $\mathcal{L}_\infty$-algebra. The left hand side is a lift of the self-enrichment of unstable coalgebras where one takes a cosimplicial resolution of the target. So Theorem \ref{thm: vrai thm mapping spaces} can be interpreted as an algebraic model of the derived primitives of this lift, in the sense of Remark \ref{Rmk: underlying infinit cats}. In particular, considering the case $C = \mathrm{H}_*(\mathrm{B}V)$ recovers Lannes' $T_V$ when passing to homology and taking the linear dual under finiteness assumptions. And applying Theorem \ref{thm: fake Berglund theorem} to the models in \ref{thm: vrai thm mapping spaces} gives a description of the homotopy groups of a general $p$-adic mapping space, and in this sense, it can be considered an "algebraic lift" of the unstable Adam's spectral sequence of \cite{BousfieldKanUnstable}.

\appendix
\section{Structural formulas}\label{Appendice formules}
The goal of this appendix is to give formulas for the relations that the elementary operations of a (curved) absolute $\mathcal{L}^\pi_\infty$-algebra satisfy. These relations are imposed by the partial decomposition maps of the counital partial cooperad $\mathcal{E}^*$; they are in general hard to understand. In other words, our main goal is to describe the linear dual of the maps in \cite[Section 1.1.3]{BergerFresse04}. 

\begin{lemma}
Let $\sigma$ be an element in $\mathbb{S}_{n+k-1}$, with $n, k \geq 1$. For any $1 \leq i \leq n$, there is at most one pair $(\tau,\nu)$ with $\tau \in \mathbb{S}_n$ and $\nu \in \mathbb{S}_k$ such that $\tau \circ_i \nu = \sigma~.$
\end{lemma}

\begin{proof}
The decomposition exists if and only if $\sigma$ stabilizes $\{1,\cdots, n + k -1\} - \{i,\cdots, i + k-1\}$. In that case, the permutation $\tau$ is determined by how $\sigma$ restricts to this subset, and therefore $\nu$ is also completely determined. 
\end{proof}

Let $\sigma$ be in $\mathbb{S}_{n+k-1}$. We say it is $(n,k,i)$-\textit{admissible} if it admits a decomposition $\Delta^{n,k}_i(\sigma)$, and we denote $\sigma^{(1)}$ and $\sigma^{(2)}$ the unique pair in $\mathbb{S}_n \times \mathbb{S}_k$ such that $\sigma^{(1)} \circ_i \sigma^{(2)} = \sigma$. 

\medskip

Let $(\sigma_0, \cdots, \sigma_r)$ be an $r$-tuple of elements in $\mathbb{S}_{n+k-1}$. Suppose that $\sigma_j$ is $(n,k,i)$-admissible for all $1 \leq j \leq r$, and denote $(\sigma_0^{(1)},\cdots, \sigma_r^{(1)})$ and $(\sigma_0^{(2)},\cdots, \sigma_r^{(2)})$ the $r$-tuples of decompositions. 

\medskip

Let $a,b$ be two natural numbers such that $a + b = r$ and let $\varphi: (0,0) \longrightarrow (a,b)$ be a lattice path. The $r$-tuple $(\sigma_0, \cdots, \sigma_r)$ is $(\varphi,i)$\textit{-admissible} if for any step $(j,k) \longrightarrow (j+1,k)$ in $\varphi$
\[
\sigma_j^{(2)} = \sigma_{j+1}^{(2)}~,
\]
and for any step $(j,k) \longrightarrow (j,k+1)$,
\[
\sigma_k^{(1)} = \sigma_{k+1}^{(1)}~.
\]
\vspace{0.1pc}

One can write $(\sigma_0^{(1)},\cdots, \sigma_a^{(1)})$ and $(\sigma_0^{(2)},\cdots, \sigma_b^{(2)})$ by suppressing the repetitions when $(\sigma_0, \cdots, \sigma_r)$ is $(\varphi,i)$-admissible for a given lattice path $\varphi: (0,0) \longrightarrow (a,b)$. 

\medskip

For any lattice path $\varphi: (0,0) \longrightarrow (a,b)$, the \textit{lattice sign} $\epsilon(\varphi)$ is the sign of the shuffle permutation which takes horizontal segments to the first place and vertical segments to the last place, like in \cite[Section 1.1.3]{BergerFresse04}. 

\begin{Proposition}\label{prop: formulas for the decompositions}
The partial decomposition map
\[
\Delta^{n,k}_i: \mathcal{E}^*(n+k-1) \longrightarrow \mathcal{E}^*(n) \otimes \mathcal{E}^*(k)
\]
is given by as follows: let $(\sigma_0, \cdots, \sigma_r)$ be an element in $\mathcal{E}^*(n+k-1)$. 

\medskip

\begin{enumerate}
\item If $\sigma_j$ is $(n,k,i)$-admissible for all $1 \leq j \leq r$, then 

\[
\Delta^{n,k}_i(\sigma_0, \cdots, \sigma_r) = \sum_{a + b = r} \sum_{\substack{\varphi: (0,0) \rightarrow (a,b), \\ (\varphi,i)\text{-}\mathrm{admissible}}} (-1)^{\epsilon(\varphi)}\left(\sigma_0^{(1)},\cdots, \sigma_a^{(1)}\right) \otimes \left(\sigma_0^{(2)},\cdots, \sigma_b^{(2)}\right)~,
\]

where the second sum is taken over all lattice paths $\varphi: (0,0) \longrightarrow (a,b)$ for which $(\sigma_0, \cdots, \sigma_r)$ is $(\varphi,i)$-admissible.

\medskip

\item If there exists a $\sigma_j$ which is not $(n,k,i)$-admissible, then the image of $(\sigma_0, \cdots, \sigma_r)$ by the decomposition $\Delta^{n,k}_i$ is zero. 
\end{enumerate}
\end{Proposition}

\begin{proof}
If one of the permutations $\sigma_j$ in $\mathbb{S}_{n+k-1}$ cannot be obtained as the $i$-th composition of a permutation in $\mathbb{S}_n$ and a permutation in $\mathbb{S}_k$, then it not reached by 
\[
\circ^{n,k}_i: \mathcal{E}(n) \otimes \mathcal{E}(k) \longrightarrow \mathcal{E}(n+k-1) 
\]
as defined in \cite[Section 1.1.3]{BergerFresse04}. Therefore its image by $\Delta^{n,k}_i = (\circ_i^{n,k})^*$ is zero. The other part of the formula follows from direct inspection.
\end{proof}

\begin{Corollary}
Let $\mathfrak{g}$ be a curved absolute partition $\mathcal{L}_\infty$-algebra. Its elementary operations $\{l_n^{(\sigma_0,\cdots,\sigma_r)}\}$ satisfy the following relations:
\begin{align*}
\partial \left(l_n^{(\sigma_0,\cdots,\sigma_r)} \right) = &\sum_{\substack{\bar{\sigma} \in \mathbb{S}_n\\ \bar{\sigma} \neq \sigma_e}} \sum_{j=0}^r (-1)^j l_n^{(\sigma_0, \cdots, \bar{\sigma}, \cdots, \sigma_r)} + \\
& +\sum_{\substack{p+q=n+1, \\ a+b =r}} \sum_{i = 0}^p \sum_{\substack{\varphi: (0,0) \rightarrow (a,b), \\ (\varphi,i)\text{-}\mathrm{admissible}}} (-1)^{\epsilon(\varphi)} l_p^{(\sigma_0^{(1)},\cdots, \sigma_a^{(1)})} \circ_i  l_q^{(\sigma_0^{(2)},\cdots, \sigma_b^{(2)})}~,
\end{align*}
where $\bar{\sigma}$ is added in the $j$-spot in $(\sigma_0, \cdots, \bar{\sigma}, \cdots, \sigma_r)$.
\end{Corollary}

\begin{proof}
Follows directly from Proposition \ref{prop: formulas for the decompositions}.
\end{proof}

\begin{Remark}
If $\mathfrak{g}$ is a partition $\mathcal{L}_\infty$-algebra, these are precisely the relations satisfied by the structural operations. In an absolute partition $\mathcal{L}_\infty$-algebra, these relations extend to formal power series of elementary operations. 
\end{Remark}

\begin{Remark}
Many of the formulas in this paper have been implemented in Python using the Python package \cite{pythonbarratteccles} written by Medina-Mardones that encodes elements in the Barratt-Eccles operad and their composition.
\end{Remark}

\bibliographystyle{alpha}
\bibliography{bibax}

\end{document}